\documentclass[12pt, reqno]{amsart}
\usepackage{inputenc}
\usepackage[T1]{fontenc}
\usepackage{lmodern}
\usepackage{amsmath}
\usepackage{amscd}
\usepackage{amssymb}
\usepackage{mathtools}
\usepackage{latexsym}
\usepackage{MnSymbol}
\usepackage{amsrefs}
\usepackage{nicefrac}
\usepackage{microtype}
\usepackage{color}

\usepackage{enumitem} 
\setlist[enumerate,1]{label=\textup{(\arabic*)}}

\usepackage[all]{xy}
\newdir{ >}{{}*!/-5pt/@{>}}

\usepackage[pdftitle={Nonarchimedean bornologies, cyclic homology and rigid cohomology},
pdfauthor={Guillermo Corti\~nas, Joachim Cuntz, Ralf Meyer, Georg Tamme},
pdfsubject={Mathematics}
]{hyperref}

\BibSpec{book}{%
  +{}  {\PrintPrimary}                {transition}
  +{,} { \textit}                     {title}
  +{.} { }                            {part}
  +{:} { \textit}                     {subtitle}
  +{,} { \PrintEdition}               {edition}
  +{}  { \PrintEditorsB}              {editor}
  +{,} { \PrintTranslatorsC}          {translator}
  +{,} { \PrintContributions}         {contribution}
  +{,} { }                            {series}
  +{,} { \voltext}                    {volume}
  +{,} { }                            {publisher}
  +{,} { }                            {organization}
  +{,} { }                            {address}
  +{,} { \PrintDateB}                 {date}
  +{,} { }                            {status}
  +{}  { \parenthesize}               {language}
  +{}  { \PrintTranslation}           {translation}
  +{;} { \PrintReprint}               {reprint}
  +{.} { }                            {note}
  +{.} {}                             {transition}
  +{} { \PrintDOI}                   {doi}
  +{} { available at \url}            {eprint}
  +{}  {\SentenceSpace \PrintReviews} {review}
}

\renewcommand*{\PrintDOI}[1]{\href{http://dx.doi.org/\detokenize{#1}}{doi: \detokenize{#1}}}

\newcommand{\comment}[1]{}  

\theoremstyle{plain}
\newtheorem{theorem}{Theorem}[section]
\newtheorem{lemma}[theorem]{Lemma}
\newtheorem{corollary}[theorem]{Corollary}
\newtheorem{proposition}[theorem]{Proposition}
\theoremstyle{remark}
\newtheorem{remark}[theorem]{Remark}
\theoremstyle{definition}
\newtheorem{definition}[theorem]{Definition}
\newtheorem{example}[theorem]{Example}
\numberwithin{theorem}{section}
\newcommand{\bgl}{\begin{equation}} 
\newcommand{\egl}{\end{equation}}
\newcommand{\btheo}{\begin{theorem}}
\newcommand{\etheo}{\end{theorem}}
\newcommand{\blemma}{\begin{lemma}}
\newcommand{\elemma}{\end{lemma}}
\newcommand{\bproof}{\begin{proof}}
\newcommand{\eproof}{\end{proof}}
\newcommand{\bbew}{\begin{beweis}}
\newcommand{\ebew}{\end{beweis}}
\newcommand{\bremark}{\begin{remark}}
\newcommand{\eremark}{\end{remark}}
\newcommand{\bex}{\begin{example}\em}
\newcommand{\eex}{\end{example}}
\newcommand{\bdefin}{\begin{definition}}
\newcommand{\edefin}{\end{definition}}
\newcommand{\bprop}{\begin{proposition}}
\newcommand{\eprop}{\end{proposition}}
\newcommand{\bcor}{\begin{corollary}}
\newcommand{\ecor}{\end{corollary}}
\newcommand{\bfa}{\begin{cases}} 
\newcommand{\efa}{\end{cases}}

\newcommand{\cC}{\mathcal C}

\newcommand{\cF}{\mathcal F}
\newcommand{\cI}{\mathcal I}

\newcommand{\cO}{\mathcal O}

\newcommand\C{\mathbb C}
\newcommand\N{\mathbb N}
\newcommand\Q{\mathbb Q}
\newcommand\R{\mathbb R}
\newcommand\Z{\mathbb Z}

\newcommand{\bdd}{\mathcal B}
\newcommand{\coma}{\widehat}
\newcommand{\comb}{\overbracket[.7pt][1.4pt]}

\newcommand*{\ling}[1]{#1_\mathrm{lg}}
\newcommand*{\comling}[1]{\comb{#1_\mathrm{lg}}}

\newcommand*{\gral}[3]{\ul{#1}{}_{#2,#3}}
\newcommand*{\comg}[3]{\comb{\ul{#1}_{#2,#3}}}
\newcommand*{\comJ}[2]{\comb{\ul{#1}_{#2}}}
\newcommand*{\tub}[3]{\mathcal{T}_{#3}(#1,#2)}

\newcommand{\updagger}{\textup{\tiny\!\!\dagger}}

\newcommand{\Bor}{\fB\mathrm{or}}
\newcommand{\borno}{\comb{\Bor}(\dvr)}
\newcommand{\fB}{\mathfrak B}

\newcommand{\Vect}{\mathrm{Vect}}
\newcommand*{\norm}[1]{\left\|#1\right\|} 
\newcommand{\defeq}{\mathrel{:=}} 

\newcommand{\triqui}{\vartriangleleft}
\newcommand{\rig}{\mathrm{rig}}
\newcommand{\Spec}{\operatorname{Spec}}

\newcommand\hotimes{\mathbin{\comb{\otimes}}}
\newcommand\haotimes{\mathbin{\coma{\otimes}}}

\DeclareMathOperator{\coker}{coker}
\DeclareMathOperator{\Tor}{Tor}

\DeclareMathOperator{\ind}{ind-}
\DeclareMathOperator*{\holim}{holim}

\newcommand{\Hy}{\mathbb{H}}

\newcommand{\cS}{\mathcal{S}}
\newcommand{\Sp}{\mathrm{Sp}}
\newcommand{\fm}{\mathfrak{m}}
\newcommand{\spec}{\mathrm{sp}}

\newcommand{\op}{\mathrm{op}}
\newcommand{\ev}{\mathrm{ev}}
\newcommand{\even}{\mathrm{ev}}
\newcommand{\odd}{\mathrm{odd}}
\newcommand{\id}{\mathrm{id}}
\newcommand{\nb}{\nobreakdash}

\newcommand{\dvr}{V}
\newcommand{\dvgen}{\pi}
\newcommand{\dvf}{K}
\newcommand{\resf}{k}
\newcommand{\born}{\mathcal B}

\newcommand*{\abs}[1]{\lvert #1\rvert}
\newcommand*{\pepsilon}{\epsilon}

\DeclareMathOperator{\Hom}{Hom}
\DeclareMathOperator{\HP}{HP}
\DeclareMathOperator{\HH}{HH}
\newcommand{\ul}{\underline}

\newcommand{\chain}{\mathsf{cc}}
\newcommand{\homo}{\mathsf{hc}}
\newcommand{\chaindR}{\mathsf{cdR}}
\newcommand{\homdR}{\mathsf{hdR}}

\begin{document}

\title[Bornologies, cyclic homology and rigid cohomology]{Nonarchimedean bornologies, cyclic homology and rigid cohomology}
\author{Guillermo Corti\~nas}
\address{Dep. Matem\'atica-IMAS, FCEyN-UBA\\ Ciudad Universitaria Pab 1\\
1428 Buenos Aires\\ Argentina}
\email{gcorti@dm.uba.ar}\urladdr{http://mate.dm.uba.ar/\~{}gcorti}

\author{Joachim Cuntz}
\address{Mathematisches Institut\\
  Westf\"alische Wilhelms-Universit\"at M\"unster\\
  Ein\-stein\-str.\ 62\\
  48149 M\"unster\\
  Germany}
\email{cuntz@math.uni-muenster.de}

\author{Ralf Meyer}
\address{Mathematisches Institut\\
  Georg-August Universit\"at G\"ottingen\\
  Bun\-sen\-stra\-\ss{}e 3--5\\
  37073 G\"ottingen\\
  Germany}
\email{rmeyer2@uni-goettingen.de}

\author{Georg Tamme}
\address{Universit\"at Regensburg\\
  Fakult\"at f\"ur Mathematik\\
  93040 Regensburg\\
  Germany}
\email{georg.tamme@ur.de}


\begin{abstract}
  Let~\(\dvr\) be a complete discrete valuation ring with residue
  field~\(\resf\) and with fraction field~\(\dvf\) of characteristic~\(0\).
  We clarify the analysis behind the Monsky--Washnitzer completion
  of a commutative \(\dvr\)\nb-algebra using spectral radius
  estimates for bounded subsets in complete bornological
  \(\dvr\)\nb-algebras.  This leads us to a functorial chain complex
  for commutative \(\resf\)\nb-algebras that computes Berthelot's
  rigid cohomology.  This chain complex is related to the
  periodic cyclic homology of certain complete bornological
  \(\dvr\)\nb-algebras.
\end{abstract}
\thanks{The first named author was supported by Conicet
and partially supported by grants UBACyT 20021030100481BA, PIP
112-201101-00800CO, PICT 2013-0454, and MTM2015-65764-C3-1-P (Feder funds).\\ The second named
  author was supported by DFG through CRC 878 and by the ERC through
  AdG 267\,079. \\
  The fourth named author was supported by DFG through CRC 1085.}

\maketitle
\section{Introduction}

The problem of defining a cohomology theory with good properties for
an algebraic variety over a field~\(\resf\)
of non-zero characteristic has a long history.  In the breakthrough
paper~\cite{mw} by Monsky and Washnitzer, such a theory for smooth
affine varieties was constructed as follows.  Take a complete discrete
valuation ring~\(\dvr\) of mixed characteristic
with uniformizer~\(\dvgen\)
and residue field \(\resf=\dvr/\dvgen\dvr\)
(for example, \(\dvr\)
the ring of Witt vectors~\(W(\resf)\)
if~\(\resf\)
is perfect).  Let~\(\dvf\)
be the fraction field of~\(\dvr\).
Choose a \(\dvr\)\nb-algebra~\(R\)
which is a lift mod~\(\dvgen\)
of the coordinate ring of the variety and which is smooth
over~\(\dvr\)
(such a lift exists by~\cite{mme}).  Monsky--Washnitzer then introduce
the `weak' or dagger-completion~\(R^\updagger\)
of~\(R\)
and define their cohomology as the de~Rham cohomology of
\(R^\updagger\otimes_\dvr\dvf\).
The construction of a weak completion has become a basis for the
definition of cohomology theories in this context ever since.  The
Monsky--Washnitzer theory has been generalized by
Berthelot~\cite{berth} to ``rigid cohomology,'' which is a
satisfactory cohomology theory for general varieties over~\(\resf\).
Its definition uses certain de~Rham complexes on rigid analytic spaces.

The definition of the Monsky--Washnitzer cohomology  depends on the
choice of certain smooth algebras over~\(\dvr\). Also Berthelot's definition of
rigid cohomology depends on choices.
So the chain complexes that compute them are not yet functorial for
algebra homomorphisms.  Only their homology is functorial.
Functorial complexes that compute
rigid cohomology have been constructed by Besser \cite{Besser}. However,
the construction is based on some abstract existence statements, and is not at all
explicit.

One aim of our article is the construction of a natural and explicit
chain complex that computes rigid cohomology. A second aim is
linking rigid cohomology to cyclic homology. For cyclic homology in
characteristic 0 it was recognized long ago that analytic versions
of the theory can be treated in an elegant way using bornological
structures on the underlying algebras, see \cites{ConEntire,
  Meyer:HLHA}.
On the other hand, the relevance of a bornological point of view in the context of dagger completions has
been highlighted recently by Bambozzi~\cite{Bambozzi:Affinoid}. For our purposes here, we again find that it is natural to work in a framework based on bornological structures. This allows to generalize the weak completions of Monsky--Washnitzer to bornological versions of $J$-adic completions for an ideal $J$ in a commutative $\dvr$-algebra. The development of the corresponding techniques is our third project.

To prepare the ground for our theory, in Section~\ref{sec:bornologies} we first recall some basics on
bornological \(\dvr\)\nb-modules
and \(\dvf\)\nb-vector
spaces.  We are particularly
interested in completions and completed tensor products, which play a
crucial role in our theory.  We also relate bornological
\(\dvr\)\nb-modules
to inductive systems of \(\dvr\)\nb-modules,
carrying over well known results for bornological vector spaces over
\(\R\)
and~\(\C\).
Section~\ref{sec:borno} contains our bornological interpretation of
the weak completions used for Monsky--Washnitzer cohomology and rigid
cohomology.  The main point here is the spectral radius~\(\varrho(S)\)
of a bounded subset~\(S\)
in a complete bornological algebra, which is concerned with the growth
rate of the powers~\(S^n\),
\(n\in\N\).
This is a non-negative \emph{real} number or~\(\infty\).
The Monsky--Washnitzer completion of a finitely generated, commutative
\(\dvr\)\nb-algebra~\(R\)
is the smallest completion of~\(R\)
in which all  finitely generated~\(\dvr\)\nb-submodules $S$ of~\(R\)
have \(\varrho(S)\le 1\).
This completion makes sense also for an infinitely generated
for noncommutative
algebra~\(R\)
and is denoted by~\(\comling{R}\)
in this generality.  Similarly, \(\comg{R}{J}{\alpha}\)
for an ideal \(J\triqui R\)
with \(\dvgen\in J\)
and \(\alpha\in [0,1]\)
is the smallest completion of~\(\ul{R}\defeq R\otimes \dvf\)
in which all finitely generated~\(\dvr\)\nb-submodules~\(S\)
of~\(J\)
have \(\varrho(S) \le\pepsilon^\alpha\) where $\pepsilon=\abs{\dvgen}$ is the absolute value of the uniformizer of $\dvr$.
Section~\ref{sec:borno} first describes these completions more
explicitly and proves some basic properties.  Then it relates them to
Monsky--Washnitzer completions of finitely generated subalgebras
of~\(R\) and certain generalized tube algebras for \(J\triqui R\).

The first step then, in our natural construction of a complex computing rigid cohomology, is to present a
commutative \(\resf\)\nb-algebra~\(A\)
by the free commutative \(\dvr\)\nb-algebra
\(R\defeq \dvr[A]\)
generated by the set~\(A\).
This comes with a canonical surjective \(\dvr\)\nb-algebra
homomorphism \(p\colon R\to A\).
Let \(J\triqui R\)
be its kernel.  Since the algebra~\(R\)
is in general neither finitely generated nor Noetherian,
many results of Monsky and Washnitzer do not apply to it.
But our bornological
approach also works fine for such infinitely generated algebras.  We
define a family of weak completions \(\comg{R}{J}{\alpha}\)
of \(\ul{R}\defeq R\otimes \dvf\)
that depend on the ideal \(J\)
and \(\alpha\in [0,1]\).  Set $\N\defeq\Z_{\ge 0}$.
The de~Rham complexes of these weak completions for
\(\alpha=\nicefrac1m\),
\(m\in\N_{\ge1}\),
form a projective system.  Its homotopy projective limit is a chain
complex that is naturally associated to~\(A\).
We show that it computes the rigid cohomology if~\(A\)
is of finite type over \(\resf\)\nb.
The first step to see this is a homotopy invariance result: the
completions \(\comg{R}{J}{\alpha}\)
and \(\comg{R'}{J'}{\alpha}\)
for another free commutative \(\dvr\)\nb-algebra~\(R'\)
with a surjection \(p'\colon R'\to A\)
and \(J'\defeq \ker(p')\)
are homotopy equivalent with `dagger continuous' homotopies. This kind of homotopy is defined
using a weak completion of the polynomial
algebra~\(\dvf[t]\). As a consequence of the homotopy,
the de~Rham chain complexes for \(\comg{R}{J}{\alpha}\)
and \(\comg{R'}{J'}{\alpha}\) are homotopy equivalent.
When~\(R'\)
is finitely generated, we then show in Section~\ref{sec:compare_rigid} that the homotopy limit of the de Rham
complexes for the system \(\left(\comg{R'}{J'}{1/m}\right)\)
computes rigid cohomology using results of Gro{\ss}e-Kl\"onne
in~\cite{gkdr}.

To establish the link to cyclic homology, we compute in Section~\ref{sec:HH_dagger} the Hochschild and periodic
cyclic homology of the completions \(\comling{R}\)
and \(\comg{R}{J}{\alpha}\).
This is based on flatness results for Monsky--Washnitzer completions
of torsion-free, finitely generated, commutative \(\dvr\)\nb-algebras,
which allow to compute their Hochschild homology.  We obtain an analogue of Connes' computation, in \cite{connes}, of the periodic cyclic homology for
algebras of smooth functions on manifolds in this setting, showing
that the periodic cyclic homology of \(\comg{R}{J}{\alpha}\)
is naturally isomorphic to the cohomology of the de~Rham complex for
\(\comg{R}{J}{\alpha}\) made periodic. This result is formally very similar to the theorem of Feigin--Tsygan in~\cite{ft}*{Theorem 5}, \cite{ft2}*{Theorem 6.1.1},  which relates cyclic homology in characteristic 0 to Grothendieck's infinitesimal cohomology (for a different proof, covering also the non-affine case, see \cite{coco}*{Theorem 6.7}). Our proof uses the fact that different flat resolutions give quasi-isomorphic complexes and in fact, we find that this method also gives a very short new proof of the Feigin--Tsygan Theorem. We further prove in this section that periodic cyclic
homology for complete bornological \(\dvf\)\nb-algebras
is invariant under dagger-continuous homotopies.

On the basis of the results in Section~\ref{sec:HH_dagger} we obtain in Section~\ref{sec:homology_residue} a second chain complex
that models rigid cohomology made periodic.  Namely, we take the cyclic bicomplexes of
\(\comg{R}{J}{\alpha}\)
for \(\alpha=\nicefrac1m\),
which compute the periodic cyclic homology of these algebras. In Section~\ref{sec:HH_dagger} we had seen that the periodic cyclic homology of \(\comg{R}{J}{\alpha}\)
is naturally isomorphic to the cohomology of the de~Rham complex for
\(\comg{R}{J}{\alpha}\) made periodic.
Thus we see that periodic rigid cohomology is isomorphic to the periodic cyclic homology (as defined in~\cite{corval}) of
the pro-algebra given by the projective system of algebras
\(\comg{R}{J}{\nicefrac1m}\);
here pro-algebras are needed, and they produce the homotopy projective limit above.

Our article develops a general conceptual framework for the study of bornological completions that generalize the weak completions of Monsky--Washnitzer, and contains much more material than what is needed for the proof of our result in Theorem \ref{thm:rig}. In the last section we describe a quick route to the construction of the natural complexes that compute rigid cohomology. This approach uses only elementary properties of bornological completions and avoids the finer analysis of spectral radius or linear growth completions. It also brings to light more clearly the analogy with Grothendieck's original construction of infinitesimal cohomology in characteristic 0.

\goodbreak

\noindent\emph{Acknowledgements.} The authors wish to thank Peter
Schneider for communicating a very helpful result at an early stage of the project.

\subsection{Notation}
\label{sec:notations}

Let~\(\dvr\) be a complete discrete valuation ring and
let~\(\dvgen\) be a generator for the maximal ideal in~\(\dvr\).
Let~\(\dvf\) be the fraction field of~\(\dvr\), that is,
\(\dvf=\dvr[\dvgen^{-1}]\).  Let \(\resf=\dvr/\dvgen \dvr\) be the
residue field.  In the sections on periodic cyclic homology and
rigid cohomology, we need~\(\dvf\) to have characteristic~\(0\).  In the
earlier sections, this assumption is not needed and not made.

Every element of~\(\dvf\) is written uniquely as \(x=u\dvgen^{\nu(x)}\),
where \(u\in \dvr\setminus\dvgen \dvr\) and \(\nu(x)\in\Z\cup\{-\infty\}\)
is the \emph{valuation} of~\(x\).  We fix \(0<\pepsilon<1\), and define
the \emph{absolute value} \(\abs{\hphantom{x}}\colon \dvf\to\R_{\ge
  0}\) by \(\abs{x}=\pepsilon^{\nu(x)}\) for \(x\neq 0\) and \(\abs{0}=0\).

If~\(M\) is a \(\dvr\)\nb-module, let~\(\ul{M}\) be the associated
\(\dvf\)\nb-vector space \(M \otimes \dvf\).  A \(\dvr\)\nb-module~\(M\) is
\emph{flat} if and only if the canonical map \(M \to \ul{M}\)
is injective, if and only if it is \emph{torsion-free}, that is,
\(\dvgen x=0\) implies \(x=0\).

\section{Bornological modules over discrete valuation rings}
\label{sec:bornologies}

This section defines (convex) bornological \(\dvr\)\nb-modules,
convergence of sequences in them, and related notions of
separatedness and completeness, completion, and completed tensor
products.  We also establish some basic results about these notions.
We show, for example, that the category of complete bornological
modules is equivalent to the category of inductive systems of
complete \(\dvr\)\nb-modules with injective transition maps
(Proposition~\ref{prop:born_versus_ind_complete}).  We use the
notation in Section~\ref{sec:notations} without further comment.

\begin{definition}[compare~\cite{Hogbe-Nlend:Bornologies}]
  Let~\(M\) be a \(\dvr\)\nb-module.  A (convex) \emph{bornology}
  on~\(M\) is a family~\(\bdd\) of subsets of~\(M\), called
  \emph{bounded} subsets, satisfying
  \begin{itemize}
  \item every finite subset is in~\(\bdd\);
  \item subsets and finite unions of sets in~\(\bdd\) are in~\(\bdd\);
  \item if \(S \in \bdd\), then the \(\dvr\)\nb-submodule generated
    by~\(S\) also belongs to~\(\bdd\).
  \end{itemize}
  A \emph{bornological $\dvr$\nb-module} is a $\dvr$\nb-module equipped with a bornology.
\end{definition}

The condition on unions holds if and only if~\(\bdd\) with inclusion
as partial order is directed.  The third condition is a convexity
condition and is related in~\cite{Bambozzi:Affinoid} to standard
notions of convexity for bornological \(\mathbb{R}\)\nb-vector
spaces.  One may also consider bornologies that do not satisfy the
convexity condition above, but we shall not do so in this article.

\begin{definition}
  \label{def:born_dvf}
  A \emph{bornological \(\dvf\)\nb-vector space} is a bornological
  \(\dvr\)\nb-module such that multiplication by~\(\dvgen\) is an
  invertible map with bounded inverse.
\end{definition}

\begin{example}
  \label{ex:fine}
  Let~\(M\) be a \(\dvr\)\nb-module.  The collection~\(\cF\) of all
  subsets of finitely generated submodules of~\(M\) is a bornology, called the
  \emph{fine} bornology.

  The fine bornology is the smallest possible bornology: a finitely
  generated submodule of~\(M\) is the set of \(\dvr\)\nb-linear
  combinations of the elements of a finite subset and hence bounded in any
  bornology.  If~\(M\) itself is finitely generated, then all bounded
  subsets are bounded in the fine bornology, so this is the only
  bornology on~\(M\) in this case.
\end{example}

\begin{example}
  \label{exa:norm_bornology}
  Let~\(M\) be a \(\dvr\)\nb-module.  A (nonarchimedean) \emph{seminorm}
  on~\(M\) is a map \(\norm{\hphantom{x}}\colon M\to \R_{\ge 0}\) such
  that
  \begin{alignat*}{2}
    \norm{ax}&=\abs{a}\norm{x},&\qquad \text{for all }&a\in \dvr,\ x\in M,\\
    \norm{x+y}&\le\max\{\norm{x},\norm{y}\}&\qquad \text{for all }&x,y\in M.
  \end{alignat*}
  A \emph{norm} is a seminorm for which \(\norm{x}=0\)
  implies \(x=0\).
  If~\(M\)
  admits a norm, \(M\) must be torsion-free.  Let~\((I,\le)\)
  be a directed set and let~\(\norm{\hphantom{x}}_i\)
  for \(i\in I\)
  be seminorms on~\(M\)
  with \(\norm{\hphantom{x}}_i \le \norm{\hphantom{x}}_j\)
  for \(i,j\in I\)
  with \(i\le j\).
  Call a subset of~\(M\)
  bounded if it is bounded with respect to all these seminorms.  This
  is a bornology on~\(M\).
  It is analogous to the standard (von Neumann) bornology on a locally
  convex topological vector space over~\(\R\), see
  \cite{Meyer:HLHA}*{Example~1.11}.

  A seminorm on a 
	\(\dvf\)\nb-vector space $X$ is 
  determined up to equivalence by its unit ball \(M\);  it is equivalent to the \emph{gauge seminorm}
	\begin{equation}\label{eq:gauge}
	\norm{x}_M = \inf \{\abs{a}^{-1} : a x\in M\}.
	\end{equation}
	Indeed, by \cite{schneider}*{Lemma~2.2.ii}, we have
\begin{equation}\label{ineq:gauge}
	\pepsilon\norm{x}_M\le \norm{x}\le \norm{x}_M \qquad (x\in X).
	\end{equation}
		
	Thus the bornology associated to a seminorm
  determines the seminorm up to equivalence.
\end{example}

An \emph{inductive system} in a category~\(\cC\) is a functor \(X\colon
I\to \cC\) from a directed set~\(I\).  We identify such a functor with
the collection \(X=\{X_i,\,\phi_{ij}\colon X_i\to X_j\}\) of the objects
\(X_i=X(i)\) and the transition maps \(\phi_{ij}=X(i<j)\).  The
inductive systems in \(\cC\) form a category \(\ind\cC\), where the
homomorphisms from \(X\) to \(Y\colon J\to\cC\) are given by
\[
\hom_{\ind\cC}(X,Y)=\varprojlim_i\varinjlim_j\hom_\cC(X_i,Y_j).
\]

\begin{proposition}
  \label{prop:born_versus_ind}
  The category of bornological \(\dvr\)\nb-modules and bounded
  \(\dvr\)\nb-\hspace*{0pt}module maps is equivalent to the full
  subcategory of the category of inductive systems of\/
  \(\dvr\)\nb-modules consisting of those inductive systems
  \((X_i,\varphi_{ij}\colon X_i \to X_j)\) that have injective
  maps~\(\varphi_{ij}\).
\end{proposition}

\begin{proof}
  Let~\(X\) be a bornological \(\dvr\)\nb-module.  Let~\(\born\) be
  the directed set of bounded subsets, and let
  \(\born'\subseteq\born\) be the subset of all bounded
  \(\dvr\)\nb-submodules.  Convexity means that~\(\born'\) is
  cofinal in~\(\born\), so it is again directed.  The
  \(\dvr\)\nb-submodules \(S\in\born'\) with the inclusion maps
  \(S\hookrightarrow T\) for \(S\subseteq T\) form an inductive
  system of \(\dvr\)\nb-modules with injective transition maps.  A
  bounded map \((X_1,\born_1)\to (X_2,\born_2)\) maps any bounded
  \(\dvr\)\nb-submodule of~\(X_1\) into some bounded
  \(\dvr\)\nb-submodule of~\(X_2\) and thus gives a morphism of
  inductive systems.  Hence we have defined a functor~\(A\) from
  bornological \(\dvr\)\nb-modules to inductive systems.

  Conversely, an inductive system of \(\dvr\)\nb-modules
  \((X_i,\varphi_{ij}\colon X_i \to X_j)\) has an inductive limit
  \(\dvr\)\nb-module~\(X\).  We equip~\(X\) with the ``inductive
  limit bornology'': a subset of~\(X\) is bounded if and only if it
  is contained in the image of~\(X_i\) for some~\(i\).  A morphism
  of inductive systems induces a bounded \(\dvr\)\nb-module map on
  the inductive limits, so we get a functor~\(L\) from inductive
  systems of \(\dvr\)\nb-modules to bornological
  \(\dvr\)\nb-modules.  The inductive limit comes with natural
  bounded \(\dvr\)\nb-module maps \(X_i\to X\).  These are injective
  if all the transition maps \(\varphi_{ij}\colon X_i\to X_j\) are
  injective.  In this case, the inductive system \(A\circ
  L(X_i,\varphi_{ij})\) is isomorphic to~\((X_i,\varphi_{ij})\)
  because the bounded \(\dvr\)\nb-submodules~\(X_i\) form a cofinal
  subset in~\(\born'\).  Given a bornological
  \(\dvr\)\nb-module~\((X,\born)\), the inclusion maps
  \(S\hookrightarrow X\) for \(S\in\born'\) induce a natural
  isomorphism \(L\circ A(X) \cong X\).  Hence the functors \(A\)
  and~\(L\) provide the desired equivalence of categories.
\end{proof}

Let~\(S\) be a \(\dvr\)\nb-module.  The
\emph{\(\dvgen\)\nb-adic completion} of~\(S\) is the \(\dvr\)\nb-module
\[
\coma{S} = \varprojlim_{m\to\infty} S/\dvgen^m S,
\]
equipped with the map \(S\to \coma{S}\),
\(x\mapsto [x\bmod \dvgen^m S]_{m\in\N}\).
The module~\(S\)
is \emph{\(\dvgen\)\nb-adically
  complete} if the map \(S\to\coma{S}\)
is an isomorphism, and \emph{\(\dvgen\)\nb-adically
  separated} if the map \(S\to\coma{S}\)
is injective, that is, \(\bigcap_{m\in\N} \dvgen^m S = \{0\}\).
We give each \(S/\dvgen^m S\)
the discrete topology, and then~\(\coma{S}\)
the projective limit topology.  The topology on~\(S\)
induced by~\(\coma{S}\)
is called the \emph{\(\dvgen\)\nb-adic
  topology} on~\(S\).
Thus a sequence~\((x_n)_{n\in\N}\)
in~\(S\)
converges to \(x\in S\)
if and only if for each \(m\in\N\)
there are only finitely many \(n\in\N\)
with \(x_n - x \notin \dvgen^m S\).
The \(\dvgen\)\nb-adic
topology on~\(S\)
is generated by the pseudometric
\(d_S(x,y) = \pepsilon^{\sup \{m: x-y \in \dvgen^m S\}}\),
which is a metric if and only if~\(S\)
is \(\dvgen\)\nb-adically separated.

\begin{definition}
  \label{def:converge}
  Let~\(X\)
  be a bornological \(\dvr\)\nb-module.
  Let~\((x_n)_{n\in\N}\)
  be a sequence in~\(X\)
  and let \(x\in X\).
  If \(S\subseteq X\)
  is bounded, then \((x_n)_{n\in\N}\)
  \emph{\(S\)\nb-converges}
  to~\(x\)
  if there is a sequence~\((\delta_n)_{n\in\N}\)
  in~\(\dvr\)
  with \(\lim \delta_n=0\)
  in the \(\dvgen\)\nb-adic
  topology and \(x_n-x \in \delta_n\cdot S\)
  for all \(n\in\N\).
  A sequence in~\(X\)
  \emph{converges} if it \(S\)\nb-converges
  for some bounded subset~\(S\)
  of~\(X\).
  If~\(S\)
  is bounded, then \((x_n)_{n\in\N}\)
  is \emph{\(S\)\nb-Cauchy}
  if there is a sequence~\((\delta_n)_{n\in\N}\)
  in~\(\dvr\)
  with \(\lim \delta_n=0\)
  and \(x_n-x_m \in \delta_l\cdot S\)
  for all \(n,m,l\in\N\)
  with \(n,m\ge l\).
  A sequence in~\(X\)
  is \emph{Cauchy} if it is \(S\)\nb-Cauchy
  for some bounded \(S\subseteq X\).
  The bornological \(\dvr\)\nb-module~\(X\)
  is \emph{separated} if limits of convergent sequences are unique.
  It is \emph{complete} if it is separated and for every bounded
  \(S\subseteq X\)
  there is a bounded \(S'\subseteq X\)
  so that all \(S\)\nb-Cauchy sequences are \(S'\)\nb-convergent.
\end{definition}

\begin{example}
  \label{exa:fine_complete}
  Any \(\dvr\)\nb-module~\(M\)
  is complete in the fine bornology of Example~\ref{ex:fine} because
  all finitely generated \(\dvr\)\nb-modules
  are \(\dvgen\)\nb-adically
  complete.  A finite-dimensional bornological \(\dvf\)\nb-vector
  space carries a unique \emph{separated} bornology, namely, the fine
  bornology.
\end{example}

\begin{proposition}
  \label{prop:separated_complete}
  Let~\(X\) be a bornological \(\dvr\)\nb-module and~\(\born\) its
  family of bounded subsets.  A sequence in~\(X\) converges if and
  only if it is contained in a bounded \(\dvr\)\nb-submodule
  \(Y\subseteq X\) and converges in~\(Y\) in the \(\dvgen\)\nb-adic
  topology.  The bornological \(\dvr\)\nb-module \(X\) is separated
  if and only
  if all bounded \(\dvr\)\nb-submodules are \(\dvgen\)\nb-adically
  separated, if and only if bounded, \(\dvgen\)\nb-adically
  separated \(\dvr\)\nb-submodules are cofinal in~\(\born\) for
  inclusion.  It is complete if and only if bounded,
  \(\dvgen\)\nb-adically complete \(\dvr\)\nb-submodules are cofinal
  in~\(\born\).
\end{proposition}

\begin{proof}
  We may choose~\(S\)
  in Definition~\ref{def:converge} as a bounded \(\dvr\)\nb-submodule
  and assume that \(x_0\in S\).
  Then \(x = (x-x_0)+x_0 \in \delta_0\cdot S + S\subseteq S\)
  and \(x_n= (x_n-x)+x\in \delta_n \cdot S + S\subseteq S\)
  for all \(n\in\N\).
  A sequence~\((x_n)\)
  in~\(S\)
  \(S\)\nb-converges
  to \(x\in S\)
  if and only if it converges in the \(\dvgen\)\nb-adic
  topology on~\(S\).
  If a sequence has two limit points \(x\neq y\),
  then we may find a single bounded \(\dvr\)\nb-module~\(S\)
  so that it \(S\)\nb-converges
  towards both \(x\)
  and~\(y\).
  Hence~\(S\)
  is not \(\dvgen\)\nb-adically
  separated.  Conversely, if there is such an~\(S\),
  then there is a \(\dvgen\)\nb-adically
  convergent sequence in~\(S\)
  with two limit points.  Hence the bornological
  \(\dvr\)\nb-module~\(X\)
  is not separated.  Since submodules of \(\dvgen\)\nb-adically
  separated \(\dvr\)\nb-modules
  remain separated, all bounded \(\dvr\)\nb-submodules
  are \(\dvgen\)\nb-adically
  separated once this happens for a cofinal set of bounded
  \(\dvr\)\nb-submodules.

  If a sequence~\((x_n)\) is Cauchy, then there is a bounded
  \(\dvr\)\nb-submodule~\(S\) such that \(x_n\in S\) for all
  \(n\in\N\) and the sequence is \(S\)\nb-Cauchy in the standard
  quasi-metric defining the \(\dvgen\)\nb-adic topology on~\(S\).
  In particular, if~\(S\) is \(\dvgen\)\nb-adically complete, then
  any \(S\)\nb-Cauchy sequence is \(S\)\nb-convergent.  Thus~\(X\)
  is complete if bounded, \(\dvgen\)\nb-adically complete
  \(\dvr\)\nb-submodules are cofinal in~\(\born\).  Conversely,
  assume that~\(X\) is complete and let \(S\in\born\).  Then~\(S\)
  is contained in a bounded \(\dvr\)\nb-submodule~\(S'\) by
  convexity.  Choose a bounded \(\dvr\)\nb-submodule \(S''\subseteq
  X\) so that \(S'\)\nb-Cauchy sequences are \(S''\)\nb-convergent.
  Any point in the \(\dvgen\)\nb-adic completion~\(\coma{S}'\)
  of~\(S'\) is the limit of a Cauchy sequence in~\(S'\).  Such
  sequences are \(S'\)\nb-Cauchy and hence \(S''\)\nb-convergent.
  Writing a point in~\(\coma{S}'\) as a limit of a Cauchy
  sequence and taking its limit in~\(S''\) defines a bounded
  \(\dvr\)\nb-module map \(\coma{S}'\to S''\).  The kernel of
  this map is \(\dvgen\)\nb-adically closed because~\(X\) is
  separated.  Hence its image is a quotient of a
  \(\dvgen\)\nb-adically complete \(\dvr\)\nb-module by a closed
  submodule; such quotients remain \(\dvgen\)\nb-adically complete.
  Hence the image of~\(\coma{S}'\) is a bounded,
  \(\dvgen\)\nb-adically complete \(\dvr\)\nb-submodule of~\(X\)
  containing~\(S\).  This proves that such submodules are cofinal
  in~\(\born\).
\end{proof}

\begin{remark}
  \label{rem:torsion-free_separated}
  Let~\(X\) be torsion-free.  Then \(y= \lambda x\) for
  \(x\in X\) and \(\lambda \in \dvr\) with \(\lambda\neq0\)
  determines~\(x\).  Thus \(\lambda^{-1} y = x\) is
  well-defined.  Then~\(X\) is not separated if and only if~\(X\)
  contains a bounded \(\dvf\)\nb-vector subspace.  If \(0\neq x\in
  \bigcap \dvgen^l S\) for some bounded \(\dvr\)\nb-submodule~\(S\),
  then the \(\dvr\)\nb-submodule generated by~\(\dvgen^{-n} x\) for
  \(n \in \N\) is a bounded \(\dvf\)\nb-vector subspace in~\(X\).
\end{remark}

\begin{proposition}
  \label{prop:born_versus_ind_complete}
  The category \(\borno\) of complete bornological \(\dvr\)\nb-modules
  and bounded \(\dvr\)\nb-module maps is equivalent to the full
  subcategory of the category of inductive systems of
  \(\dvgen\)\nb-adically complete \(\dvr\)\nb-modules consisting of
  all inductive systems \((X_i,\varphi_{ij}\colon X_i \to X_j)\)
  with injective maps~\(\varphi_{ij}\).  A similar statement holds
  for separated bornological \(\dvr\)\nb-modules and inductive
  systems of \(\dvgen\)\nb-adically separated \(\dvr\)\nb-modules.
\end{proposition}

\begin{proof}
  Let~\(\born\) be the family of bounded subsets of~\(X\).
  Proposition~\ref{prop:separated_complete} shows that the
  \(\dvgen\)\nb-adically complete or separated bounded
  \(\dvr\)\nb-modules of~\(X\) are cofinal in~\(\born\) if and only
  if~\(X\) is complete or separated, respectively.  Now copy the
  proof of Proposition~\ref{prop:born_versus_ind} with these cofinal
  subsets in the place of~\(\born'\).
\end{proof}

\begin{example}
  \label{ex:banach_mods}
  Let~\(M\) be a \(\dvr\)\nb-module with a \emph{seminorm} as in
  Example~\ref{exa:norm_bornology}.  If this is a norm, \(M\)
  becomes an ultra-metric space; when it is complete, we call~\(M\)
  a \emph{Banach} \(\dvr\)\nb-module.  A \emph{Banach
    \(\dvf\)\nb-vector space} is a Banach \(\dvr\)\nb-module where
  multiplication by~\(\dvgen\) is bijective.  Since
  \(\norm{ax}=\abs{a}\norm{x}\), every torsion element \(x\in M\)
  has \(\norm{x}=0\); thus normed modules are torsion-free or,
  equivalently, flat.  If~\(M\) is any torsion-free
  $\dvr$\nb-module, then the valuation \(\nu\colon M\to
  \N\cup\{\infty\}\), \(x\mapsto \sup\{n: x\in \dvgen^nM\}\),
  defines a seminorm \(\norm{x}_c=\pepsilon^{\nu(x)}\), the
  \emph{canonical} seminorm.  Its associated topology is the
  \(\dvgen\)\nb-adic topology.  In particular,
  \(\norm{\hphantom{x}}_c\) is a norm if and only if~\(M\) is flat
  and~\(0\) is the only divisible submodule of~\(M\).  If~\(M\) is
  flat, then \((M,\norm{\hphantom{x}}_c)\) is Banach if and only
  if~\(M\) is \(\dvgen\)\nb-adically complete.  Thus
  Proposition~\ref{prop:born_versus_ind_complete} implies that any
  flat complete bornological \(\dvr\)\nb-module is a filtered union
  of Banach submodules with norm-decreasing inclusions.  The norm
  \(\norm{\hphantom{x}}_c\) is always bounded above by~\(1\).
  If~\(M\) is \(\dvgen\)\nb-adically separated,
  \(\norm{\hphantom{x}}\) is any other seminorm, and \(x=\dvgen^ny\)
  with \(y\notin \dvgen M\), then \(\norm{x}=\norm{x}_c\norm{y}\).
  Hence if also~\(\norm{\hphantom{x}}\) is bounded, then the
  identity is bounded as a map \((M,\norm{\hphantom{x}}_c)\to
  (M,\norm{\hphantom{x}})\).  The identity is bounded as a map
  \((M,\norm{\hphantom{x}})\to (M,\norm{\hphantom{x}}_c)\) if and
  only if~\(\norm{\hphantom{x}}\) is bounded below on \(M\setminus
  \dvgen M\) if and only if~\(\dvgen M\) is open in~\(M\).

  We call a Banach module \((M,\norm{\hphantom{x}})\)
  \emph{bornological} if \(\norm{\hphantom{x}}\) is equivalent to
  the canonical norm on each closed ball \(B_\rho\defeq \{x\in M :
  \norm{x}\le\rho\}\) for \(\rho>0\).  Equivalently, there is
    \(\delta>0\) so that \(B_\delta \subset\dvgen\cdot M\) because then automatically \(x\in
    \dvgen\cdot B_\rho\) if
    \(\norm{x}\le\min\{\delta,\pepsilon\cdot\rho\}\).  Thus any
    Banach \(\dvf\)\nb-vector space is bornological.  The
    \(\dvr\)\nb-module \(M=\prod_{n\ge 0}\dvgen^n \dvr\) with the
    supremum norm is a Banach \(\dvr\)\nb-module that is not
    bornological.
\end{example}

\begin{definition}
  \label{def:closed}
  A subset~\(S\) in a bornological \(\dvr\)\nb-module is
  \emph{bornologically closed} if all limits of convergent sequences
  in~\(S\) again belong to~\(S\).  The \emph{bornological closure}
  of~\(S\) is the smallest bornologically closed subset
  containing~\(S\).  The \emph{separated quotient} of a bornological
  \(\dvr\)\nb-module is the quotient by the bornological closure
  of~\(\{0\}\), equipped with the quotient bornology.
\end{definition}

The bornological closure of~\(S\)
may not be easy to compute explicitly.  We certainly have to add all
limits of convergent sequences in~\(S\).
There are usually more convergent sequences in this larger set.  So we
may have to repeat this step to arrive at a bornologically closed
subset, compare \cite{Hogbe-Nlend:Theorie}*{p.~14}.  The notion of a closed subset defines a topology. The translations
\(x\mapsto x+y\)
for a fixed element~\(y\)
and the maps \(x\mapsto \lambda x\)
for \(\lambda \in \dvr\)
are then continuous and closed.  This implies that the
bornological closure of a \(\dvr\)\nb-submodule
is again a \(\dvr\)\nb-submodule,
compare \cite{Hogbe-Nlend:Theorie}*{Corollaire~1 on p.~15}.  A
bornological \(\dvr\)\nb-module~\(X\)
is separated if and only if~\(\{0\}\)
is bornologically closed.  Thus the separated
quotient~\(X/\overline{\{0\}}\)
is indeed separated, and it is the largest quotient with this
property: any bounded \(\dvr\)\nb-module
map \(X\to Y\)
into a separated bornological \(\dvr\)\nb-module
factors uniquely through~\(X/\overline{\{0\}}\).
A \(\dvr\)\nb-submodule~\(X\)
of a complete bornological \(\dvr\)\nb-module~\(Y\),
equipped with the subset bornology, is complete if and only if~\(X\)
is bornologically closed in~\(Y\)
because a sequence in~\(X\)
converges in the subset bornology on~\(X\)
if and only if it converges in~\(Y\).

\begin{definition}
  \label{def:completion}
  The \emph{completion} of a bornological \(\dvr\)\nb-module~\(X\)
  is a complete bornological \(\dvr\)\nb-module~\(\comb{X}\)
  with a bounded map \(X\to \comb{X}\)
  that is universal in the sense that any map from~\(X\)
  to a complete bornological \(\dvr\)\nb-module
  factors uniquely through it.
\end{definition}

In order to construct completions, we first discuss inductive limits
in the categories of separated and complete bornological
\(\dvr\)\nb-modules.  The category of bornological
\(\dvr\)\nb-modules is complete and cocomplete, that is, any diagram
has both a limit and a colimit.  They are constructed by equipping
the limit or colimit in the category of \(\dvr\)\nb-modules with a
canonical bornology.  Any limit may be described as a submodule of a
product; hence limits of separated bornological \(\dvr\)\nb-modules
remain separated.  Thus the category of separated bornological
\(\dvr\)\nb-modules has all limits, and these are just the same as
in the larger category of bornological \(\dvr\)\nb-modules.  In
contrast, a colimit may be described as a quotient of a direct sum,
and such a quotient need not remain separated.  To make it so, we
take the separated quotient.  This has exactly the right universal
property for a colimit in the category of separated bornological
\(\dvr\)\nb-modules.

For a diagram of complete bornological \(\dvr\)\nb-modules,
the limit in the category of bornological \(\dvr\)\nb-modules
is easily seen to be complete again because it is the kernel of a
certain map between products.  Products inherit completeness from
their factors, and the kernel of a bounded linear map between
separated bornological \(\dvr\)\nb-modules
is bornologically closed.  Hence limits also inherit completeness.
Quotients of complete bornological \(\dvr\)\nb-modules
by bornologically closed \(\dvr\)\nb-submodules
remain complete.  Hence the separated quotient of the usual colimit is
a complete bornological \(\dvr\)\nb-module
and so has the right universal property for the colimit in the
category of complete bornological \(\dvr\)\nb-modules.

\begin{proposition}
  \label{prop:completion}
  Completions of bornological \(\dvr\)\nb-modules
  always exist and may be constructed as follows.  Write
  \(X=\varinjlim {}(X_i)_{i\in I}\)
  as the inductive limit of the directed set of its bounded
  \(\dvr\)\nb-submodules.
  Then~\(\comb{X}\)
  is the separated quotient of the bornological inductive limit
  \(\varinjlim {}(\coma{X_i})_{i\in I}\).
  The completion functor commutes with colimits, that is, the
  completion of a colimit of a diagram of bornological
  \(\dvr\)\nb-modules
  is the \emph{separated quotient} of the colimit of the diagram of
  completions.
\end{proposition}

\begin{proof}
  Let~\(Y\)
  be a complete bornological \(\dvr\)\nb-module.
  Then \(Y \cong \varinjlim {}(Y_j)_{j\in J}\)
  with \(\dvgen\)\nb-adically
  complete \(\dvr\)\nb-modules~\(Y_j\)
  by Proposition~\ref{prop:born_versus_ind_complete}.  Hence
  \begin{multline*}
    \Hom(X,Y)
    \cong \varprojlim_i \varinjlim_j \Hom(X_i,Y_j)
    \cong \varprojlim_i \varinjlim_j \Hom(\coma{X_i},Y_j)
    \\\cong \varprojlim_i \Hom(\coma{X_i},Y)
    \cong \Hom\bigl(\varinjlim_i \coma{X_i},Y\bigr)
    \cong \Hom\Bigl(\varinjlim_i \coma{X_i} \Bigm/ \overline{\{0\}},Y\Bigr).
  \end{multline*}
  Thus \(\varinjlim_i \coma{X_i} \Bigm/ \overline{\{0\}}\),
  the separated quotient of the inductive limit of
  \((\coma{X_i})_{i\in I}\),
  has the universal property that defines the completion.  The
  completion functor is defined as a left adjoint to the inclusion
  functor and therefore commutes with colimits, where the colimit in
  the category of complete bornological \(\dvr\)\nb-modules
  is the separated quotient of the usual colimit.
\end{proof}

The separated quotient in Proposition~\ref{prop:completion} is hard to
control.  It may be needed because the maps
\(\coma{X_i} \to \coma{X_j}\)
for \(i\le j\)
need not be injective, although the maps \(X_i\to X_j\)
are injective by construction:

\begin{example}
  \label{exa:non-injective_on_completion}
  Let \(\dvr=\mathbb{Z}_p\),
  \(\dvf=\Q_p\)
  for some prime~\(p\).
  Let \(X_0 = \mathbb{Z}_p[t]\)
  and let~\(X_1\)
  be the \(\Z_p\)\nb-submodule
  of \(\Q_p[t]\)
  generated by~\(X_0\)
  and by the polynomials
  \(f_m = p^{-m}(1+pt+(pt)^2+\dotsb + (pt)^{m-1})\)
  for \(m\in\N\).
  Let \(\coma{X_0}\)
  and~\(\coma{X_1}\)
  be the \(\dvgen\)\nb-adic
  completions of the \(\dvr\)\nb-modules
  \(X_0\)
  and~\(X_1\),
  respectively.  The series \(\sum_{i=0}^\infty (pt)^i\)
  converges towards a non-zero point in~\(\coma{X_0}\)
  that is mapped to~\(0\) in~\(\coma{X_1}\) because
  \[
  \sum_{i=0}^n (pt)^i = p^m f_m + \sum_{i=m+1}^n (pt)^i \in p^m X_1
  \]
  for all \(n\ge m\) and all \(m\in\N\).  Thus the map
  \(\coma{X_0}\to\coma{X_1}\) induced by the inclusion
  \(X_0\to X_1\) is not injective.
\end{example}

Hence the transition maps in the inductive system
\((\coma{X_i})_{i\in I}\)
in Proposition~\ref{prop:completion} need not be injective.  This is
why separated quotients may really be needed in
Proposition~\ref{prop:completion}.  There are examples of separated
bornological vector spaces over~\(\mathbb{R}\)
for which the completion is zero.  In particular, the canonical map
\(X\to\comb{X}\) may fail to be injective.  We expect such examples
also over~\(\dvr\), but we have not tried hard to find one.

There is a canonical map \(\comb{X}\to \coma{X}\)
by Proposition~\ref{prop:completion}.  This map may fail to be
injective and surjective: with the fine bornology,
any \(\dvr\)\nb-module~\(X\)
is bornologically complete so that \(\comb{X}=X\)
need not be \(\dvgen\)\nb-adically
separated or \(\dvgen\)\nb-adically complete.

Next we consider tensor products, following the recipe for
\(\mathbb{R}\)-vector spaces in~\cite{Hogbe-Nlend:Completions}.
Undecorated tensor products are taken over~\(\dvr\).  Let \(S\)
and~\(T\) be \(\dvr\)\nb-modules.  Then \(S\otimes T\) is the target
of the universal \(\dvr\)\nb-bilinear map on~\(S\times T\): any
\(\dvr\)\nb-bilinear map \(S\times T\to U\) for a
\(\dvr\)\nb-module~\(U\) extends uniquely to a \(\dvr\)\nb-bilinear
map \(S\otimes T\to U\).  The complete bornological tensor product
of two \(\dvr\)\nb-modules is defined by an analogous universal
property with respect to bounded bilinear maps.  A bilinear map
\(b\colon S\times T\to U\) between bornological modules is
\emph{bounded} if \(b(M\times N)\) is bounded whenever \(M\)
and~\(N\) are bounded in \(S\) and~\(T\).

\begin{definition}
  \label{def:complete_tensor}
  First let \(X\) and~\(Y\) be bornological \(\dvr\)\nb-modules.
  Their \emph{bornological tensor product} is a bornological
  \(\dvr\)\nb-module \(X\otimes Y\) with a bounded
  \(\dvr\)\nb-bilinear map \(b\colon X\times Y\to X\otimes Y\) that
  is universal in the sense that any bounded
  \(\dvr\)\nb-bilinear map \(X\times Y \to W\) to a bornological
  \(\dvr\)\nb-module~\(W\) factors uniquely through~\(b\).
  Secondly, let \(X\) and~\(Y\) be complete bornological
  \(\dvr\)\nb-modules.  Their \emph{complete bornological tensor
    product} is a complete bornological \(\dvr\)\nb-module
  \(X\hotimes Y\) with a bounded \(\dvr\)\nb-bilinear map \(b\colon
  X\times Y\to X\hotimes Y\) that is universal in the sense that any
  bounded \(\dvr\)\nb-bilinear map \(X\times Y \to W\) to a
  complete bornological \(\dvr\)\nb-module~\(W\) factors uniquely
  through~\(b\).
\end{definition}

Such tensor products exist and are described more concretely in the
following lemma; it also justifies using the same
notation~\(\otimes\) both for the \(\dvr\)\nb-module tensor product
and the bornological \(\dvr\)\nb-module tensor product.

\begin{lemma}
  \label{lem:complete_tensor}
  Let \(X\) and~\(Y\) be bornological \(\dvr\)\nb-modules.  Then
  their bornological tensor product \(X\otimes Y\) is the usual
  \(\dvr\)\nb-module tensor product equipped with the bornology that
  is generated by the images of \(S\otimes T\) in \(X\otimes Y\),
  where \(S\) and~\(T\) run through the bounded
  \(\dvr\)\nb-submodules in \(X\) and~\(Y\), respectively.

  Let \(X\) and~\(Y\) be complete bornological \(\dvr\)\nb-modules.
  Then their complete bornological tensor product \(X\hotimes Y\) is
  the bornological completion of the bornological tensor product
  \(X\otimes Y\).  Write \(X\cong \varinjlim (X_i, \varphi_{ij})\)
  and \(Y\cong \varinjlim (Y_l, \psi_{lm})\) for inductive systems
  of complete \(\dvr\)\nb-modules with injective maps
  \(\varphi_{ij}\) and~\(\psi_{lm}\) as in
  Proposition~\textup{\ref{prop:born_versus_ind_complete}}.  Then
  \(X\otimes Y\) is the colimit of the inductive system of bounded
  \(\dvr\)\nb-modules \((X_i\otimes Y_l, \varphi_{ij}\otimes
  \psi_{lm})\).  The complete bornological tensor product is the
  separated quotient of the colimit of the inductive system of
  bounded \(\dvr\)\nb-modules \((X_i\haotimes Y_l,
  \varphi_{ij}\haotimes \psi_{lm})\), where \(X_i \haotimes Y_l\)
  for two \(\dvr\)\nb-modules denotes the \(\dvgen\)\nb-adic
  completion of their tensor product.
\end{lemma}

\begin{proof}
  Let~\(Z\) be a bornological \(\dvr\)\nb-module.  A bilinear map
  \(b\colon X\times Y\to Z\) induces a linear map \(b_*\colon
  X\otimes Y\to Z\).  The map~\(b\) is bounded if and only if
  \(b(S,T)\subseteq Z\) is bounded for all bounded
  \(\dvr\)\nb-submodules \(S\subseteq X\), \(T\subseteq Y\).  This
  is equivalent to \(b_*(\iota(S\otimes T))\) being bounded for all
  such \(S,T\), where \(\iota(S\otimes T)\) denotes the image of
  \(S\otimes T\) in~\(X\otimes Y\).  Hence there is a natural
  bijection between bounded bilinear maps \(X\times Y\to Z\) and
  bounded linear maps \(X\otimes Y\to Z\), where we equip~\(X\otimes
  Y\) with the bornology generated by the \(\dvr\)\nb-submodules
  \(\iota(S\otimes T)\) for all bounded \(\dvr\)\nb-submodules
  \(S\subseteq X\), \(T\subseteq Y\).  Since the sets of bounded
  \(\dvr\)\nb-submodules \(S\subseteq X\), \(T\subseteq Y\) are
  directed, so is the resulting set of \(\dvr\)\nb-submodules
  \(\iota(S\otimes T) \subseteq X\otimes Y\).  Hence a subset of
  \(X\otimes Y\) is bounded if and only if it is contained in
  \(\iota(S\otimes T)\) for some bounded \(\dvr\)\nb-submodules
  \(S\subseteq X\), \(T\subseteq Y\).

  The bornological completion \(\comb{X\otimes Y}\) is a complete
  bornological \(\dvr\)\nb-module with the property that bounded
  linear maps \(\comb{X\otimes Y}\to Z\) for complete bornological
  \(\dvr\)\nb-modules~\(Z\) correspond bijectively to bounded linear
  maps \(X\otimes Y\to Z\).  Since the latter correspond bijectively
  to bounded bilinear maps \(X\times Y\to Z\), the completion
  \(\comb{X\otimes Y}\) has the correct universal property for a complete
  bornological tensor product.

  If \(X=\varinjlim X_i\), \(Y=\varinjlim Y_l\), then a bounded
  \(\dvr\)\nb-bilinear map \(b\colon X\times Y\to Z\) is equivalent
  to a family of bounded \(\dvr\)\nb-bilinear maps \(b_{il}\colon
  X_i\times Y_l \to Z\) that are compatible in the usual sense,
  \((b_{il})\in \varprojlim_{il} \Hom^{(2)}(X_i\times Y_l,Z)\),
  where~\(\Hom^{(2)}\) denotes the set of bounded bilinear maps.
  If~\(Z\) is complete, then such a family is equivalent to a family
  of bounded \(\dvr\)\nb-linear maps \((b_{il})_*\colon X_i\haotimes
  Y_l \to Z\) with the same compatibility, \((b_{il})_*\in
  \varprojlim_{il} \Hom(X_i\haotimes Y_l,Z)\).  Thus \(X\hotimes Y\)
  has the same universal property as the colimit of the inductive
  system \(X_i\haotimes Y_l\) in the category of complete
  bornological \(\dvr\)\nb-modules.  This is the separated quotient of
  the colimit.
\end{proof}

\begin{example}
  \label{ex:otimesK}
  The fraction field \(\dvf=\dvr[\dvgen^{-1}]\) is \(\varinjlim_\N
  \dvr\), where the transfer maps are multiplication by~\(\dvgen\).
  Hence if~\(M\) is a flat, complete bornological module and
  \(\{M_\lambda: \lambda\in\Lambda\}\) is the direct system of its
  complete bounded submodules, then
  \[
  M\otimes \dvf
  = \varinjlim_{(\lambda,j)\in\Lambda\times\N} M_\lambda \otimes (\dvgen^{-j}\dvr)
  = M\hotimes \dvf.
  \]
  More generally, if~\(M\)
  is a flat, complete bornological \(\dvr\)\nb-module
  and~\(N\)
  is any \(\dvr\)\nb-module
  with the fine bornology, then \(M\hotimes N=M\otimes N\).
  This holds because \(M_\lambda\otimes N_\nu\)
  is \(\dvgen\)\nb-adically
  complete if \(N_\nu\subseteq N\)
  is finitely generated and~\(M_\lambda\)
  is \(\dvgen\)\nb-adically complete.
\end{example}

If \(X\) and~\(Y\) are bornological \(\dvf\)\nb-vector spaces, then
\(X\otimes Y = X\otimes_\dvf Y\), and the complete bornological
tensor product \(X\hotimes Y\) is a bornological \(\dvf\)\nb-vector
space as well; thus we may define \(X\hotimes_\dvf Y\) to be
\(X\hotimes Y\).

Routine arguments show that the category of complete bornological
\(\dvr\)\nb-modules
with the tensor product~\(\hotimes\)
and with~\(\dvr\)
as tensor unit is a symmetric monoidal category.  Even more, this is a
closed symmetric monoidal category; the internal hom functor is
obtained by equipping the \(\dvr\)\nb-module
\(\Hom(X,Y)\)
of bounded \(\dvr\)\nb-linear
maps \(X\to Y\)
with the \emph{equibounded bornology}, where a set~\(S\)
of maps \(X\to Y\)
is \emph{equibounded} if for each bounded subset \(T\subseteq X\)
the set of \(s(t)\)
with \(s\in S\),
\(t\in T\)
is bounded in~\(Y\).
The \(\dvr\)\nb-module
\(\Hom(X,Y)\)
with this bornology is complete if~\(Y\)
is complete.  A bounded linear map \(X\to\Hom(Y,Z)\)
is equivalent to a bounded bilinear map \(X\times Y\to Z\),
which is then equivalent to a bounded linear map \(X\hotimes Y\to Z\)
by the universal property of the tensor product.  Thus the internal
\(\Hom\)
functor has the correct universal property for a closed monoidal
category.  We shall only use one consequence, which can also be proved
directly:

\begin{proposition}
  \label{pro:tensor_colimits}
  The tensor product functor on the category of bornological
  \(\dvr\)\nb-modules commutes with colimits in each variable.  The
  completed tensor product functor on the category of complete
  bornological \(\dvr\)\nb-modules commutes with
  \textup{(}separated\textup{)} colimits in each variable.
\end{proposition}
The adjective ``separated'' clarifies that a colimit in the
  category of complete bornological \(\dvr\)\nb-modules is the
  separated quotient of the usual colimit.
\begin{proof}
    Any functor with a right
  adjoint commutes with colimits, so this follows from the existence
  of an internal Hom functor.
\end{proof}

\begin{definition}
  \label{def:exactness}
  A bounded \(\dvr\)\nb-linear map \(\varphi\colon X\to Y\) is
  called a \emph{bornological quotient map} if any bounded
  subset~\(S\subseteq Y\) of~\(Y\) is \(\varphi(R)\) for some
  bounded subset~\(R\) of~\(X\).  Equivalently, the map
  \(X/\ker(\varphi) \to Y\) induced by~\(\varphi\) is a bornological
  isomorphism.  Any bornological quotient map is surjective.  An
  \emph{extension} of bornological \(\dvr\)\nb-modules is a diagram
  of bornological \(\dvr\)\nb-modules
  \begin{equation}
    \label{seq:exact}
    0 \longrightarrow X' \overset{f}\longrightarrow X
    \overset{g}\longrightarrow X'' \longrightarrow 0
  \end{equation}
  that is exact in the algebraic sense and such that \(X'\subseteq X\)
  has the subspace bornology and~\(g\) is a bornological quotient
  map.  Equivalently, \(g\) is a cokernel for~\(f\) and~\(f\) is a
  kernel for~\(g\).  A \emph{split extension} is an extension with a
  bounded \(\dvr\)\nb-linear section.
\end{definition}

\begin{lemma}
  \label{lem:exactness}
  Given a complete bornological \(\dvr\)\nb-module~\(Y\) and an
  extension~\eqref{seq:exact} of complete bornological
  \(\dvr\)\nb-modules, we have \(g\hotimes \id_Y = \coker(f\hotimes
  \id_Y)\); moreover, \(g\hotimes \id_Y\) is a bornological quotient
  map and \(\ker(g\hotimes \id_Y)\) is the bornological closure of
  the image of \(f\hotimes \id_Y\) in \(X\hotimes Y\).  The
  functor~\({-}\hotimes Y\) maps split extensions of complete
  bornological \(\dvr\)\nb-modules again to split extensions.
\end{lemma}

\begin{proof}
  Since~\({-}\hotimes Y\) commutes with colimits by
  Proposition~\ref{pro:tensor_colimits}, it preserves cokernels and direct sums.
  The cokernel of \(f\hotimes \id_Y\) in the category of complete
  bornological \(\dvr\)\nb-modules is the quotient of \(X\hotimes
  Y\) by the bornological closure of the image of \(f\hotimes
  \id_Y\).  This gives the second statement.  The last statement
  follows from the fact that~\({-}\hotimes Y\) preserves direct sums: if the
  sequence~\eqref{seq:exact} admits a bounded \(\dvr\)\nb-linear
  section, then \(X\cong X'\oplus X''\).
\end{proof}

\begin{lemma}
  \label{lem:completion_quotient}
  Taking completions preserves bornological quotient maps.
\end{lemma}

\begin{proof}
  Proposition~\ref{prop:completion} describes bornological completions
  through the \(\dvgen\)\nb-adic
  completions of the bounded \(\dvr\)\nb-submodules.
  This reduces the assertion to the fact that \(\dvgen\)\nb-adic
  completions preserve surjections.
\end{proof}

\section{Bornological algebras}
\label{sec:borno}
\numberwithin{theorem}{subsection}
A \emph{bornological algebra}~\(R\) over~\(\dvr\) is an associative
\(\dvr\)\nb-algebra and a bornological \(\dvr\)\nb-module such that the
multiplication map is bounded.  That is, if \(S,T\subseteq R\) are
bounded subsets, then so is~\(S\cdot T\).  We shall
consider only bornological algebras with unit.  The unit map
\(\dvr\to R\) is automatically bounded.

We first define the spectral radius of a bounded
\(\dvr\)\nb-submodule in a bornological \(\dvf\)\nb-algebra.  Based
on this, we define the linear growth bornology on a bornological
\(\dvr\)\nb-algebra and a \(1\)\nb-parameter family of bornologies
on \(\ul{R}\defeq R\otimes \dvf\) for a
\(\dvr\)\nb-algebra~\(R\) and an ideal \(J\triqui R\) with \(\dvgen\in
J\).  These
bornologies yield bornological completions of \(R\) and~\(\ul{R}\).
If~\(R\) is finitely generated and commutative, then we identify the
completion of~\(R\) for the linear growth bornology with the
Monsky--Washnitzer completion of~\(R\).  The completions
of~\(\ul{R}\) for the bornologies associated to~\(J\) interpolate
between the Monsky--Washnitzer completions of the tube algebras (see Definition \ref{def:alpha-tube} and Proposition \ref{pro:linear_growth_on_tube})
of~\(R\) with respect to the powers~\(J^m\) of the ideal~\(J\).

\subsection{Spectral radius estimates}
\label{sec:spectral_radius}

The spectral radius of a bounded subset in a bornological algebra
over \(\R\) or~\(\C\) is defined in
\cite{Meyer:HLHA}*{Definition~3.4}.  We need some notation to carry
this definition over to algebras over the fraction field~\(\dvf\)
of~\(\dvr\).  To some extent, the definition also works for algebras
over~\(\dvr\).

Let~\(R\) be a \(\dvf\)\nb-algebra and let \(M\subseteq R\) be a
\(\dvr\)\nb-submodule.  Fix \(r\in\R_{>0}\).  Recall that we write
\(\pepsilon=\abs{\dvgen}\).  There is a smallest \(j\in\Z\) with
\(\pepsilon^j \le r\),
namely, \(\lceil \log_\pepsilon(r)\rceil\).  We
abbreviate
\[
r\star M \defeq
\dvgen^{\lceil \log_\pepsilon(r)\rceil} M.
\]
Let \(s\in \R_{\ge 0}\),
\(j\in\N\)
and \(N\subseteq R\)
a \(\dvr\)\nb-submodule.
We shall use the following elementary properties of the
operation~\(\star\):
\begin{equation}
  \label{eq:star_props}
  \left\{
    \begin{array}[c]{c}
      \dvgen\left((rs)\star M\right)\subseteq r\star(s\star M)\subseteq (rs)\star M\\
      (r\pepsilon^j)\star M=r\star(\dvgen^j\cdot M)\\
      (r\star M)(s\star N)=r\star(s\star MN)=s\star(r\star MN)\\
      r\le s\Rightarrow r\star M\subseteq s\star M.
    \end{array}
  \right.
\end{equation}

Let \(\sum_{n\in\N} r^{-n} \star M^n\) be the set of all finite sums of
elements in \(\bigcup_{n\in\N} r^{-n} \star M^n\).  This is a
\(\dvr\)\nb-submodule of~\(R\).

\begin{definition}
  \label{def:spectral_radius}
  Let~\(R\) be a bornological \(\dvf\)\nb-algebra and let
  \(M\subseteq R\) be a bounded \(\dvr\)\nb-submodule.  The
  \emph{spectral radius}~\(\varrho(M)\) is defined as the infimum of
  the set of all numbers \(r\in \R_{>0}\) for which \(\sum_{n\in\N}
  r^{-n} \star M^n\) is bounded.  It is~\(\infty\) if no such~\(r\)
  exists.

  Let~\(R\) be a \(\dvr\)\nb-algebra.  If \(r\ge 1\), then we may
  define \(\sum_{n\in\N} r^{-n} \star M^n\) exactly as above.  This
  suffices to give meaning to the assertion \(\varrho(M)\le
  \varrho\) if \(\varrho\ge 1\).
  If~\(R\) is torsion-free, we may define
  \(\varrho(M)\in\R_{\ge 0}\) for a bounded \(\dvr\)\nb-submodule
  \(M\subseteq R\) by viewing~\(M\) as a subset of \(\ul{R}\defeq
  R\otimes \dvf\) with the tensor product bornology, see
  Example~\ref{ex:otimesK}.
\end{definition}

\begin{lemma}
  \label{lem:spectral_radius_comparison}
  Let~\(R\) be a bornological \(\dvr\)\nb-algebra and \(M\subseteq
  R\) a bounded \(\dvr\)\nb-\hspace{0pt}submodule.  Let \(j,c\in\Z\),
  \(c\ge1\).  Then \(\varrho(\dvgen^j M^c) = \pepsilon^j
  \varrho(M)^c\).
\end{lemma}

\begin{proof}
  We first treat the case \(j=0\).  Let
  \begin{gather*}
    S=(\dvr\cdot 1_R + r^{-1/c} \star M+r^{-2/c} \star M^2 + \dotsb
    + r^{-(c-1)/c} \star M^{c-1}),\\
    T=\sum_{i=0}^\infty r^{-i} \star (M^c)^i, \qquad
    U=\sum_{n=0}^\infty r^{-n/c} \star M^n.
  \end{gather*}
  Then \(T\subseteq S\cdot T\subseteq U\)
    by~\eqref{eq:star_props}.  So~\(T\) is bounded if~\(U\) is
    bounded.  Hence \(\varrho(M^c) \le \varrho(M)^c\).  Conversely,
    \eqref{eq:star_props} also gives \(\dvgen \cdot U \subseteq
    S\cdot T\).  Since~\(S\) is bounded anyway, \(\dvgen U\) is
    bounded if~\(T\) is bounded.  If \(s>r\) is arbitrary, then
    \(s^{-n/c} <\pepsilon r^{-n/c}\) for all but finitely
    many~\(n\).  Hence \(\sum_{n=0}^\infty s^{-n/c} \star M^n
    \subseteq \sum_{n=0}^N s^{-n/c} \star M^n + \dvgen\cdot U\) for
    some \(N\in\N\), and the first summand is bounded anyway.  So
    \(\sum_{n=0}^\infty s^{-n/c} \star M^n\) is bounded for all
    \(s>r\) if~\(T\) is bounded.  Hence \(\varrho(M^c) \ge
    \varrho(M)^c\) as well.
    Finally, the assertion extends to \(j\neq0\) because, by
    \eqref{eq:star_props}, \(\sum r^{-n}\star (\dvgen^j M^c)^n =
    \sum (r\pepsilon^{-j})^{-n} \star (M^c)^n\).
\end{proof}

\begin{proposition}
  \label{pro:spectral_radius_linear_growth}
  Let~\((R,\bdd)\) be a bornological \(\dvr\)\nb-algebra.  The
  following are equivalent:
  \begin{enumerate}
  \item \label{en:spectral_radius_linear_growth1} \(\varrho(M)\le
    1\) for all \(M\in\bdd\);
  \item \label{en:spectral_radius_linear_growth3}
    \(\sum_{j=0}^\infty \dvgen^j M^{c j+d}\) is bounded for all
    \(M\in\bdd\), \(c,d\in\N\);
  \item \label{en:spectral_radius_linear_growth2}
    \(\sum_{j=0}^\infty \dvgen^j M^{j+1}\) is bounded for all
    \(M\in\bdd\);
  \item \label{en:spectral_radius_linear_growth4} any bounded subset
    of~\(R\) is contained in a bounded \(\dvr\)\nb-submodule~\(M\)
    with \(\dvgen \cdot M^2 \subseteq M\).
  \end{enumerate}
\end{proposition}

\begin{proof}
  If \(\sum r^{-n} \star M^n\) is bounded for some~\(r\in\R_{\ge0}\),
  then it is also bounded for all \(r' \ge r\).  Hence
  \(\varrho(M)\le1\) if and only \(\sum r^{-n} \star M^n\) is bounded for
  all~\(r\) of the form \(r=\pepsilon^{-1/c}\) with
  \(c\in\Z_{\ge1}\).  The proof of
  Lemma~\ref{lem:spectral_radius_comparison} shows that
  \(\sum_{n\in\N} \pepsilon^{n/c} \star M^n\) is bounded if and only if
  \(\sum_{j=1}^\infty \dvgen^j M^{c j}\) is.  Then
  \(\sum_{j=1}^\infty \dvgen^j M^{c j}\cdot M^d = \sum_{j=1}^\infty
  \dvgen^j M^{c j+d}\) is bounded as well.  So
  \ref{en:spectral_radius_linear_growth1}
  and~\ref{en:spectral_radius_linear_growth3} are equivalent.

  The implication
  \ref{en:spectral_radius_linear_growth3}\(\Rightarrow
  \)\ref{en:spectral_radius_linear_growth2} is trivial.  We prove
  \ref{en:spectral_radius_linear_growth2}\(\Rightarrow
  \)\ref{en:spectral_radius_linear_growth4}.  Let~\(S\) be bounded.
  It is contained in a bounded \(\dvr\)\nb-submodule~\(M\).  Then
  \(U\defeq \sum_{j\ge 0} \dvgen^j M^{j+1}\) is bounded
  by~\ref{en:spectral_radius_linear_growth2}.  We have \(S\subseteq
  M\subseteq U\) by construction, and \(\dvgen \cdot U\cdot U
  \subseteq U\) because
  \begin{equation}
    \label{eq:pi1-multiplicative}
    \dvgen \cdot \sum_{i=0}^\infty \dvgen^i M^{i+1}
    \cdot \sum_{j=0}^\infty \dvgen^j M^{j+1}
    = \sum_{i,j=0}^\infty \dvgen^{i+j+1} M^{i+j+2}
    = \sum_{j=1}^\infty \dvgen^j M^{j+1}\subseteq U.
  \end{equation}
  Finally, we prove
  \ref{en:spectral_radius_linear_growth4}\(\Rightarrow
  \)\ref{en:spectral_radius_linear_growth3}.  Let \(M \subseteq R\)
  and \(c,d\in\N\) be as in~\ref{en:spectral_radius_linear_growth3}.
  Condition~\ref{en:spectral_radius_linear_growth4} gives a bounded
  \(\dvr\)\nb-submodule \(U\subseteq R\) with \(M^c \cup M^d
  \subseteq U\) and \(\dvgen \cdot U\cdot U\subseteq U\).  Hence
  \(\dvgen M^c\cdot U \subseteq \dvgen U\cdot U \subseteq U\).  By
  induction, \(\dvgen^i M^{c i}\cdot U \subseteq U\) for all
  \(i\in\N\).  Thus \(\dvgen^i M^{c i+d} = \dvgen^i M^{ci} \cdot M^d
  \subseteq \dvgen^i M^{c i} U \subseteq U\) for all \(i\in\N\).  So
  \(\sum_{i\ge 0} \dvgen^i M^{c i +d}\) is bounded.
\end{proof}

Let \(R\)
be a flat \(\dvr\)-algebra
and let \(M\subseteq R\)
be a \(\dvr\)\nb-submodule.
Let \(R_M\subseteq R\)
be the subset of all \(x\in R\)
for which there is some \(l\in\N\)
with \(\dvgen^l x\in M\).
In other words, \(R_M=R\cap \ul{M}\);
this is the largest \(\dvr\)\nb-submodule
of~\(R\)
whose elements are absorbed by~\(M\).
There is a unique seminorm on~\(R_M\)
with unit ball~\(M\)
taking values in \(\abs{\dvf}\),
namely, the gauge seminorm~\(\norm{\hphantom{x}}_M\)
of~\eqref{eq:gauge}.  The submodule~\(M\)
satisfies \(\dvgen^m M^2 \subseteq M\)
if and only if \(R_M\cdot R_M\subseteq R_M\) and
\[
\norm{x \cdot y}_M \le \pepsilon^{-m}\cdot \norm{x}_M \norm{y}_M
\]
for all \(x,y\in R_M\).
>From this and~\eqref{ineq:gauge} we obtain that~\(M\)
is the unit ball of a seminormed \(\dvr\)\nb-subalgebra
of~\(R\)
if and only if it satisfies \(\dvgen^m M^2 \subseteq M\)
for some \(m\in\N\).

A complete bornological $\dvr$\nb-algebra~\(R\) is a \emph{Banach algebra} if
there is a norm \(\norm{\hphantom{x}}\) for which the multiplication
map is bounded and \((R,\norm{\hphantom{x}})\) is a bornological
Banach module (in the sense of Example~\ref{ex:banach_mods}) whose
closed balls generate the bornology.  A complete bornological
algebra is a \emph{local Banach algebra} if it is a filtered union
of bornological Banach subalgebras such that the inclusion maps are
norm-decreasing (see \cite{cmr}*{Definition~2.11}).

For example, a complete bornological \(\dvr\)\nb-algebra with \(\varrho(M)\le1\)
for all bounded \(\dvr\)\nb-sub\-modules~\(M\) is a local Banach
algebra.

\begin{remark}\label{rem:local_banach}  An argument similar to that of Proposition \ref{pro:spectral_radius_linear_growth} shows that if \((R,\bdd)\) is a bornological \(\dvf\)-algebra, then \(\varrho(M)<\infty\) for all
\(M\in\bdd\) if and only if the following condition holds:
\[
\forall M\in\bdd  \quad (\exists M\subseteq U\in\bdd,\ l\in\Z)\quad
\dvgen^{l}U^2\subseteq U.
\]
It follows from this and from the arguments above that a complete bornological
\(\dvf\)\nb-algebra is a local Banach algebra if and only if
\(\varrho(M)<\infty\) for all bounded \(\dvr\)\nb-submodules~\(M\)
(compare \cite{Meyer:HLHA}*{Theorem~3.10}).
\end{remark}

\begin{lemma}
  \label{lem:Neumann_series}
  If~\(R\) is a complete bornological \(\dvr\)\nb-algebra and
  \(M\subseteq R\) a bounded \(\dvr\)\nb-submodule with
  \(\varrho(M)<\pepsilon^{-1}\), then \(1-\dvgen z\) is
  invertible in~\(R\) for all \(z\in M\).  If~\(R\) is a complete
  bornological \(\dvf\)\nb-algebra and \(M\subseteq R\) a bounded
  \(\dvr\)\nb-submodule with \(\varrho(M)<1\), then \(1-z\) is
  invertible in~\(R\) for all \(z\in M\).
\end{lemma}

\begin{proof}
  If \(\varrho(M)<\pepsilon^{-1}\), then there is~\(r\) with
  \(\pepsilon<r<\varrho(M)^{-1}\).  Then \(\sum r^n \star M^n\)
  is bounded, and \(\lim_{n\to\infty} (\dvgen z)^n=0\) in the
  \(\dvgen\)\nb-adic topology on \(\sum r^n \star M^n\).  So
  \(\sum_{n=0}^\infty (\dvgen z)^n\) is a bornological Cauchy
  series.  It converges because~\(R\) is complete.  The limit is an
  inverse for~\(1-\dvgen z\).  If~\(R\) is a \(\dvf\)\nb-algebra and
  \(\varrho(M)<1\), then the argument above for~\(\dvgen^{-1} M\)
  gives the second statement by
  Lemma~\ref{lem:spectral_radius_comparison}.
\end{proof}

Now we define certain weak completions of a bornological algebra
through spectral radius estimates.

\begin{definition}
  \label{def:linear_growth}
  Let \((R,\bdd)\)
  be a bornological \(\dvr\)\nb-algebra.
  The \emph{linear growth bornology} on~\(R\)
  is the smallest algebra bornology~\(\ling{\bdd}\)
  on~\(R\)
  that contains~\(\bdd\),
  such that all \(\dvr\)\nb-submodules
  \(M\in\ling{\bdd}\)
  satisfy \(\varrho(M;\ling{\bdd})\le1\).
  A subset of~\(R\)
  has \emph{linear growth} with respect to \(\bdd\)
  if it is in~\(\ling{\bdd}\).
  Let~\(\comling{R}\)
  be the bornological completion of \((R,\ling{\bdd})\).
\end{definition}

\begin{definition}
  \label{def:lambda_completion}
  Let \((R,\bdd)\)
  be a bornological \(\dvr\)\nb-algebra,
  \(J\)
  an ideal in $R$
  and \(\alpha\in[0,1]\).
  Let \(\ul{R} \defeq R\otimes \dvf\) and~\(\gral{\bdd}{J}{\alpha}\)
   the smallest \(\dvf\)\nb-vector
  space bornology on~\(\ul{R}\)
  such that
  \(\varrho(M; \gral{\bdd}{J}{\alpha}) \le \pepsilon^\alpha\)
  whenever \(M\subseteq J\)
  is a \(\dvr\)\nb-submodule
  with \(M\in \bdd\);
  here we
  abusively  denote by $M$ also its image in $\ul{R}$.
  Let \(\comg{R}{J}{\alpha}\)
  be the completion of \((\ul{R},\gral{\bdd}{J}{\alpha})\).

  We wrote \(\varrho(M;\ling{\bdd})\) and \(\varrho(M;
  \gral{\bdd}{J}{\alpha})\) here to emphasize the bornologies for
  which these spectral radii are computed.
\end{definition}

\begin{remark}
  \label{rem:BJ_contains_B}
  If $\dvgen^k\in J$ for some $k\in \N$, then any \(\dvf\)\nb-vector
  space bornology on~$\ul{R}$ that contains the $M\in\bdd$ with
  $M\subseteq J$ must contain all of~$\bdd$.  Thus we have
  $\gral{\bdd}{J}{\alpha}\supseteq\bdd$ in
  Definition~\ref{def:lambda_completion}.
\end{remark}

\begin{remark}
  The condition \(\varrho(M)\le 1\)
  in Definition~\ref{def:linear_growth} is the strongest spectral
  radius constraint that makes sense in bornological
  \(\dvr\)\nb-algebras.
  Even in~\(\ul{R}\),
  asking for \(\varrho(M)< 1\)
  is unreasonable if \(1_R\in M\)
  because then Lemma~\ref{lem:Neumann_series} implies that~\(0\)
  is invertible in the completion, forcing the completion to
  be~\(\{0\}\).
  If \(M\subseteq J\)
  for an ideal~\(J\),
  however, then it makes sense to require
  \(\varrho(M)\le \pepsilon^\alpha\)
  with \(\alpha\ge0\)
  as in Definition~\ref{def:lambda_completion}.  If
  \(\dvgen \in J\),
  \(\varrho(M)\le \pepsilon^\alpha\)
  with \(\alpha>1\)
  would once more imply the invertibility of
  \(0=1-\dvgen^{-1} \dvgen\)
  in the completion, which is unreasonable.  This is why we restrict
  to \(\alpha\le1\)
  and why the ideal \(J=R\)
  is not a good choice unless \(\alpha=0\),
  when we get the bornology on~\(\ul{R}\)
  induced by the linear growth bornology on~\(R\):
  \(\gral{\bdd}{R}{0} = \ul{\ling{\bdd}}\).
\end{remark}



\begin{lemma}
  \label{lem:linear_growth}
  Let~\((R,\bdd)\)
  be a bornological \(\dvr\)-algebra
  and \(T\subseteq R\).  The following are equivalent:
  \begin{enumerate}
  \item \label{en:linear_growth1} \(T\) has linear growth;
  \item \label{en:linear_growth2} \(T\) is contained in
    \(\sum_{i=0}^\infty \dvgen^i S^{i+1}\) for some \(S\in\bdd\);
  \item \label{en:linear_growth3} \(T\) is contained in
    \(\sum_{i=0}^\infty \dvgen^i S^{c i+d}\) for some \(S\in\bdd\),
    \(c,d\in\N\).
  \end{enumerate}
  In particular, $\ling{\bdd}$ is also the smallest
   algebra bornology containing~$\bdd$ such that all
    $\dvr$\nb-submodules $M \in \bdd$ satisfy
    \(\varrho(M;\ling{\bdd})\le1\).
\end{lemma}

\begin{proof}
  The implication
  \ref{en:linear_growth2}\(\Rightarrow\)\ref{en:linear_growth3}
  is trivial, and
  \ref{en:linear_growth3}\(\Rightarrow\)\ref{en:linear_growth1}
  follows from
  \ref{en:spectral_radius_linear_growth3}\(\Rightarrow\)\ref{en:spectral_radius_linear_growth1}
  in Proposition~\ref{pro:spectral_radius_linear_growth}.  The subsets
  of~\(R\)
  as in~\ref{en:linear_growth2} form a bornology~\(\bdd'\)
  because
  \(\sum \dvgen^i S_1^{i+1} + \sum \dvgen^i S_2^{i+1} \subseteq \sum
  \dvgen^i (S_1+S_2)^{i+1}\)
  and \(S_1+S_2\)
  is bounded if \(S_1,S_2\)
  are.  We claim that \(S_1\cdot S_2\)
  is in~\(\bdd'\)
  as well, that is, \(\bdd'\)
  makes~\(R\)
  a bornological algebra.
  There are bounded \(\dvr\)\nb-submodules
  \(U_1,U_2\subseteq R\)
  with \(S_j \subseteq \sum \dvgen^i U_j^{i+1}\)
  for \(j=1,2\).
  A computation as in~\eqref{eq:pi1-multiplicative} gives
  \(S_1 \cdot S_2 \subseteq \sum \dvgen^i (U_1+U_2)^{i+2}\),
  so \(S_1\cdot S_2\in \bdd'\).
  The bornology~\(\bdd'\)
  contains~\(\bdd\).
  Equation~\eqref{eq:pi1-multiplicative} shows that any subset of the
  form \(U\defeq \sum_{i=0}^\infty \dvgen^i S^{i+1}\)
  satisfies \(\dvgen U^2 \subseteq U\).
  Hence \(\varrho(M;\bdd')\le 1\)
  for all \(M\in\bdd'\)
  by the equivalence of \ref{en:spectral_radius_linear_growth1}
  and~\ref{en:spectral_radius_linear_growth4} in
  Proposition~\ref{pro:spectral_radius_linear_growth}.  Since the
  linear growth bornology is the smallest algebra bornology
  containing~\(\bdd\)
  with this property,
  \ref{en:linear_growth1}\(\Rightarrow\)\ref{en:linear_growth2}.
\end{proof}

\begin{lemma}
  \label{lem:alpha_bornology}
  Let~\((R,\bdd)\) be a bornological algebra, \(J\triqui R\) an ideal
  with \(\dvgen^k\in J\) for some $k\in \N$, \(\alpha\in[0,1]\), and \(S\subseteq
  \ul{R}\).  We have \(S\in \gral{\bdd}{J}{\alpha}\) if and only if
  there are \(l\in\N\), \(r>\pepsilon^\alpha\), and a bounded
  \(\dvr\)\nb-submodule \(M\subseteq J\) such that
  \[
  S\subseteq \dvgen^{-l} \sum_{n=1}^\infty r^{-n} \star M^n.
  \]
\end{lemma}

\begin{proof}
  Let~\(\bdd'\)
  be the family of subsets described in the statement of the lemma.
  This is a $\dvf$\nb-vector space bornology on~$\ul{R}$.
  Notice that $\bdd'\supseteq\bdd$ by Remark~\ref{rem:BJ_contains_B}.
  Let~\(\bdd''\)
  be some \(\dvf\)\nb-vector
  space bornology on~\(\ul{R}\)
  and let \(M\subseteq J\)
  be a bounded \(\dvr\)\nb-submodule.
  Then \(\varrho(M;\bdd'')\le \pepsilon^\alpha\)
  if and only if \(\sum r^{-n} \star M^n\in \bdd''\)
  for all \(r>\pepsilon^\alpha\),
  if and only if \(\dvgen^{-l} \sum r^{-n} \star M^n\in \bdd''\)
  for all \(r>\pepsilon^\alpha\),
  \(l\in\N\).
  Thus~\(\bdd'\)
  is the smallest \(\dvf\)\nb-vector
  space bornology on~\(\ul{R}\)
  with \(\varrho(M)\le \pepsilon^\alpha\)
  for all bounded \(\dvr\)\nb-submodules
  \(M\subseteq J\).  That is, \(\bdd' = \gral{\bdd}{J}{\alpha}\).
\end{proof}

\begin{lemma}
  \label{lem:dagger_local_Banach}
  Let~\(R\)
  be a bornological \(\dvr\)\nb-algebra.
  The multiplication on~\(R\)
  is bounded as a map \(\ling{R} \times \ling{R} \to \ling{R}\)
  and extends uniquely to a bornological algebra structure
  on~\(\comling{R}\).
  We have \(\varrho(M)\le 1\)
  for all bounded \(\dvr\)\nb-submodules
  \(M\subseteq\comling{R}\)
  and so~\(\comling{R}\) is a local Banach algebra.
\end{lemma}

\begin{proof}
The multiplication \(\ling{R} \times \ling{R} \to \ling{R}\)
  is bounded by our definition of the linear growth bornology.  Like
  any bounded bilinear map, it extends uniquely to the completions.
  It remains associative and unital on~\(\comling{R}\)
  by the universal property of completions.  So~\(\comling{R}\)
  is a complete bornological \(\dvr\)\nb-algebra.
  Any bounded \(\dvr\)\nb-submodule~\(M\)
  of~\(\comling{R}\)
  is contained in the image of the \(\dvgen\)\nb-adic
  completion~\(\coma{U}\)
  of~\(U\)
  in~\(\comling{R}\)
  for some \(\dvr\)\nb-submodule~\(U\)
  of~\(R\)
  of linear growth.  Since \(\varrho(U;\ling{R})\le1\),
  there is a \(\dvr\)\nb-submodule
  \(U'\subseteq R\)
  of linear growth with \(U\subseteq U'\)
  and \(\dvgen U'\cdot U' \subseteq U'\).
  The image of~\(\coma{U'}\)
  in the completion satisfies
  \(\dvgen \coma{U'}\cdot \coma{U'} \subseteq \coma{U'}\).
  So \(\varrho(M) \le 1\)
  by Proposition~\ref{pro:spectral_radius_linear_growth}.
\end{proof}

\begin{proposition}
  \label{pro:lambda_local_Banach}
  Let~\(R\)
  be a bornological \(\dvr\)\nb-algebra,
  \(J\triqui R\)
  with \(\dvgen^k\in J\) for some $k\in \N$,
  and \(\alpha\in[0,1]\).
  The multiplication on~\(\ul{R}\)
  is bounded for the bornology~\(\gral{\bdd}{J}{\alpha}\),
  and~\(\comg{R}{J}{\alpha}\) is a local Banach
  \(\dvf\)\nb-algebra.
\end{proposition}

\begin{proof}
  Products of subsets of the form \(\dvgen^{-l} \sum_{n=1}^\infty
  r^{-n} \star M^n\) as in Lemma~\ref{lem:alpha_bornology} are again
  contained in a subset of this form
  by~\eqref{eq:star_props}.  Thus the multiplication
  on~\(\ul{R}\) is bounded for the
  bornology~\(\gral{\bdd}{J}{\alpha}\).  Hence it extends
  to~\(\comg{R}{J}{\alpha}\).  Any bounded subset
  of~\(\comg{R}{J}{\alpha}\) is contained in the image of the
  \(\dvgen\)\nb-adic completion of \(U\defeq \dvgen^{-l}
  \sum_{n=1}^\infty r^{-n} \star M^n\) for some \(l\in\N\) and some
  bounded \(\dvr\)\nb-submodule \(M\subseteq J\).  We have \(U\cdot
  U \subseteq \dvgen^{-l} U\).  Hence \(\coma{U}\cdot \coma{U}
    \subseteq \dvgen^{-l} \coma{U}\).  So the gauge seminorm
    associated with~\(\dvgen^l \coma{U}\) is submultiplicative.
  Thus~\(\comg{R}{J}{\alpha}\) is a local Banach
  \(\dvf\)\nb-algebra, see Remark \ref{rem:local_banach}.
\end{proof}

\begin{example}
  \label{exa:alpha-completion_JpiR}
  Let \(J=\dvgen R\)
  and assume that any bounded \(M\subseteq J\)
  is contained in \(\dvgen M_0\)
  for a bounded \(\dvr\)\nb-submodule
  \(M_0\subseteq R\).
  Let \(\alpha\in[0,1]\).
  The condition \(\varrho(M) \le \pepsilon^\alpha\)
  for bounded \(M\subseteq J\)
  is equivalent to \(\varrho(M_0) \le \pepsilon^{\alpha-1}\)
  for bounded \(M_0\subseteq R\)
  by Lemma~\ref{lem:spectral_radius_comparison}.  This condition for
  all~\(M_0\)
  implies \(\varrho(M_0) \le \pepsilon^{(\alpha-1)/c}\)
  for all \(c\in\N_{\ge1}\)
  because \(\varrho(M_0)^c = \varrho(M_0^c)\le \pepsilon^{\alpha-1}\)
  as well.  So \(\varrho(M_0) \le1\).
  Hence \(\gral{\bdd}{\dvgen R}{\alpha} = \ul{\ling{\bdd}}\)
  and
  \(\comg{R}{\dvgen R}{\alpha} = \ul{\comling{R}} =\dvf \otimes
  \comling{R}\).
\end{example}

The following two universal properties are immediate from
  Lemma~\ref{lem:linear_growth} and from the definition of the respective bornologies.

\begin{proposition}
  \label{pro:dagger_universal}
  Let \(R\)
  and~\(S\)
  be bornological \(\dvr\)\nb-algebras.
  Assume that~\(S\)
  is complete and that \(\varrho(M)\le 1\)
  for all bounded \(\dvr\)\nb-submodules~\(M\)
  in~\(S\).
  Any bounded homomorphism from~\(R\)
  to~\(S\)
  extends uniquely to a bounded homomorphism \(\comling{R} \to S\).
  If \(\varrho(M)\le1\)
  for all bounded \(\dvr\)\nb-submodules~\(M\)
  in~\(R\), then \(R=\ling{R}\).
\end{proposition}

\begin{proposition}
  \label{pro:lambda_universal}
  Let \(R\) be a bornological \(\dvr\)\nb-algebra, \(J\triqui R\)
  with \(\dvgen^k\in J\) for some $k\in \N$, and \(\alpha\in[0,1]\).  Let~\(S\) be a
  complete bornological \(\dvf\)\nb-algebra.  A bounded
  \(\dvr\)\nb-algebra homomorphism \(\varphi\colon R\to S\) extends
  to a bounded \(\dvf\)\nb-algebra homomorphism
  \(\comg{R}{J}{\alpha}\to S\) if and only if
  \(\varrho(\varphi(M)) \le \pepsilon^\alpha\) for all bounded
  \(\dvr\)\nb-submodules \(M\subseteq J\).
\end{proposition}

\begin{proposition}
  \label{pro:dagger_functorial}
  A bounded unital algebra homomorphism \(\varphi\colon R\to S\)
  induces a bounded unital algebra homomorphism
  \(\comling{\varphi}\colon \comling{R} \to \comling{S}\).
  If~\(\varphi\)
  is a bornological quotient map, then so is~\(\comling{\varphi}\).
\end{proposition}

\begin{proof}
  The first assertion follows from
    Proposition~\ref{pro:dagger_universal}.
  If~\(\varphi\)
  is a bornological quotient map, then so is
  \(\ling{\varphi}\colon \ling{R} \to \ling{S}\)
  by the concrete description of the linear growth bornology in
  Lemma~\ref{lem:linear_growth}.  Hence so is~\(\comling{\varphi}\)
  because taking completions preserves bornological quotient maps by
  Lemma~\ref{lem:completion_quotient}.
\end{proof}

\begin{proposition}
  \label{pro:lambda_functorial}
  Let \(R\)
  and~\(S\)
  be bornological \(\dvr\)\nb-algebras.
  Let \(I\triqui R\)
  and \(J\triqui S\)
  be ideals containing~\(\dvgen^k\) for some $k\in \N$.
  Let \(\alpha,\beta\in[0,1]\)
  satisfy \(\alpha\le\beta\).
  A bounded unital algebra homomorphism \(\varphi\colon R\to S\)
  with \(\varphi(I)\subseteq J\)
  extends uniquely to a bounded unital algebra homomorphism
  \(\comg{R}{I}{\alpha} \to \comg{S}{J}{\beta}\).
 This extension is a bornological quotient map if
    \(\varphi|_I\colon I \to J\)
  is a bornological quotient map and \(\alpha=\beta\).
\end{proposition}

\begin{proof}
  If \(M\subseteq I\)
  is bounded, then \(\varphi(M)\)
  is bounded and contained in~\(J\).
  Thus \(\varrho(\varphi(M))\le \pepsilon^\beta \le \pepsilon^\alpha\)
  in \(\comg{S}{J}{\beta}\).
  This verifies the criterion in
  Proposition~\ref{pro:lambda_universal} for~\(\varphi\)
  to extend uniquely to~\(\comg{R}{I}{\alpha}\).
  If~\(\varphi|_I\)
  is a bornological quotient map, then for any bounded
  \(\dvr\)\nb-submodule
  \(M\subseteq J\)
  there is a bounded \(\dvr\)\nb-submodule
  \(N\subseteq I\)
  with \(\varphi(N)=M\).
  Hence the map \(\ul{R}\to\ul{S}\)
  induced by~\(\varphi\)
  maps \(\dvgen^{-l}\sum_{n=1}^\infty r^{-n} \star N^n\)
  onto \(\dvgen^{-l}\sum_{n=1}^\infty r^{-n} \star M^n\).
  Thus it is a bornological quotient map from
  \((\ul{R},\gral{\bdd}{I}{\alpha})\)
  to \((\ul{S},\gral{\bdd}{J}{\alpha})\) by
  Lemma~\ref{lem:alpha_bornology}.
  By Lemma~\ref{lem:completion_quotient},
  being a bornological quotient map is preserved by the
  bornological completion.
\end{proof}

Next we are going to rewrite the completions \(\comg{R}{J}{\alpha}\)
using linear growth completions of tube algebras.

\begin{definition}
  \label{def:alpha-tube}
  Let \(J\) be an ideal in $R$ and let
  \(\alpha\in[0,1]\).  The \emph{\(\alpha\)\nb-tube algebra} of~\(R\)
  around~\(J\) is
  \begin{equation}
    \label{eq:alpha-tube}
    \tub{R}{J}{\alpha} \defeq
    \sum_{n=0}^\infty {}\pepsilon^{-\alpha n} \star J^n\subseteq \ul{R},
  \end{equation}
  where the \(0\)th summand is \(J^0 \defeq R\).  This is a
  \(\dvr\)\nb-algebra.  We equip it with the bornology generated by
  the bounded submodules of \(J^0=R\) and by \(\dvr\)\nb-submodules
  of the form~\(\pepsilon^{-\alpha n} \star M^n\) for \(n\in\N_{\ge1}\)
  and a bounded \(\dvr\)\nb-submodule \(M\subseteq J\).
\end{definition}

Almost by definition, a subset of~\(\tub{R}{J}{\alpha}\) is bounded
if and only if it is contained in \(M_0 + \sum_{n=1}^N
\pepsilon^{-\alpha n} \star M^n\) for some \(N\in\N\) and some
bounded \(\dvr\)\nb-submodules \(M_0\subseteq R\), \(M\subseteq J\).
The multiplication on~\(\tub{R}{J}{\alpha}\) is bounded
by~\eqref{eq:star_props}.  So it is a bornological
\(\dvr\)\nb-algebra.  The inclusion map \(\tub{R}{J}{\alpha} \to
\ul{R}\) extends to an isomorphism \(\tub{R}{J}{\alpha} \otimes \dvf
\cong \ul{R}\) of bornological \(\dvf\)\nb-algebras.

\begin{lemma}
  \label{lem:tube_fine}
  If~\(R\) carries the fine bornology, then so
  does~\(\tub{R}{J}{\alpha}\).
\end{lemma}

\begin{proof}
  If $M_0$ and~$M$ are finitely generated, then so is
  \(M_0 + \sum_{n=1}^N \pepsilon^{-\alpha n} \star M^n\).
\end{proof}

\begin{proposition}
  \label{pro:linear_growth_on_tube}
  Let \(J\triqui R\) be an ideal with \(\dvgen \in J\) and let
  \(\alpha\in[0,1]\).  There is an isomorphism of bornological
  \(\dvf\)\nb-algebras
  \[
  \ling{\tub{R}{J}{\alpha}} \otimes \dvf
  \cong (\ul{R},\gral{\bdd}{J}{\alpha}).
  \]
  It induces an isomorphism of the completions
  \(\comling{\tub{R}{J}{\alpha}} \otimes \dvf \cong
  \comg{R}{J}{\alpha}\).
\end{proposition}

\begin{proof}
  By Lemma \ref{lem:linear_growth}, the bornology on \(\ling{\tub{R}{J}{\alpha}} \otimes \dvf\)
  is the smallest \(\dvf\)\nb-vector space bornology where all
  bounded subsets of \(\tub{R}{J}{\alpha}\) have spectral radius at
  most~\(1\).  Equivalently, if \(r<1\), \(N\in\N\) and
  \(M_0\subseteq R\), \(M\subseteq J\) are bounded
  \(\dvr\)\nb-submodules, then
  \[
  S_{r,N,M_0,M} \defeq
  \sum_{n=1}^\infty r^n \star \biggl(M_0 + \sum_{i=1}^N
  \pepsilon^{-\alpha i} \star M^i\biggr)^n
  \]
  is bounded.  This implies that
  \(\pepsilon^{-\lfloor \alpha n\rfloor} \varrho(M)^n =
  \varrho(\pepsilon^{-\alpha n} \star M^n) \le 1\)
  if \(n\in\N_{\ge1}\)
  and \(M\subseteq J\)
  is bounded.  This is equivalent to
  \(\varrho(M)\le \pepsilon^\alpha\)
  for all bounded \(M\subseteq J\)
  by Lemma~\ref{lem:spectral_radius_comparison}.  Therefore, all
  subsets in~\(\gral{\bdd}{J}{\alpha}\)
  are bounded in \(\ling{\tub{R}{J}{\alpha}} \otimes \dvf\).

  It remains to prove that for $r,N,M_0$ and $M$ as above,
  $S\defeq S_{r,N,M_0,M}\in \gral{\bdd}{J}{\alpha}$.  We may and do
  assume that $M_0\owns 1$.  Let \(\ell\in\N_{\ge1}\).
  Consider the following $\dvr$-submodule of $R$
  \[
  M'\defeq M_0^{\ell-1}MM_0^{\ell-1}+\dvgen M_0^\ell
  \]
  Since \(\dvgen\in J\)
  and \(M\subseteq J\)
  and~\(J\)
  is an ideal and $1\in M_0$, $M'$ is a bounded \(\dvr\)\nb-submodule
  in~\(J\)
  which contains \(M_0^iM M_0^j\)
  for all $0\le i,j\le \ell-1$.  Hence
  \(\varrho(M') \le \pepsilon^\alpha\)
  in the bornology~\(\gral{\bdd}{J}{\alpha}\).  Thus
  \[
  T_{s,M'}\defeq\sum_{n=1}^\infty s^n \pepsilon^{-\alpha n}  \star (M')^n
  \]
  is in \(\gral{\bdd}{J}{\alpha}\)
  for all \(s<1\).
  We shall show that $S\subseteq  T_{s,M'}$ if
  \begin{equation}
    \label{condi_s}
    r^{\ell-1}<\pepsilon^{1-\alpha},\quad 1>s\ge\max\{r^{\ell-1}/\pepsilon^{1-\alpha}, r^{1/N}\}.
  \end{equation}
  We may choose~\(\ell\) so that this holds.

  An element of~\(S\)
  is a \(\dvr\)\nb-linear
  combination of products \(\dvgen^h x_1 \dotsm x_n\)
  with \(x_1,\dotsc,x_n\in M_0 \cup \bigcup_{i=1}^N M^i\);
  here~\(h\)
  is the sum of \(\lceil n\cdot \log(r)/\log \pepsilon\rceil\)
  and one summand \(- \lfloor \alpha i\rfloor\)
  for each factor \(x_j\in M^i\).
  In most cases, we may group these factors~\(x_i\)
  together so that \(x_1\dotsm x_n = x'_1\dotsm x'_{n'}\),
  where each~\(x'_j\)
  consists either of~\(\ell\)
  factors in~\(M_0\)
  or of one factor in~\(M^i\)
  surrounded by at most \(\ell-1\)
  factors in~\(M_0\)
  on each side.  The only exception is the case where all
  factors~\(x_i\)
  are from~\(M_0\);
  then we allow~\(x'_1\)
  to be a product of less than~\(\ell\)
  factors in~\(M_0\).
  By construction, each factor~\(x'_j\)
  belongs to \(M_0^\ell\subseteq \dvgen^{-1} M'\)
  or to \((M')^i\).
  For factors~\(x'_j\)
  in \(M_0^\ell\subseteq \dvgen^{-1} M'\),
  we put~\(\dvgen^{-1}\)
  into the scalar factor~\(\dvgen^h\);
  and we split factors~\(x'_j\)
  in \((M')^i\)
  into \(i\)~factors
  in~\(M'\).
  This gives
  \(\dvgen^h x_1 \dotsm x_n = \dvgen^{h-b} x''_1\dotsm x''_{n''}\)
  with \(x''_j\in M'\)
  for all~\(j\), where~\(b\) is the number of factors~\(x'_j\) in~\(M_0^\ell\).
  We claim that
  \(\abs{\dvgen^{h-b}} = \pepsilon^{h-b} \le s^{n''}
  \pepsilon^{-\alpha n''}\),
  so that
  \(\dvgen^h x_1 \dotsm x_n \in  \sum_{m=1}^\infty s^m
  \pepsilon^{-\alpha m} \star (M')^m\) as desired.

  We have \(n''-b=\sum_{x_j\in M^i} i\).  Hence
  \begin{multline*}
    h-b
    = \lceil n\cdot \log(r)/\log \pepsilon\rceil
    - \sum_{x_j\in M^i} \lfloor \alpha i \rfloor - b
    \\\ge n\cdot \log(r)/\log \pepsilon - b - \alpha
    \sum_{x_j\in M^i} i
    = n\cdot \log(r)/\log \pepsilon - b - \alpha (n'' - b).
  \end{multline*}
  Each of the \(b\)~factors
  \(x'_j \in M_0^\ell\)
  comes from~\(\ell\)
  consecutive factors \(x_i\in M_0\),
  and each factor \(x'_j\in (M')^i\)
  contains at least one factor \(x_m\in M^i\).
  Hence \(n' - b \le n - \ell b\).
  The \(n'-b\)
  factors \(x'_j\in (M')^i\)
  each produce \(i\le N\)
  factors in~\(M'\).
  Therefore, \(n''-b \le N\cdot (n'-b) \le N\cdot (n-\ell b)\)
  or, equivalently, \(n \ge (\ell-N^{-1})\cdot b + n''/N\).
  Using this and \(s\ge r^{1/N}\), we estimate
  \begin{align*}
    h-b &\ge
    (\ell-N^{-1}) b \log(r)/\log \pepsilon
    + n'' \log(r^{1/N})/\log \pepsilon
    - b(1-\alpha) - \alpha n''
    \\ &\ge - \alpha n'' + n'' \log(s)/\log \pepsilon
    + (\ell-N^{-1}) b \log(r)/\log \pepsilon
    - b(1-\alpha).
  \end{align*}
  The first two summands are exactly what we need for our estimate.
  And the sum of the other two is non-negative because
  \(r^{\ell - N^{-1}} {}<{} \pepsilon^{1-\alpha}\)
  by~\eqref{condi_s}.  The exceptional summand where all factors are
  in~\(M_0\)
  belongs to
  \(r^n M_0^n \subseteq r^n \dvgen^{-\lceil n/\ell\rceil} \star
  (M')^{\lceil n/\ell\rceil}\).
  Since \(n \ge \ell \lceil n/\ell\rceil - 1\),
  this is contained in
  \(r^{(\ell-1) m} \dvgen^{-m} \star (M')^m \subseteq
  \pepsilon^{-\alpha m} \star (M')^m\) by~\eqref{condi_s} with $m = \lceil n/\ell\rceil$.
\end{proof}

\begin{example}
  \label{exa:alpha_tube}
  Let \(\alpha=1\).
  Then \(\tub{R}{J}{1} = \sum_{n=0}^\infty \dvgen^{-n} J^n\).
  If \(\dvgen \in J\),
  we have
  \(R\subseteq \dvgen^{-1} J \subseteq \dvgen^{-2} J^2 \subseteq
  \dotsb\).
  So \(\tub{R}{J}{1}\)
  is the union of the increasing chain of \(\dvr\)\nb-submodules
  \(\dvgen^{-n} J^n \subseteq \ul{R}\).
  A subset of \(\tub{R}{J}{1}\)
  is bounded if and only if it is contained in \(\dvgen^{-n} M^n\)
  for some bounded \(\dvr\)\nb-submodule
  \(M\subseteq J\).
  In this case, the proof of
  Proposition~\ref{pro:linear_growth_on_tube} is much easier.

  The case \(\alpha=1/m\) for some \(m\in\N_{\ge1}\) is also
  somewhat easier, at least under a mild hypothesis.  We have the
  following identity of rings:
  \[
  \tub{R}{J}{1/m} = \sum \pepsilon^{n/m} \star J^n
  = \sum \dvgen^{-n} J^{m n} = \tub{R}{J^m}{1}.
  \]
  The two bornologies are the same as well if we assume that every
  bounded subset of~$J^m$ is contained in~$M^m$ for some bounded
  submodule $M\subseteq J$; in other words, if the map $J^{\otimes m}\to
  J^m$ is a bornological quotient map for all~$m$.  This hypothesis is
  satisfied, for example, when~$R$ carries the fine bornology.
  Let \(\alpha>0\) be rational, \(\alpha=i/m\) with
  \(i,m\in\N_{\ge1}\), \(\gcd(i,m)=1\).  Then
  \begin{multline*}
    \tub{R}{J}{i/m}
    = \sum_{\ell=0}^\infty \dvgen^{-\ell} J^{\lceil \ell m/i\rceil}
    \\= (R + \dvgen^{-1} J^{\lceil m/i\rceil} + \dvgen^{-2}
    J^{\lceil 2 m/i\rceil} + \dotsb + \dvgen^{-(i-1)} J^{\lceil (i-1) m/i\rceil})\cdot
    \sum_{\ell=0}^\infty (\dvgen^{-i} J^m)^\ell
    \\= \sum_{\ell=0}^\infty (R + \dvgen^{-1} J^{\lceil m/i\rceil} + \dvgen^{-2}
    J^{\lceil 2 m/i\rceil} + \dotsb + \dvgen^{-i} J^{\lceil i m/i\rceil})^\ell.
  \end{multline*}
  If \(\alpha=0\), then \(\tub{R}{J}{0}=R\).
\end{example}

\begin{remark}
  \label{rem:tube_fgc}
  Let~\(R\)
  be commutative and finitely generated and equipped with the fine
  bornology.  Let \(\dvgen\in J\triqui R\)
  and \(\alpha\in[0,1]\).
  Then~\(\tub{R}{J}{\alpha}\)
  also carries the fine bornology by Lemma~\ref{lem:tube_fine}.  The
  tube algebra~\(\tub{R}{J}{1}\)
  is described in Example~\ref{exa:alpha_tube}.
  It is closely related to the rings of functions on tubes used in the construction of rigid cohomology, see Lemma~\ref{lem:covering-of-tube}.
  Example~\ref{exa:alpha_tube} also identifies
  \(\tub{R}{J}{1/m} \cong \tub{R}{J^m}{1}\),
  which is the usual tube algebra of~\(R\) with respect to~\(J^m\).
\end{remark}

\begin{lemma}
  \label{lem:tube_fg}
  Let~\(R\) be commutative and finitely generated and equipped with
  the fine bornology.  Let \(\dvgen\in J\triqui R\) and
  \(\alpha\in[0,1]\).  If \(\alpha\in\Q\), then the tube algebra
  \(\tub{R}{J}{\alpha}\) is finitely generated and in particular Noetherian.
\end{lemma}

\begin{proof}
  If~\(\alpha\)
  is rational, \(\alpha=m/i\),
  then Example~\ref{exa:alpha_tube} shows that~\(\tub{R}{J}{\alpha}\)
  is generated as an algebra by
  \(R + \dvgen^{-1} J^{\lceil m/i\rceil} + \dvgen^{-2} J^{\lceil 2
    m/i\rceil} + \dotsb + \dvgen^{-i} J^{\lceil i m/i\rceil}\).
  Since~\(R\)
  is Noetherian, there are finite generating sets \(S_j\subseteq J^j\)
  for the ideals \(J^j\triqui R\)
  for \(1\le j\le m\),
  and~\(R\)
  is generated as an algebra by a finite set of generators~\(S_0\).
  Then the finite subset
  \(\bigcup_{l=0}^{i-1} \dvgen^{-l} S_{\lceil l m/i\rceil}\)
  generates~\(\tub{R}{J}{\alpha}\) as an algebra.
\end{proof}

In contrast, we claim that \(\tub{R}{J}{\alpha}\) for
irrational~\(\alpha\) is usually not finitely generated, except in
trivial cases.  If \(\alpha'\le \alpha\), then \(\tub{R}{J}{\alpha'}
\subseteq \tub{R}{J}{\alpha}\).  Any element of \(\tub{R}{J}{\alpha}\)
belongs to \(\tub{R}{J}{\alpha'}\) for some \(\alpha'\in \Q\) with
\(\alpha'\le \alpha\).  Thus \(\tub{R}{J}{\alpha}\) is the increasing
union of its subalgebras \(\tub{R}{J}{\alpha'}\) for \(\alpha'\in \Q\)
with \(\alpha'\le \alpha\).  However, except in trivial cases such as
\(J=\dvgen R\) or if this holds up to torsion, we
have \(\tub{R}{J}{\alpha'} \neq \tub{R}{J}{\alpha}\) if
\(\alpha'<\alpha\).

\begin{proposition}
  \label{pro:tensor_dagger}
  Let \(R\) and~\(S\) be bornological algebras.  Then
  \(\ling{(R\otimes S)} = \ling{R} \otimes \ling{S}\)
  and hence \(\comling{(R\otimes S)} \cong
  \comling{R} \hotimes \comling{S}\).
\end{proposition}

\begin{proof}
  A \(\dvr\)\nb-submodule of \(R\otimes S\) is bounded if and only
  if it is contained in the image of \(T\otimes U\) for bounded
  \(\dvr\)\nb-submodules \(T\subseteq R\), \(U\subseteq S\); we may
  assume \(1_R\in T\), \(1_S\in U\).  Then
  \begin{multline*}
    \sum \dvgen^{2 i} (T\otimes U)^{i+1}
    \subseteq \Bigl(\sum \dvgen^i T^{i+1}\Bigr)
    \otimes \Bigl(\sum \dvgen^j U^{j+1}\Bigr)
    \\\subseteq \sum \dvgen^{i+j} T^{\max\{i,j\}+1}\otimes U^{\max\{i,j\}+1}
    \subseteq \sum \dvgen^m (T\otimes U)^{m+1}.
  \end{multline*}
  By Lemma~\ref{lem:linear_growth}, this shows that a subset is
  bounded in \(\ling{R} \otimes \ling{S}\)
  if and only if it is bounded in \(\ling{(R\otimes S)}\).
  The isomorphism
  \(\comling{(R\otimes S)} \cong \comling{R} \hotimes \comling{S}\)
  follows from
  \(\ling{(R\otimes S)} = \ling{R} \otimes \ling{S}\)
  by taking completions on both sides, see
  Lemma~\ref{lem:complete_tensor}.
\end{proof}

\begin{proposition}
  \label{pro:lambda_tensor_dagger}
  Let \(R\)
  and~\(S\)
  be commutative bornological $\dvr$-algebras.  Let \(I\triqui R\)
  and \(J\triqui S\)
  be ideals containing~\(\dvgen\).
  Let \(\alpha\in[0,1]\).
  Let \(I+J\)
  denote
  the sum of the images of  \(I\otimes S\) and  \(R\otimes J\)
  in~\(R\otimes S\);
  this is an ideal.  Assume that any bounded \(\dvr\)\nb-submodule
  of \(I+J\)
  is the sum of the images of \(M_I\otimes M_S\) and of \(M_R\otimes M_J\)
  for bounded \(\dvr\)\nb-submodules
  \(M_I \subseteq I\),
  \(M_S \subseteq S\),
  \(M_R \subseteq R\), and \(M_J \subseteq J\).  Then
  \[
  \comg{R}{I}{\alpha} \hotimes
  \comg{S}{J}{\alpha} \cong
  \comg{R\otimes S}{I+J}{\alpha}.
  \]
\end{proposition}

\begin{proof}
  We show that
  \(\comg{R}{I}{\alpha} \hotimes
  \comg{S}{J}{\alpha}\)
  has the universal property of
  \(T\defeq \comg{R\otimes S}{I+J}{\alpha}\)
  in Proposition~\ref{pro:lambda_universal}.  So let~\(A\)
  be a complete bornological algebra.  Bounded unital homomorphisms
  \(T\to A\)
  are in natural bijection with bounded unital homomorphisms
  \(\varphi\colon R\otimes S\to A\)
  such that \(\varrho(\varphi(M)) \le \pepsilon^\alpha\)
  for all bounded \(\dvr\)\nb-submodules
  \(M\subseteq I+J\).
  The bounded unital homomorphism~\(\varphi\)
  corresponds to a pair of bounded unital homomorphisms
  \(\varphi^R\colon R\to A\)
  and \(\varphi^S\colon S\to A\)
  with commuting images through
  \(\varphi(r\otimes s) = \varphi^R(r)\cdot \varphi^S(s)\)
  for all \(r\in R\),
  \(s\in S\).
  By our assumption, any bounded \(\dvr\)\nb-submodule
  \(M\subseteq I+J\)
  is contained in the sum of the images of \(M_I\otimes M_S\) and \(M_R\otimes M_J\)
  for bounded \(\dvr\)\nb-submodules
  \(M_I \subseteq I\),
  \(M_S \subseteq S\),
  \(M_R \subseteq R\), and \(M_J \subseteq J\).
  Conversely, the sum of the images of \(M_I\otimes M_S\) and \(M_R\otimes M_J\) as above
  is clearly a bounded \(\dvr\)\nb-submodule
  of~\(I+J\) for the tensor product bornology.
  Hence \(\varrho(\varphi(M)) \le \pepsilon^\alpha\)
  for all~\(M\)
  as above if and only if the same spectral radius estimate holds for
  \(\varphi^R(M_I)\cdot \varphi^S(M_S)+ \varphi^R(M_R)\cdot
  \varphi^S(M_J)\)
  for all bounded \(\dvr\)\nb-submodules
  \(M_I \subseteq I\),
  \(M_S \subseteq S\),
  \(M_R \subseteq R\),
  and \(M_J \subseteq J\).
  Since \(\varphi^R\) and~\(\varphi^S\) have commuting images,
  \[
  \bigl(\varphi^R(M_I)\cdot \varphi^S(M_S)
  + \varphi^R(M_R)\cdot \varphi^S(M_J)\bigr)^n
  = \sum_{l=0}^n \varphi^R(M_I^l M_R^{n-l}) \cdot \varphi^S(M_J^{n-l} M_S^l),
  \]
  where \(M_I^0\defeq M_R^0 \defeq \dvr\cdot 1_R\)
  and \(M_J^0\defeq M_S^0\defeq \dvr\cdot 1_S\).  Let \(\beta,\gamma>0\), then
  \begin{multline*}
    \sum_{n=0}^\infty
    (\beta\gamma)^n\star \bigl(\varphi^R(M_I)\cdot \varphi^S(M_S)
    + \varphi^R(M_R)\cdot \varphi^S(M_J)\bigr)^n
    \\=
    \sum_{n,m=0}^\infty
    (\beta\gamma)^{n+m}\star \varphi^R(M_I^n) \cdot \varphi^S(M_J^m)\cdot
    \varphi^R(M_R^m) \cdot \varphi^S(M_S^n).
  \end{multline*}
  The modules
  \((\beta\gamma)^{n+m}\star
  \varphi^R(M_I^n) \cdot \varphi^S(M_J^m)\cdot
  \varphi^R(M_R^m) \cdot \varphi^S(M_S^n)\)
  and
  \((\beta^n\star \varphi^R(M_I^n)) \cdot
  (\gamma^m \star \varphi^R(M_R^m))
  \cdot (\beta^m\star\varphi^S(M_J^m))
  \cdot (\gamma^n\star \varphi^S(M_S^n))\)
  differ at most by four factors of~\(\dvgen\)
  due to rounding errors when replacing \(\beta^n\),
  \(\beta^m\),
  \(\gamma^n\), \(\gamma^m\) and \((\beta\gamma)^{n+m}\)
  by powers of~\(\pepsilon\), see~\eqref{eq:star_props}.
  For the boundedness of the sum, this is irrelevant.  We may assume
  that \(1\in M_R\)
  and \(1\in M_S\).
  Then \(1\in\varphi^R(M_R^n)\)
  and \(1\in\varphi^S(M_S^n)\)
  for all \(n\in\N\).
  Therefore, if the sum above is bounded for all
  \(\beta<\pepsilon^\alpha\) and \(\gamma<1\),
  then both \(\sum_{n=0}^\infty \beta^n\star \varphi^R(M_I^n)\)
  and \(\sum_{m=0}^\infty \beta^m\star \varphi^S(M_J^m)\)
  are bounded for all such~\(\beta\).
  Conversely, assume that the latter are bounded for all such
  \(\beta\),
  \(M_I\)
  and~\(M_J\).
  Here we may also take \(M_I = \dvgen M_R^m\)
  and \(M_J = \dvgen M_S^m\)
  for any \(m\in\N\).
  By Lemma~\ref{lem:spectral_radius_comparison}, this implies that
  \(\sum_{n=0}^\infty \gamma^n\star \varphi^R(M_R^n)\)
  and \(\sum_{m=0}^\infty \gamma^m\star \varphi^S(M_S^m)\)
  are bounded for any \(\gamma<1\).  Thus
  \[
  \sum_{n,m=0}^\infty
  (\beta^n\star \varphi^R(M_I^n)) \cdot
  (\gamma^m \star \varphi^R(M_R^m))
  \cdot (\beta^m\star\varphi^S(M_J^m))
  \cdot (\gamma^n\star \varphi^S(M_S^n))
  \]
  is bounded if \(\beta<\pepsilon^\alpha\) and \(\gamma<1\).
  By the argument above, this is equivalent to
  \[
  \sum_{n=0}^\infty
  (\beta \gamma)^n\star \bigl(\varphi^R(M_I)\cdot \varphi^S(M_S)
  + \varphi^R(M_R)\cdot \varphi^S(M_J)\bigr)^n
  \]
  being bounded.  Since \(\gamma<1\)
  is arbitrary, we see that~\(\varphi\)
  satisfies the condition in Proposition~\ref{pro:lambda_universal}
  that characterizes when it extends to a homomorphism on~\(T\)
  if and only if both \(\varphi^R\)
  and~\(\varphi^S\)
  satisfy the corresponding condition to extend to
  \(\comg{R}{I}{\alpha}\)
  and \(\comg{S}{J}{\alpha}\),
  respectively.  These unique extensions still commute when they
  exist.  Thus the pairs~\((\varphi^R,\varphi^S)\)
  as above are in natural bijection with bounded unital homomorphisms
  \(\comg{R}{I}{\alpha} \hotimes \comg{S}{J}{\alpha}\to A\).
  Putting things together, we have got a natural bijection between
  bounded unital homomorphisms
  \(\comg{R}{I}{\alpha} \hotimes \comg{S}{J}{\alpha}\to A\)
  and \(\comg{R\otimes S}{I+J}{\alpha} \to A\).
\end{proof}

\begin{corollary}
  \label{cor:lambda_tensor_dagger}
  Let \(R\)
  and~\(S\)
  be commutative bornological algebras.  Let \(I\triqui R\)
  be an ideal containing~\(\dvgen\).
  Let \(\alpha\in[0,1]\).
  Let~\(I'\)
  denote the image of \(I\otimes S\)
  in~\(R\otimes S\);
  this is an ideal.  Assume that any bounded \(\dvr\)\nb-submodule
  of \(I'\)
  is the image of \(M_I\otimes M_S\)
  for bounded \(\dvr\)\nb-submodules
  \(M_I \subseteq I\) and \(M_S \subseteq S\).  Then
  \[
  \comg{R}{I}{\alpha} \hotimes
  \comling{S} \cong
  \comg{R\otimes S}{I'}{\alpha}.
  \]
\end{corollary}

\begin{proof}
  Apply Proposition~\ref{pro:lambda_tensor_dagger} in the special case
  \(J=\dvgen\cdot S\)
  and use
  \(\comg{S}{\dvgen S}{\alpha} =
  \comling{S}\)
  for all \(\alpha\in[0,1]\)
  by Example~\ref{exa:alpha-completion_JpiR}.
\end{proof}

We shall only apply Propositions \ref{pro:tensor_dagger}
and~\ref{pro:lambda_tensor_dagger} and
Corollary~\ref{cor:lambda_tensor_dagger} when \(R\)
and~\(S\)
carry the fine bornology.  Then the technical assumptions in
Proposition~\ref{pro:lambda_tensor_dagger} and
Corollary~\ref{cor:lambda_tensor_dagger} about bounded
\(\dvr\)\nb-submodules of \(I+J\) or~\(I'\) hold automatically.

\begin{remark}
  \label{rem:lambda-completion_flat}
  If~\(R\) is not flat over~\(\dvr\), then the canonical map
  \(R\to\ul{R}\) is not injective: its kernel is the
  \(\dvr\)\nb-torsion submodule~\(\tau R\).  Let \(R'\defeq R/\tau R\).
  Then \(\ul{R'}=\ul{R}\).  If \(J\triqui R\) is an ideal, let
  \(J'\triqui R'\) be the image in~\(R'\).  Then
  \(\gral{\bdd}{J}{\alpha} = \gral{\bdd}{J'}{\alpha}\).  So it
  suffices to study the bornology~\(\gral{\bdd}{J}{\alpha}\) and the
  resulting completion of~\(\ul{R}\) if~\(R\) is flat over~\(\dvr\).
\end{remark}

\subsection{Dagger completions and linear growth bornologies}
\label{subsec:dag}

Let~\(R\) be a finitely generated, commutative \(\dvr\)\nb-algebra.
Monsky--Washnitzer~\cite{mw} define the \emph{weak
  completion}~\(R^\updagger\) of~\(R\) as the subset of the
\(\dvgen\)\nb-adic completion
\[
\coma{R} = \varprojlim\limits_{\scriptstyle S} R/\dvgen^j R
\]
consisting of elements~\(z\) having representations
\begin{equation}
  \label{MW}
  z = \sum_{j=0}^\infty \dvgen^j w_j
\end{equation}
with \(w_j \in M^{\kappa_j}\),
where \(M\owns 1\)
is a finitely generated \(\dvr\)\nb-submodule
of~\(R\)
and \(\kappa_j \le c(j+1)\)
for some constant \(c>0\)
depending on~\(z\).
We equip~\(R^\updagger\)
with the following bornology: call a subset~\(S\)
bounded if all its elements are of the form
\(z = \sum_{j=0}^\infty \dvgen^j w_j\)
as in~\eqref{MW} with \(w_j \in M^{\kappa_j}\)
and \(\kappa_j \le c(j+1)\)
for a fixed finitely generated \(\dvr\)\nb-submodule
\(M\subseteq R\) and a fixed \(c>0\).

\begin{theorem}
  \label{the:MW_completion}
  Let~\(R\) be a finitely generated, commutative
  \(\dvr\)\nb-algebra, equipped with the fine bornology.  The
  canonical map \(R\to R^\updagger\) extends uniquely to an
  isomorphism of bornological algebras
  from~\(\comling{R}\) onto~\(R^\updagger\).
\end{theorem}

The proof of Theorem~\ref{the:MW_completion} will be finished by
Proposition~\ref{pro:dagger_for_quotient}.  We first treat the case
of the full polynomial algebra and then reduce the general case to
it.

\begin{lemma}
  \label{lem:polynomial_dagger_injective}
  Let \(P = \dvr[x_1,\dotsc,x_n]\).
  Then the canonical map \(\comling{P} \to \coma{P}\)
  is an isomorphism of bornological algebras onto the Monsky--Washnitzer
  completion~\(P^\updagger\) of~\(P\).
\end{lemma}

\begin{proof}
  Let \(S_0 = \dvr 1 + \dvr x_1 + \dotsb + \dvr x_n\) be the obvious
  generating \(\dvr\)\nb-submodule for~\(P\) containing~\(1\).  Any
  finitely generated \(\dvr\)\nb-submodule of~\(P\) is contained
  in~\(S_0^c\) for some \(c\in\N_{\ge1}\).  By
  Lemma~\ref{lem:linear_growth}, a subset of~\(P\) has linear growth
  if and only if it is contained in \(P_c \defeq \sum_{j=0}^\infty
  \dvgen^j S_0^{c(j+1)}\) for some \(c\in\N_{\ge1}\).  This is the
  set of polynomials \(\sum b_\alpha x^\alpha\), \(b_\alpha\in
  \dvr\), with \(\nu(b_\alpha)+1 \ge \abs{\alpha}/c\).

  Since~\(\dvr\) is \(\dvgen\)\nb-adically complete, the
  \(\dvgen\)\nb-adic completion~\(\coma{P_c}\) of~\(P_c\) is
  the set of all formal power series \(\sum b_\alpha x^\alpha\) with
  \(b_\alpha\in \dvr\), \(\nu(b_\alpha)+1 \ge \abs{\alpha}/c\) for
  all~\(\alpha\) and \(\lim_{\abs{\alpha}\to\infty} \nu(b_\alpha)+1
  - \abs{\alpha}/c = +\infty\).  In particular, $\coma{P_c}$ is
  contained in the $\dvgen$-adic completion~$\widehat{P}$.  Since
  \(\ling{P} = \varinjlim P_c\),
  Proposition~\ref{prop:completion} implies
  \(\comling{P} = \varinjlim \coma{P_c}\).  This is
  contained in~\(\coma{P}\).

  If $c'>c$ then $1/c'-1/c<0$.  So if the series $f=\sum b_\alpha x^\alpha$ satisfies \(\nu(b_\alpha)+1 \ge \abs{\alpha}/c\), then $f\in \coma{P_{c'}}$ because
  \[
  0\le \nu(b_\alpha)+1-\abs{\alpha}/c=\nu(b_\alpha)+1-\abs{\alpha}/c'+\abs{\alpha}(1/c'-1/c).
  \]
	
  Thus \(\varinjlim \coma{P_c}
  \subseteq \coma{P}\) is equal to the set of all formal power
  series \(\sum b_\alpha x^\alpha \in \coma{P}\) such that
  \begin{equation}
    \label{converpc}
    (\exists\delta>0) (\forall \alpha)\quad \nu(b_\alpha)+1 \ge \delta \abs{\alpha}.
  \end{equation}
  A subset
  of~\(\comling{P}\) is bounded if and only if there
  is one~\(\delta\) that works for all its elements.  We may
  replace the condition~\eqref{converpc} by
  \[
  \liminf_{\abs{\alpha}\to\infty} \frac{\nu(b_\alpha)}{\abs{\alpha}}>0.
  \]
  This gives the Monsky--Washnitzer completion~\(P^\updagger\) with
  the bornology specified above, compare \cite{mw}*{Theorem 2.3}.
\end{proof}

\begin{remark}
  \label{rem:separated}
  Let~\(R\) be a finitely generated, commutative \(\dvr\)\nb-algebra.
  Then~\(R^\updagger\) is Noetherian by~\cite{ful}.  The ideal \(\dvgen
  R^\updagger\) is contained in the Jacobson radical because any
  element of the form \(1-\dvgen z\) with \(z\in R^\updagger\) has an
  inverse given by \(\sum_{j=0}^{\infty} \dvgen^j z^j\in
  R^\updagger\).  Krull's Intersection Theorem implies that any
  finitely generated \(R^\updagger\)\nb-module is
  \(\dvgen\)\nb-adically separated (see
  \cite{Atiyah-Macdonald:Commutative}*{Corollary 10.19}).  In
  particular, \(R^\updagger/J\cdot R^\updagger\) is
  \(\dvgen\)\nb-adically separated for any ideal~\(J\) in~\(R\).
  Equivalently, \(J\cdot R^\updagger\) is \(\dvgen\)\nb-adically
  closed in~\(R^\updagger\).
\end{remark}

\begin{lemma}
  \label{lem:onto}
  Let~\(R\) be a finitely generated, commutative \(\dvr\)\nb-algebra and
  let \(J\triqui R\) be an ideal.  The natural map
  \(R^\updagger/J R^\updagger\to (R/J)^\updagger\) is an isomorphism.
\end{lemma}

\begin{proof}
  By Remark~\ref{rem:separated}, \( R^\updagger/J R^\updagger\)
  is \(\dvgen\)\nb-adically separated.  As a quotient of a weakly
  complete algebra, it is then weakly complete (see~\cite{mw}*{Theorem~1.3}).
 Hence the natural map $R/J \to R^{\updagger}/J R^{\updagger}$ extends to a map $(R/J)^{\updagger} \to R^{\updagger}/J R^{\updagger}$. This map is inverse to the map in the statement of the lemma.
\end{proof}

\begin{proposition}
  \label{pro:dagger_for_quotient}
  Let \(P = \dvr[x_1,\dotsc,x_n]\), let \(J\triqui P\) be an
  ideal, and \(R = P/J\).  The bornological closure of~\(J\) in~\(P^\updagger\), the closure of~\(J\) in the
  \(\dvgen\)\nb-adic topology on~\(P^\updagger\), and the
  ideal~\(J\cdot P^\updagger\) generated by~\(J\) in~\(P^\updagger\) are
  all the same.  The resulting quotient \(P^\updagger/ (J\cdot
  P^\updagger)\) is isomorphic to both~\(R^\updagger\) and
  \(\comling{R}\).  Thus
  \(R^\updagger\cong\comling{R}\).
\end{proposition}

\begin{proof}
  The multiplication map in~\(P^\updagger\) is bounded and every element of $P^\updagger$ is in the $\dvgen$-adic closure of a bounded subset of $P$ (with respect to the bornology defined at the beginning of Section \ref{subsec:dag}).  Thus \(J\cdot P^\updagger\)
  is contained in the bornological closure of \(J=J\cdot P\).  This is
  further contained in the \(\dvgen\)\nb-adic closure of~\(J\)
  because any bornologically convergent sequence converges in the
  \(\dvgen\)\nb-adic topology on~\(P^\updagger\) and so a
  \(\dvgen\)\nb-adically closed subset is bornologically closed.
  Remark~\ref{rem:separated} applied to~\(P\) shows that \(J\cdot
  P^\updagger\) is \(\dvgen\)\nb-adically closed
  in~\(P^\updagger\).  Thus the bornological and the
  \(\dvgen\)\nb-adic closures of~\(J\) are both equal to~\(J\cdot
  P^\updagger\).  Lemma~\ref{lem:onto} applied to~\(P\) identifies
  \(P^\updagger / J\cdot P^\updagger\) with~\(R^\updagger\).

  Let~\(T\) be a complete bornological \(\dvr\)\nb-algebra with
  \(\varrho(M)\le1\) for all bounded \(M\subseteq T\).  A
  homomorphism \(\comling{R} \to T\) is equivalent to
  a homomorphism \(R\to T\) by
  Proposition~\ref{pro:dagger_universal}.  This is equivalent to a
  homomorphism \(P\to T\) that vanishes on~\(J\).  By the universal
  property of~\(\comling{P}\) in
  Proposition~\ref{pro:dagger_universal}, this is equivalent to a
  homomorphism \(\comling{P} \to T\) that vanishes on
  the image of~\(J\).  Any such homomorphism vanishes on the
  bornological closure of~\(J\) and then descends to a homomorphism
  on the quotient of~\(\comling{P}\) by this
  bornological closure.  This quotient is complete, and
  \(\varrho(M)\le1\) for all bounded \(\dvr\)\nb-submodules in this
  quotient.  Thus it has the universal property that
  characterizes~\(\comling{R}\), that is, it is
  naturally isomorphic to~\(\comling{R}\).  Since
  \(\comling{P} = P^\updagger\) by
  Lemma~\ref{lem:polynomial_dagger_injective} and since the
  bornological and \(\dvgen\)\nb-adic closures of~\(J\) are the
  same, we get \(R^\updagger\cong\comling{R}\).
\end{proof}

This finishes the proof of Theorem~\ref{the:MW_completion}.  This
theorem allows us to apply results about linear growth completions
such as Propositions \ref{pro:dagger_functorial}
and~\ref{pro:tensor_dagger} to dagger completions of finitely
generated, commutative \(\dvr\)\nb-algebras.

Now we turn to the completions~\(\comg{R}{J}{ \alpha}\)
for \(\alpha\in \Q\cap[0,1]\).  We use the tube algebra
\(\tub{R}{J}{\alpha}\) introduced in Definition~\ref{def:alpha-tube}
and its Monsky--Washnitzer completion \(\tub{R}{J}{\alpha}^\updagger\).

\begin{theorem}
  \label{the:MW_tube}
  Let~\(R\) be a finitely generated, commutative
  \(\dvr\)\nb-algebra, let \(J\triqui R\) be an ideal with \(\dvgen
  \in J\), and \(\alpha\in\Q\cap[0,1]\).  The
  completion~\(\comg{R}{J}{\alpha}\) is
  \(\dvf\otimes \tub{R}{J}{\alpha}^\updagger\).
\end{theorem}

\begin{proof}
  Since \(\alpha\in\Q\), the tube algebra~\(\tub{R}{J}{\alpha}\) is
  again commutative and finitely generated and carries the fine
  bornology by Lemmas \ref{lem:tube_fg} and~\ref{lem:tube_fine}.
  Proposition~\ref{pro:linear_growth_on_tube} and
  Theorem~\ref{the:MW_completion} for the tube algebra identify
  \[
  \comg{R}{J}{\alpha}
  \cong \dvf\otimes \comling{\tub{R}{J}{\alpha}}
  \cong \dvf\otimes \tub{R}{J}{\alpha}^\updagger.\qedhere
  \]
\end{proof}

\section{Homological algebra for completions
  of commutative algebras}
\label{sec:HH_dagger}
\subsection{Hochschild homology.}
The Hochschild homology of the commutative algebras of functions on
smooth algebraic varieties is described by the
Hochschild--Kostant--Rosenberg Theorem. Connes proves an analogous
result for algebras of smooth functions on smooth manifolds, with
Hochschild homology defined using completed tensor products.  We shall
carry over Connes' result to the completions \(\HH_*(\comling{R})\)
and \(\HH_*(\comg{R}{J}{\alpha})\).
Our computation uses purely algebraic tools and hence describes
\(\HH_*(R^\updagger)\)
and \(\HH_*(\comg{R}{J}{\alpha})\)
only as \(\dvr\)\nb-modules, without their canonical bornologies.

\begin{definition}
  \label{def:HH}
  Let~\(A\) be a complete bornological \(\dvr\)\nb-algebra.  The
  \emph{Hochschild chain complex} \((C(A),b)\) consists of the
  complete bornological \(\dvr\)\nb-modules \(C_n(A) \defeq
  A^{\hotimes n+1}\) for \(n\ge0\) with the boundary maps \(b\colon
  C_n(A) \to C_{n-1}(A)\),
  \begin{multline*}
    b(a_0 \otimes \dotsb \otimes a_n) \defeq
    \sum_{i=0}^{n-1} (-1)^i a_0 \otimes \dotsb \otimes a_{i-1} \otimes
    (a_i\cdot a_{i+1}) \otimes a_{i+2} \otimes\dotsb \otimes a_n
    \\+ (-1)^n a_n a_0\otimes a_1 \otimes \dotsb \otimes a_{n-1}.
  \end{multline*}
  The \emph{Hochschild homology} \(\HH_*(A)\) is the 
  homology of~\((C(A),b)\).
\end{definition}
If no bornology on a \(\dvr\)\nb-algebra~\(A\)
  is specified, then we give it the fine bornology.  Then
  \(A^{\otimes n} = A^{\hotimes n}\),
  and so the bornological Hochschild chain complex of~\(A\)
  as a bornological algebra is the usual, purely algebraic one.

There are many equivalent chain complexes that compute periodic cyclic
homology, and all of them would work for our purposes.  To be
definite, we choose the total complex of the cyclic bicomplex.

\begin{definition}[\cite{mix}]
  A \emph{mixed complex} in an additive category~\(\cC\) is a
  \(\Z\)\nb-graded object~\(M\) of~\(\cC\)
  together with homogeneous maps \(b\colon M\to M[-1]\) and
  \(B\colon M\to M[+1]\) such that \(b^2=B^2=bB+Bb=0\); here \([\pm1]\)
  denotes a degree shift, so \(b\)
  and~\(B\)
  have degree \(-1\)
  and~\(+1\),
  respectively.  The \emph{Hochschild complex} of~\(M\) is the chain
  complex \((M,b)\).  We will only consider mixed complexes that are
  nonnegatively graded in the sense that \(C_n=0\) for \(n<0\).

  Given a nonnegatively graded mixed complex~\((M,b,B)\)
  in a complete, additive category~\(\cC\),
  let~\(CM\)
  be the \(\Z/2\)\nb-graded
  chain complex that is~\(\prod_n M_n\)
  graded by parity of~\(n\)
  as a \(\Z/2\)\nb-graded
  object of~\(\cC\),
  with the boundary map~\(b+B\).
  The homology of~\(CM\)
  is called the \emph{periodic cyclic homology} of~\(M\).

  Let~\(R\)
  be a unital, complete bornological \(\dvr\)\nb-algebra.
  Then \((C(R),b,B)\)
  with Connes' boundary map~\(B\)
  is a mixed complex.  Its periodic cyclic homology is called the
  \emph{periodic cyclic homology} of~\(R\) and denoted \(\HP_*(R)\).
\end{definition}

Explicitly, \(\HP_*(R)\)
is the homology of the \(\Z/2\)\nb-graded
chain complex \(CC(R) \defeq CC^\even(R) \oplus CC^\odd(R)\)
with the boundary map \(b+B\), where
\[
CC^\even(R) \defeq \prod_{j=0}^\infty R^{\hotimes 2j},\qquad
CC^\odd(R) \defeq \prod_{j=0}^\infty R^{\hotimes 2j+1},
\]
and \(B \defeq (1-t)sN\) as in~\cite{lq} with
\begin{align*}
  N(a_0\otimes \dotsb \otimes a_n)
  &\defeq \sum_{i=0}^n
    (-1)^n a_i \otimes a_{i+1} \otimes \dotsb \otimes a_n \otimes a_0
    \otimes \dotsb \otimes a_{i-1},\\
  s(a_0\otimes \dotsb \otimes a_n)
  &\defeq 1\otimes a_0 \otimes \dotsb \otimes a_n,\\
  (1-t)(1\otimes a_0 \otimes \dotsb \otimes a_n)
  &= 1\otimes a_0 \otimes \dotsb \otimes a_n +
    (-1)^n a_n \otimes 1\otimes a_0 \otimes \dotsb \otimes a_{n-1}.
\end{align*}

\begin{example}\label{ex:periodification}
	Let $(\Omega,d)$ be a nonnegatively graded cochain complex.  Consider the mixed complex $M(\Omega)\defeq (\Omega,0,d)$ with
	zero boundary map $\Omega_*\to \Omega_{*-1}$.  The periodic
        cyclic homology of $M(\Omega)$ is the homology of~$\Omega$
        made periodic; we have
	\[
	\HP_j(M(\Omega))=\prod_{n\ge 0} H^{2n+j}(\Omega,d)\qquad (j=0,1).
	\]
	\end{example}
After these general definitions, we describe these homology theories
for the complete bornological algebras introduced above.  Let~\(R\)
be a finitely generated, commutative, torsion-free
\(\dvr\)\nb-algebra.
Let \(J\triqui R\)
be an ideal with \(\dvgen\in J\).
Let \(\alpha\in[0,1]\cap \Q\).
We shall describe the Hochschild homology of the bornological
\(\dvr\)\nb-algebras
\(\comling{R}\)
and~\(\comg{R}{J}{\alpha}\).

Hochschild homology is a special case of the derived tensor product
functor~\(\Tor\)
on the category of complete bornological bimodules.  Hence we need
some homological properties of the dagger completion~\(R^\updagger\),
which is identified with \(\comling{R}\)
in Theorem~\ref{the:MW_completion}.  Intermediate results in our proof
suggest that \(R \to R^\updagger\)
and \(R\to \comg{R}{J}{\alpha}\)
for \(\alpha\in\Q\)
behave like the localisations of Taylor~\cite{Taylor:General_fun}.

\begin{lemma}
  \label{lem:flat2}
  Let~\(R\) be a finitely generated, commutative \(\dvr\)\nb-algebra.
  Then the canonical homomorphism \(R\to R^\updagger\) is flat and the map to the $\dvgen$-adic completion $R^{\updagger} \to \coma{R}$ is faithfully flat.
\end{lemma}

\begin{proof}
  Since \(R\) and~\(R^\updagger\) are Noetherian, the maps to their
  \(\dvgen\)\nb-adic completion \(R \to \coma{R}\) and \(R^\updagger \to
  \coma{R}\) are flat (see
  \cite{Atiyah-Macdonald:Commutative}*{Proposition 10.14}).  The
  ideal \(\dvgen R^\updagger\) is contained in the Jacobson radical
  of~\(R^\updagger\) by Remark~\ref{rem:separated}.  Hence the map
  \(R^\updagger \to \coma{R}\) is  faithfully flat.  This implies
  that \(R\to R^\updagger\) is flat.
\end{proof}

\begin{lemma}
  \label{lem:flat}
  Let \(f\colon R\to S\) be a homomorphism of commutative
  \(\dvr\)\nb-algebras of finite type.  If~\(f\) is flat, then so is
  \(f^\updagger\colon R^\updagger\to S^\updagger\).
\end{lemma}

\begin{proof}
It is enough to prove that $\coma{f}\colon \coma{R} \to \coma{S}$ is flat. Indeed, if this is the case, then the composition $R^{\updagger} \xrightarrow{f^{\updagger}} S^{\updagger} \to \coma{S}$, which equals the composition $R^{\updagger} \to \coma{R} \xrightarrow{\coma{f}} \coma{S}$, is flat. Since $S^{\updagger} \to \coma{S}$ is faithfully flat by Lemma~\ref{lem:flat2} this implies that $f^{\updagger}$ is flat.

To see that $\coma{f}$ is flat, we use a flatness criterion of Bourbaki \cite{BourbakiAC}*{Ch.~III, \S5, Thm.~1}.
Let $I$ be any ideal of $\coma{R}$. Since $\coma{R}$ is Noetherian, $I$ is finitely generated. Hence $I \otimes_{\coma{R}} \coma{S}$ is a finitely generated $\coma{S}$-module. Since also $\coma{S}$ is Noetherian, $I \otimes_{\coma{R}} \coma{S}$ is $\dvgen$-adically separated. This means that $\coma{S}$ as an $\coma{R}$-module is `id\'ealement s\'epar\'e pour $\dvgen R$' in the sense of Bourbaki.
Since $R \to S$ is flat, also $\coma{R} / \dvgen^{n}\coma{R} \to \coma{S} / \dvgen^{n}\coma{S}$ is flat for every $n$. Now the flatness criterion implies that $\coma{S}$ is flat as an $\coma{R}$-module.
\end{proof}

\begin{lemma}
  \label{lem:dagmod}
  Let \(R\) be a finitely generated, commutative \(\dvr\)\nb-algebra
  and \(M\) a finitely generated module over~\(R\).
  Then \(R^\updagger\otimes_R M\) is a complete
  bornological \(\dvr\)\nb-module, in the bornology induced by
    the bornology on~$R^\updagger$, introduced in
    Section~\textup{\ref{subsec:dag}}, and the fine bornology
    on~$M$.
\end{lemma}
\begin{proof}
The symmetric bimodule structure on~\(M\) gives an algebra
    structure on \(M\oplus R\) so that~\(M\) is an ideal with
    \(M^2=0\) and the projection to~\(R\) is an algebra homomorphism
    with kernel~\(M\).
We denote this algebra by~\(M\rtimes R\).
Since M is finitely generated we may write it as $M=\bigoplus_{i=1}^n Rx_i/L$, so
$$M\rtimes R=
R[x_1,\dots, x_n]/(\langle x_1,\dots,x_n\rangle^2 +\langle L\rangle ).$$
If S is of finite type and $J$ is an ideal in $S$, then $(S/J)^\updagger=S^\updagger/JS^\updagger$ by Lemma \ref{lem:onto}. Applying this to the identity above, and using that $(R[X]/\langle X\rangle^2)^\updagger=R^\updagger[X]/\langle X\rangle^2$, we get
\begin{multline*}(M\rtimes R)^\updagger=
 R^\updagger[x_1,\dots,x_n]/(\langle x_1,\dots,x_n\rangle^2+\langle L\rangle)   \\  =
R^\updagger\otimes_RR[x_1,\dots, x_n]/(\langle x_1,\dots,x_n\rangle^2+\langle L\rangle)=
R^\updagger\otimes_R(M\rtimes R). \end{multline*}
Finally, $R^\updagger \otimes_R M$ is the kernel of the map $R^\updagger\otimes_R(M\rtimes R) \to R^\updagger$.
Since $R^\updagger$ is bornologically separated, this kernel is bornologically closed and inherits bornological completeness from $(M\rtimes R)^\updagger$.
\end{proof}
Let~\(R\) be a commutative bornological \(\dvr\)\nb-algebra of finite
type equipped with the fine bornology.
The Hochschild complex \(C(R)\) is the chain complex associated to a cyclic object in the category of $\dvr$\nb-modules.
The face and degeneracy maps as well as the (unsigned)
cyclic permutations of this cyclic object are in fact $\dvr$\nb-algebra maps.

All algebra maps are bounded for the linear growth bornology associated to the fine bornology.
Hence~\(\ling{C(R)}\) is a cyclic object in the category of bornological
commutative \(\dvr\)\nb-algebras.  The bornological
completion of~\(\ling{C(R)}\) is the Monsky--Washnitzer completion
\(C(R)^\updagger\) by Theorem~\ref{the:MW_completion}.
Proposition~\ref{pro:tensor_dagger} implies
\begin{equation}
  \label{eq:cyclic_module_dagger}
  C_n(R)^\updagger
  \cong \comling{C_n(R)}
  = \comling{(R^{\otimes n+1})}
  \cong (\comling{R})^{\hotimes n+1}
  = C_n(\comling{R})
  \cong C_n(R^\updagger)
\end{equation}
by the definition of~\(C_n\) for complete bornological
\(\dvr\)\nb-algebras such as~\(R^\updagger\).

The \emph{bar resolution} of a complete bornological algebra $R$ is the complex
$(B_n(R)=R^{\comb{\otimes}n+2})_{n\ge 0}$ with boundary map
\[
    b'(a_0 \otimes \dotsb \otimes a_{n+1}) \defeq
    \sum_{i=0}^{n} (-1)^i a_0 \otimes \dotsb \otimes a_{i-1} \otimes
    (a_i\cdot a_{i+1}) \otimes a_{i+2} \otimes\dotsb \otimes a_{n+1}.
\]
The bar resolution is an $R$-bimodule resolution of $R$. It admits a bounded contraction $s: B_n \to B_{n+1}$ defined by $s(a_0 \otimes \dotsb \otimes a_{n+1}) = 1\otimes a_0 \otimes \dotsb \otimes a_{n+1}$. One has $b's +sb'=$ id. Thus $B_n(R)$ even is a split exact resolution.
Write $R^{e} \defeq R \otimes_{\dvr} R$.

\begin{proposition}
  \label{prop:hh2}
  Let~\(R\)
  be a finitely generated, torsion-free, commutative
  \(\dvr\)\nb-algebra.
  \begin{enumerate}
    \item[(a)]
  Both \((R^e)^\updagger \otimes_{R^e} B(R)\)
  and \(B(R^\updagger)\)
  are flat \((R^e)^\updagger\)\nb-module
  resolutions of~\(R^\updagger\).
  The natural map \(C(R) \to C(R^\updagger)\) induces a quasi-isomorphism
  \begin{equation}
    \label{map:hh2}
    R^\updagger \otimes_R (C(R),b) \to (C(R^\updagger),b).
  \end{equation}
  The natural map \(\HH(R)\to \HH(R^\updagger)\)
  induces an isomorphism
  \begin{equation}
    \label{eq:Hochschild_dagger}
    R^\updagger \otimes_R \HH_*(R) \cong \HH_*(R^\updagger).
  \end{equation}
  \item[(b)] Let~\(R\) be equipped with the fine
  bornology, let \(J\triqui R\) be an ideal with \(\dvgen\in J\), and
  \(\alpha\in [0,1]\) rational.  The canonical map is an isomorphism
  \[
  \comg{R}{J}{\alpha} \otimes_{\ul{R}} \HH_*(\ul{R})
  \cong \HH_*(\comg{R}{J}{\alpha}).
  \]
  \end{enumerate}
\end{proposition}
\begin{proof}
  (a) The bar resolution \(B(R)\)
  is a resolution of~\(R\)
  by flat \(R^e\)\nb-modules
  because~\(R\)
  is torsion-free.  The inclusion \(R^e\to (R^e)^\updagger\)
  is flat by Lemma~\ref{lem:flat2}.  Hence
  \((R^e)^\updagger\otimes_{R^e} B(R)\)
  is a resolution of \((R^e)^\updagger \otimes_{R^e} R\).
  The multiplication map \(\mu\colon R\otimes R\to R\)
  is a surjective algebra homomorphism, so \(R\cong R^e / \ker\mu\).
  Hence
  \((R^e)^\updagger \otimes_{R^e} R \cong (R^e)^\updagger \otimes_{R^e}
  (R^e/\ker \mu) \cong (R^e)^\updagger/ (R^e)^\updagger\cdot (\ker \mu)\).
  This is isomorphic to~\(R^\updagger\)
  by Lemma~\ref{lem:onto}.  The resolution
  \((R^e)^\updagger\otimes_{R^e} B(R)\)
  of~\(R^\updagger\)
  consists of flat \((R^e)^\updagger\)\nb-modules
  because
  \(Q \otimes_{(R^e)^\updagger} ((R^e)^\updagger \otimes_{R^e} B_n(R))
  \cong Q \otimes_{R^e} B_n(R) \cong Q\otimes R^{\otimes n}\)
  for any \((R^e)^\updagger\)-module~\(Q\),
  and~\(R^{\otimes n}\) is a flat \(\dvr\)\nb-module.

  Equip \(B_n(R) = R^{\otimes n+2}\) with the \(\dvr\)\nb-algebra
  structure as a tensor product of copies of~\(R\).  Then
  \[
  R^e \to B_n(R),\qquad
  b_1 \otimes b_2
  \mapsto b_1 \otimes 1 \otimes 1 \otimes \dotsb \otimes 1 \otimes
  b_2,
  \]
  is a ring homomorphism.  It is flat because~\(R\) is a flat
  \(\dvr\)\nb-module.  Then the induced homomorphism \((R^e)^\updagger
  \to B_n(R)^\updagger\) is also flat by Lemma~\ref{lem:flat}.
  Proposition~\ref{pro:tensor_dagger} and
  Theorem~\ref{the:MW_completion} imply \(B_n(R)^\updagger \cong
  (R^\updagger)^{\hotimes n+2} = B_n(R^\updagger)\).  Thus
  \(B(R^\updagger)\) is a resolution of~\(R^\updagger\) by flat
  \((R^e)^\updagger\)\nb-modules (even admitting a bounded contraction).

  So far, we have shown that both \((R^e)^\updagger \otimes_{R^e}
  B(R)\) and~\(B(R^\updagger)\) are resolutions of~\(R^\updagger\) by
  flat \((R^e)^\updagger\)\nb-modules.  They are related by an obvious
  chain map
  \begin{equation}
    \label{eq:bar_compare}
    (R^e)^\updagger \otimes_{R^e} B_n(R)
    \cong (R^e)^\updagger \otimes R^{\otimes n} \to
    (R^e)^\updagger \hotimes (R^\updagger)^{\hotimes n}
    = B_n(R^\updagger).
  \end{equation}
  If~\(Q\) is an \((R^e)^\updagger\)-module, then both
  \(Q\otimes_{(R^e)^\updagger} ((R^e)^\updagger \otimes_{R^e} B(R))
  \cong Q \otimes_{R^e} B(R)\) and \(Q\otimes_{(R^e)^\updagger}
  B(R^\updagger)\) compute the derived functors
  \(\Tor^{(R^e)^\updagger}_n(Q,R^\updagger)\).  Hence the chain map \(Q
  \otimes_{R^e} B(R) \to Q\otimes_{(R^e)^\updagger} B(R^\updagger)\)
  induced by~\eqref{eq:bar_compare} is a quasi-isomorphism.

  We claim that this quasi-isomorphism specializes
  to~\eqref{map:hh2} if \(Q=R^\updagger\).  We have
  \[
  R^\updagger \otimes_{(R^e)^\updagger} \bigl((R^e)^\updagger \otimes_{R^e} B(R)\bigr)
  \cong R^\updagger \otimes_{R^e} B(R)
  \cong R^\updagger \otimes_R (R\otimes_{R^e} B(R))
  \cong R^\updagger \otimes_R C(R).
  \]
  It remains to identify
  \(R^\updagger \otimes_{(R^e)^\updagger} B_n(R^\updagger)\)
   with~\(C_n(R^\updagger)\). Here we take
  an incomplete tensor product of two completions.
  This often causes trouble (compare
  \cite{Meyer:Homological_algebra_Schwartz}*{Section~3}), but not
  here.  We have identified
  \(R^\updagger \cong (R^e)^\updagger \otimes_{R^e} R\) above.  Hence
  \begin{align*}
    R^\updagger \otimes_{(R^e)^\updagger} B_n(R^\updagger)
    &\cong R \otimes_{R^e} (R^e)^\updagger \otimes_{(R^e)^\updagger} B_n(R^\updagger)
      \cong R \otimes_{R^e} B_n(R^\updagger)
    \\ &\cong B_n(R^\updagger)/(\ker\mu)\cdot B_n(R^\updagger),
  \end{align*}
  where \(\ker\mu\triqui R^e\)
  is the kernel of the multiplication homomorphism
  \(\mu\colon R^e\to R\).
  We have
  \((\ker\mu)\cdot B_n(R^\updagger) = (B_n(R)\cdot \ker\mu)\cdot
  B_n(R^\updagger)\)
  for the ideal \(B_n(R)\cdot \ker\mu \triqui B_n(R)\)
  generated by~\(\ker \mu\)
  in the finitely generated commutative algebra~\(B_n(R)\).
  Lemma~\ref{lem:onto} identifies
  \(B_n(R^\updagger)/(B_n(R)\cdot \ker\mu)\cdot B_n(R^\updagger)\)
  with the dagger completion of the quotient
  \(B_n(R)/B_n(R)\cdot\ker\mu \cong C_n(R)\).
  Thus
  \(R^\updagger \otimes_{(R^e)^\updagger} B_n(R^\updagger) \cong
  C_n(R)^\updagger \cong C_n(R^\updagger)\).
  This finishes the proof of~\eqref{map:hh2}.

The homology on the left in~\eqref{map:hh2} is
  \(R^\updagger \otimes_R H_*(C(R),b) = R^\updagger \otimes_R \HH_*(R)\)
  because~\(R^\updagger\)
  is a flat \(R\)\nb-module
  by Lemma~\ref{lem:flat2}.  The homology on the right
  in~\eqref{map:hh2} is~\(\HH_*(R^\updagger)\)
  by definition.  Thus the quasi-isomorphism~\eqref{map:hh2} implies
  the isomorphism~\eqref{eq:Hochschild_dagger}.

  (b) Let
  \(T\defeq \tub{R}{J}{\alpha}\) be the \(\alpha\)\nb-tube algebra
  for \(J\triqui R\).  This is a torsion-free, commutative
  \(\dvr\)\nb-algebra. It is finitely generated
    because~$\alpha$ is rational.  It carries the fine bornology by
  Lemma~\ref{lem:tube_fine}.
  Proposition~\ref{pro:linear_growth_on_tube} gives an isomorphism
  \(\comg{R}{J}{\alpha} \cong \comling{T} \otimes \dvf\).
  The functor \(R\mapsto C(R)\) commutes with \({-} \otimes \dvf\).
  These facts and~\eqref{eq:cyclic_module_dagger} give
  \[
  C_n(\comg{R}{J}{\alpha})
  \cong C_n(\comling{T}\otimes \dvf)
  \cong C_n(\comling{T})\otimes \dvf
  \cong \comling{C_n(T)}\otimes \dvf.
  \]
  Equation~\eqref{eq:Hochschild_dagger} says that
  \[
  \comling{T} \otimes_T \HH_*(T) \cong \HH_*(\comling{T}).
  \]
  Taking Hochschild homology commutes with tensoring with~\(\dvf\),
  that is,
  \[
  \HH_*(X) \otimes \dvf \cong \HH_*(X \otimes \dvf)
  \]
  for \(X=R,T,\comling{T}\).  By construction of~\(T\), \(\ul{T}
  \defeq \dvf\otimes T\) and \(\ul{R}\defeq \dvf\otimes
  R\) are isomorphic.  Thus
  \begin{multline*}
    \phantom{{}\cong{}}
    \comg{R}{J}{\alpha} \otimes_{\ul{R}} \HH_*(\ul{R})
    \cong \comling{T} \otimes \dvf \otimes_{\ul{T}} \HH_*(\ul{T})
    \cong \comling{T} \otimes \dvf \otimes_T (\HH_*(T) \otimes \dvf)
    \\\cong \comling{T} \otimes_T \HH_*(T) \otimes \dvf
    \cong \HH_*(\comling{T}) \otimes \dvf
    \cong \HH_*(\comling{T} \otimes \dvf)
    \cong \HH_*(\comg{R}{J}{\alpha}),
  \end{multline*}
  where the first and last step use \(\comg{R}{J}{\alpha} \cong \dvf
  \otimes \comling{T}\).
\end{proof}

\begin{definition}
  \label{def:dR}
  Let~$R$ be a commutative $\dvr$-algebra. Then
  \(\Omega_{\ul{R}}^*\) denotes the dg-algebra of K\"ahler
  differential forms for~\(\ul{R}\) (over the ground field~\(\dvf\))
  and \((\Omega_{\ul{R}}^*,d)\) is the de Rham complex
  for~\(\ul{R}\).  The de Rham complex for $\ul{R^\updagger}$ is
    \((\ul{R^\updagger}\otimes_{\ul{R}}\Omega_{\ul{R}}^*,d)\).  For
    an ideal $J$ in $R$ and $\alpha \in [0,1]$, the de Rham complex
    for $\comg{R}{J}{\alpha}$ is
    \((\comg{R}{J}{\alpha}\otimes_{\ul{R}}\,\Omega_{\ul{R}}^*,d)\).
\end{definition}

\begin{remark}
  \label{rem:dR}
  For finitely generated~$R$ and rational~$\alpha$,
  Lemma~\ref{lem:dagmod} (in combination with
  Proposition~\ref{pro:linear_growth_on_tube}) shows that
  $\underline{R}^\updagger\otimes_{\ul{R}}\Omega_{\ul{R}}^*$ and
  $\comg{R}{J}{\alpha}\otimes_{\ul{R}}\Omega_{\ul{R}}^*$ are the
  $\dagger$- and $J,\alpha$-completions of $\Omega_{R}^*$ tensored by $K$,
  respectively. This justifies our definition of the de Rham
  complexes for $\underline{R}^\updagger$ and $\comg{R}{J}{\alpha}$ above.
\end{remark}
We combine Corollary~\ref{prop:hh2}(b) with the
Hochschild--Kostant--Rosenberg Theorem for~\(\ul{R}\):

\begin{theorem}
  \label{the:HKR_dagger}
  Let~\(R\)
  be a finitely generated commutative \(\dvr\)\nb-algebra
  with the fine bornology, \(J\triqui R\)
  an ideal with \(\dvgen\in J\), and \(\alpha\in [0,1]\) rational.
  If \(\ul{R}\defeq R\otimes \dvf\)
  is smooth over~\(\dvf\),
  then the antisymmetrisation map
  \(\Omega_{\ul{R}}^* \to \HH_*(\ul{R})\) induces an isomorphism
  \[
  \comg{R}{J}{\alpha} \otimes_{\ul{R}} \Omega_{\ul{R}}^*
  \cong \HH_*(\comg{R}{J}{\alpha}).
  \]
\end{theorem}

\begin{proof}
  The usual Hochschild--Kostant--Rosenberg Theorem for smooth algebras
  over fields says that the antisymmetrisation map
  \(\Omega_{\ul{R}}^* \to \HH_*(\ul{R})\)
  is an isomorphism (see \cite{lod}*{Theorem 3.4.4}).  This map is
  \(\ul{R}\)\nb-linear,
  so it induces an isomorphism
  \(\comg{R}{J}{\alpha} \otimes_{\ul{R}} \Omega_{\ul{R}}^*
  \cong \comg{R}{J}{\alpha} \otimes_{\ul{R}}
  \HH_*(\ul{R})\).
  This is isomorphic to \(\HH_*(\comg{R}{J}{\alpha})\)
  by Corollary~\ref{prop:hh2}(b).
\end{proof}

In particular, if \(J=\dvgen R\), then \(\comg{R}{J}{\alpha} \cong
R^\updagger \otimes \dvf\) by
Example~\ref{exa:alpha-completion_JpiR}.  Theorem
  \ref{the:HKR_dagger} shows that the natural map from the mixed
  complex $(C(\ul{R^\updagger} ),b,B)$ to the mixed complex
  $(\ul{R^\updagger } \otimes_{\ul{R}} \Omega_{\ul{R}},0,d)$
  (mapping $a_0\otimes \ldots \otimes a_n$ to
  $(\nicefrac1{n!})a_0da_1\ldots da_n$, see \cite{lod}*{Proposition
    1.3.15}) is an isomorphism on `Hochschild' homology (that is,
  the $b$-homology). But then it also is an isomorphism on the
  associated cyclic homology of the mixed complexes (see
  \cite{lod}*{Proposition 2.5.15}).  So our result specializes to an
  isomorphism
\begin{equation}
  \label{eq:HP_Rdagger}
  \HP_j(R^\updagger \otimes\dvf)
  \cong \bigoplus_{n\ge 0} H_{2n+j}(R^\updagger \otimes_R (\Omega_{\ul{R}},d)).
\end{equation}
If~\(R\)
is smooth, then the right-hand side is, by definition, the
Monsky--Washnitzer homology of the quotient \(A\defeq R/J\) made periodic.

\subsection{A short proof of the Feigin--Tsygan Theorem}
The Feigin--Tsygan Theorem establishes an important property of cyclic homology: The periodic cyclic homology of the coordinate ring $B$ of an affine algebraic variety in characteristic 0 gives exactly Grothendieck's infinitesimal cohomology of that variety.

The proofs of Proposition \ref{prop:hh2}(a) and Theorem \ref{the:HKR_dagger} can be adapted to give a short proof of the Feigin--Tsygan Theorem. Actually, the proofs substantially simplify in this case since we can replace some of the preparatory lemmas by well-known easy facts from commutative algebra.

In this section we always assume that $P$ is a finitely generated commutative algebra over a field $\dvf$ of characteristic 0. If $J$ is an ideal in $P$, we denote by $\overline{P_J}$ the $J$-adic completion. In the following proof we will consider tensor powers $P^{\otimes n}$ (over $\dvf$) and we will denote by $J_n$ the kernel of the natural map $P^{\otimes n}\to (P/J)^{\otimes n}$.

We denote by $B_n(\overline{P_J})$ the $J_{n+2}$-adic completion of $B_n(P)$, by $C_n(\overline{P_J})$ the $J_{n+1}$-adic completion of $C_n(P)$ and by $\HH(\overline{P_J})$ the homology of $C(\overline{P_J})$. These are the usual conventions and the analogue of the $\dagger$-versions in Section \ref{sec:HH_dagger}. Again $B_n(\overline{P_J})$ is a resolution of $\overline{P_J}$ by $\overline{P_J}$-bimodules, admitting the continuous contraction $s$.
\bprop
Let $P$ and $J$ be as above and denote by $P^e \defeq P\otimes P$ the enveloping algebra. Then
\begin{enumerate}
  \item [(a)] $\overline{P^e_{J_2}} \otimes_{P^e}B(P)$ is a flat $\overline{P^e_{J_2}}$-module resolution of $\overline{P_J}$.
  \item [(b)] $B(\overline{P_J})$ is a flat $\overline{P^e_{J_2}}$-module resolution of $\overline{P_J}$.
  \item [(c)] One has $\overline{P_J}\otimes_{\overline{P^e_{J_2}}}B(\overline{P_J}) \cong C(\overline{P_J})$.
\end{enumerate}
As a consequence the natural map $C(P)\to C(\overline{P_J})$ induces a quasi-isomorphism
\bgl  \overline{P_J}\otimes_P (C(P),b) \,\to\,(C(\overline{P_J}),b)\egl
and the natural map $\HH_*(P)\to \HH_*(\overline{P_J})$ induces an isomorphism
\bgl  \overline{P_J}\otimes_P \HH_*(P)\cong \HH_*(\overline{P_J})\egl
\eprop
\bproof
(a) Since $P^e$ is Noetherian, $P^e\to \overline{P^e_{J_2}}$ is flat and $\overline{P^e_{J_2}}\otimes_{P^e}P=\overline{P_J}$ (see \cite{Atiyah-Macdonald:Commutative}*{10.14} and \cite{Atiyah-Macdonald:Commutative}*{10.13}). Thus $\overline{P^e_{J_2}}\otimes_{P^e}B(P)$ is a resolution of $\overline{P_J}$. Moreover, $\overline{P^e_{J_2}}\otimes_{P^e}B_n(P)= \overline{P^e_{J_2}}\otimes_\dvf P^{\otimes n}$ is a flat (even free) $\overline{P^e_{J_2}}$-module.

(b) As already mentioned above, the existence of the continuous contraction~$s$ shows that $B(\overline{P_J})$ is a resolution of $\overline{P_J}$.
Identifying the algebra $P^e\otimes P^{\otimes n}$ with $P\otimes P^{\otimes n}\otimes P=B_n(P)$, we can consider $B_n(\overline{P_J})$ as the completion of the Noetherian algebra $\overline{P^e_{J_2}}\otimes P^{\otimes n}$. Therefore, the maps $\overline{P^e_{J_2}} \to \overline{P^e_{J_2}}\otimes P^{\otimes n}\to B_n(\overline{P_J})$ are flat.

(c) Let $I$ be the kernel of the algebra map $B_n(P) = P\otimes
P^{\otimes n}\otimes P \to P\otimes P^{\otimes n} = C_n(P)$ induced
by the homomorphism $P^e\to P$. Since $B_n(P)$ is Noetherian,
we have $I\cdot\overline{B_n(P)}_{J_{n+2}}=
\overline{I}_{J_{n+2}}$ and
$\overline{B_n(P)}_{J_{n+2}}/\overline{I}_{J_{n+2}} =
(\overline{B_n(P)/I})_{J_{n+2}}$. Finally, we have
\[
C_n(\overline{P_J}) \defeq \overline{C_n(P)}_{J_{n+1}} =
\overline{B_n(P)}_{J_{n+2}}/I\cdot\overline{B_n(P)}_{J_{n+2}}=
P\otimes_{P^e}\overline{B_n(P)}_{J_{n+2}},
\]
whence $\overline{P_J}\otimes_{\overline{P^e_{J_2}}}\overline{B_n(P)}_{J_{n+2}} = P\otimes_{P^e}\overline{P^e_{J_2}}\otimes_{\overline{P^e_{J_2}}}\overline{B_n(P)}_{J_{n+2}} = P\otimes_{P^e}\overline{B_n(P)}_{J_{n+2}}=C_n(\overline{P_J})$.

Now, since $\overline{P_J}\otimes_P (C(P),b)$ and $(C(\overline{P_J}),b)$ both compute the same Tor-functor, we get the asserted quasi-isomorphism.
\eproof
Repeating now verbatim the (short) proof of Theorem~\ref{the:HKR_dagger} we obtain
\btheo Let $P$ be a smooth commutative algebra over a field $\dvf$ of characteristic 0 and $J$ an ideal in $P$. Then
\[
\overline{P_J}\otimes_P\Omega^*_P\cong \HH_*(\overline{P_J})
\]
\etheo
As an immediate consequence we get the Feigin--Tsygan Theorem
\btheo
Let $P$ and $J$ be as above. Then
\[
\HP_*(P/J) = \HP_*(\overline{P_J}) = HdR_*(\overline{P_J})
\]
Here, $HdR_*(\overline{P_J})$ stands for the homology of the complex $\left(\overline{P_J}\otimes_P\Omega^*_P,\, d\right)$ made $\Z/2$-periodic.
\etheo
The proof of the second equality in the displayed formula follows
again by repeating the discussion after
Theorem~\ref{the:HKR_dagger}.  In addition, the first equality
follows from the invariance of~$\HP_*$ under nilpotent extensions.
Alternatively, it follows directly from the definition of~$\HP_*$ in
the description of cyclic homology in~\cite{cqhc}.

\begin{remark} Since $\Omega^*_P$ is a finitely generated module over the Noetherian algebra $P$, we see that $\overline{P_J}\otimes_P\Omega^*_P$ is its completion. Thus $HdR_*(\overline{P_J})$ describes the infinitesimal cohomology of $P/J$ in the sense of Grothendieck, see \cite{coco}*{Theorem 6.1}.
\end{remark}

\subsection{Homotopy invariance}
\label{sec:homotopy_invariance}
Let \(R\) and~\(S\)
be complete bornological \(\dvr\)\nb-\hspace{0pt}algebras
and let \(f_0,f_1\colon R\to S\)
be bounded unital algebra homomorphisms.
\begin{definition}A (dagger-continuous)
\emph{homotopy between \(f_0\)
  and~\(f_1\)}
is a bounded unital algebra homomorphism
\(F\colon R\to S\hotimes \dvr[x]^\updagger\)
with \((\id_S\hotimes \ev_t)\circ F = f_t\)
for \(t=0,1\).\end{definition}
As we shall see, if such a homotopy exists, then  \(f_0\)
and~\(f_1\)
induce the same map in periodic cyclic homology,
\(\HP_*(f_0)=\HP_*(f_1)\).
Put in a nutshell, periodic cyclic homology is invariant under
dagger-continuous homotopies.  This is well known for polynomial
homotopies, that is, \(F\colon R\to S\otimes \dvr[x]\).
The proof of polynomial homotopy invariance shows, in fact, that a
polynomial homotopy induces a bounded chain homotopy between the maps
\(CC(f_0)\)
and \(CC(f_1)\).
Making this chain homotopy explicit, it can be checked then that the
same formulas still define a bounded chain homotopy between the maps
\(CC(f_0)\)
and \(CC(f_1)\)
when we start from a homotopy that is only dagger-continuous.  We
shall outline another proof that exhibits more clearly why
dagger-continuous homotopies work here (in contrast, continuous
homotopies in the usual sense do not work).

Equation~\eqref{eq:HP_Rdagger} implies easily that the
evaluation map at~\(0\) induces an isomorphism \(\HP_*(\dvf \otimes
\dvr[x]^\updagger) \cong \dvf\) if~$\dvf$ has characteristic~$0$.
We shall need a stronger statement:

\begin{lemma}
  \label{lem:HP_dagger_1-variable}
  Assume that~$\dvf$ has characteristic~$0$.
  The kernel of the chain map from
  \(CC(\dvf\otimes\dvr[x]^\updagger)\)
  to~\(\dvf\)
  with zero boundary map induced by evaluation at~\(0\)
  on \(C_0(\dvf\otimes\dvr[x]^\updagger)\)
  is contractible through a bounded contracting homotopy.
  Similarly, the kernel of the chain map from
  $\left((\dvf\otimes\dvr[x]^\updagger)\otimes_{\dvf[x]}
    \Omega_{\dvf[x]}^*, d\right)$ to~$\dvf$ is contractible
  through a bounded chain homotopy.
\end{lemma}

\begin{proof}
  Let \(R\defeq \dvr[x]\)
  and identify \(R\otimes R\cong \dvr[x,y]\).  The map
  \[
  R\otimes R \to R\otimes R,\qquad
  f\mapsto (x-y)\cdot f,
  \]
  is an injective \(R\)\nb-bimodule
  homomorphism, and its cokernel is isomorphic to~\(R\)
  with the usual bimodule structure through the map
  \(R\otimes R\to R\),
  \(f\mapsto f(x,x)\).
  The chain complex $0 \to R \otimes R \to R \otimes R \to R \to 0$
  remains exact when we tensor with
  \(\dvr[x,y]^\updagger\).  Even more,
  \begin{equation}
    \label{eq:bimodule_resolution_one-variable}
    0 \to \dvr[x,y]^\updagger \to \dvr[x,y]^\updagger \to \dvr[x]^\updagger
    \to 0
  \end{equation}
  is a short exact sequence of bornological left
  \(\dvr[x]^\updagger\)\nb-modules
  with a bounded linear section.  The non-trivial (but easy)
  observation is that division by \(x-y\)
  is a bounded linear map from the kernel of the map
  \(\dvr[x,y]^\updagger \to \dvr[x]^\updagger\)
  to \(\dvr[x,y]^\updagger\).
  The bar resolution of~\(\dvr[x]^\updagger\)
  is a resolution with bounded linear section as well, and the free
  bornological bimodules in it are projective with respect to
  \(\dvr[x]^\updagger\)\nb-bimodule
  extensions with a bounded linear section.
  Hence~\eqref{eq:bimodule_resolution_one-variable} is homotopy
  equivalent as a bornological bimodule resolution to the bar
  resolution of~\(\dvr[x]^\updagger\).
  This implies a bounded homotopy equivalence between
  \((C(\dvr[x]^\updagger),b)\)
  and the chain complex with \(\dvr[x]^\updagger\)
  in degrees \(0\)
  and~\(1\)
  and~\(0\)
  in other degrees and with the zero boundary map.  This homotopy
  equivalence intertwines Connes' boundary operator~\(B\)
  and the differential
  \(d\colon \dvr[x]^\updagger\to \dvr[x]^\updagger\),
  \(f\mapsto f'\).
  By a variant of Kassel's perturbation lemma, this implies a chain
  homotopy equivalence between \((CC(\dvr[x]^\updagger),b+B)\)
  and \(d\colon \dvr[x]^\updagger\to \dvr[x]^\updagger\).
  All this remains true after tensoring with~\(\dvf\).

  The absolute integration map
  \(i\colon \sum a_n x^n \mapsto \sum \frac{a_n}{n+1} x^{n+1}\)
  is a bounded linear map on~\(\dvf\otimes \dvr[x]^\updagger\).
  It satisfies \(d\circ i = \id\)
  and \(i\circ d = \id - P_0\),
  where~\(P_0\) is defined by \(P_0(x^n) = \delta_{n,0} x^n\).  Hence
  \(d\colon \dvf\otimes \dvr[x]^\updagger\to \dvf\otimes \dvr[x]^\updagger\) is homotopy
  equivalent to~\(\dvf\) concentrated in degree~\(0\).
\end{proof}

\begin{proposition}
  \label{pro:dagger-continuous_homotopy_HP}
  Let \(P\) and~\(Q\) be complete bornological \(\dvf\)\nb-algebras
  and let \(f_0,f_1\colon P\rightrightarrows Q\) be bounded unital
  homomorphisms.  Assume that
  there is a dagger-continuous homotopy between
  \(f_0\) and~\(f_1\) and  that~\(\dvf\) has characteristic~\(0\).
  \begin{enumerate}
  \item[\textup{(a)}] Assume that $R,S$ are finitely generated
    $\dvr$-algebras, $J\lhd R$ and $I \lhd S$ are ideals, $\alpha
    \in [0,1]$ is rational and that $P=\comg{R}{J}{\alpha}$,
    $Q=\comg{S}{I}{\alpha}$.  Then the maps induced by $f_0, f_1$
    between the de Rham complexes
    \((\comg{R}{J}{\alpha}\otimes_{\ul{R}}\,\Omega_{\ul{R}}^*,\,d)\)
    and
    \((\comg{S}{I}{\alpha}\otimes_{\ul{S}}\,\Omega_{\ul{S}}^*,\,d)\)
    are homotopic with a bounded \(\dvf\)\nb-linear chain homotopy
    \textup{(}for the bornology explained in
    Remark~\textup{\ref{rem:dR}}).
  \item[\textup{(b)}]
    For general complete bornological \(\dvf\)\nb-algebras $P$
    and~$Q$, there is a bounded \(\dvr\)\nb-linear chain
      homotopy between the induced chain maps
    \[
    CC(f_0),CC(f_1)\colon (CC(P),B+b) \rightrightarrows (CC(Q),B+b).
    \]
    Hence \(\HP_*(f_0) = \HP_*(f_1)\).
  \end{enumerate}
\end{proposition}

\begin{proof}
(a) The de Rham complex for $\comg{S}{I}{\alpha} \hotimes\dvr [x]^\updagger$ is $(\comg{S}{I}{\alpha}\otimes_{\ul{S}}\Omega_{\ul{S}}^* ) \bigwedge (\dvr [x]^\updagger \otimes_{\dvr[x]} \Omega_{\dvf[x]}^*)$, where $\bigwedge$ stands for the completed antisymmetric tensor product. The assertion then follows from Lemma \ref{lem:HP_dagger_1-variable}.

  (b) We are working in the symmetric monoidal category of complete
  bornological \(\dvf\)\nb-vector spaces with the tensor
  product~\(\hotimes\).  The definitions of Hochschild and periodic
  cyclic homology above are the standard ones for algebras in this
  symmetric monoidal category.  Thus by \cite{Meyer:HLHA}*{Theorem 4.74} (or by \cite{pus})
  \[
  (CC(Q\hotimes (\dvf \hotimes \dvr[x]^\updagger)), B+b) \simeq
  (CC(Q), B+b) \hotimes (CC(\dvf \hotimes \dvr[x]^\updagger), B+b),
  \]
  where~\(\simeq\) means a bounded homotopy equivalence as chain
  complexes of complete bornological vector spaces.  Now
  Lemma~\ref{lem:HP_dagger_1-variable} says that \((CC(\dvf \hotimes
  \dvr[x]^\updagger), B+b)\simeq \dvf\) concentrated in degree~\(0\),
  so that
  \[
  (CC(Q\hotimes (\dvf \hotimes \dvr[x]^\updagger)), B+b) \simeq
  (CC(Q), B+b).
  \]
  Thus evaluation at \(0\) and at~\(1\)
  induce chain homotopic maps \(CC(Q\hotimes \dvr[x]^\updagger) \to
  CC(Q)\) and \(CC(f_0)\) and~\(CC(f_1)\) are chain
  homotopic as asserted.  Then the induced maps \(\HP_*(f_0)\) and
  \(\HP_*(f_1)\) on homology must be equal.
\end{proof}

\section{Natural chain complexes for commutative algebras over the
  residue field}
\label{sec:homology_residue}
\numberwithin{theorem}{section}

We are going to associate some natural chain complexes to
commutative algebras over the residue field \(\resf\defeq\dvr/\dvgen
\dvr\).  For two of them, we show in Section~\ref{sec:compare_rigid}
that they compute rigid cohomology.  From now on, we will always
  assume that~$\dvf$ has characteristic 0.

First we fix some \(\alpha\in[0,1]\).
Let~\(A\)
be a commutative \(\resf\)\nb-algebra.
We may view~\(A\)
as a \(\dvr\)\nb-algebra
with \(\dvgen\cdot A=0\).
Let \(S\subseteq A\)
be a set of generators for~\(A\).
Let~$\dvr[S]$ be the free commutative algebra on the set~$S$.
The inclusion map \(S\to A\)
defines a unital homomorphism \(p\colon \dvr[S]\to A\)
because~\(A\)
is commutative.  This is surjective because~\(S\)
generates~\(A\)
by assumption.  Equip~\(\dvr[S]\)
with the fine bornology.  Let \(J\defeq \ker p\);
this contains~\(\dvgen\)
because \(\dvgen\cdot A=0\).
We shall be interested in the completion
\(\comg{\dvr[S]}{J}{\alpha}\),
which is a complete bornological \(\dvr\)\nb-algebra,
and in the chain complex \((CC(\comg{\dvr[S]}{J}{\alpha}),B+b)\)
that computes its periodic cyclic homology.  To obtain a natural
chain complex, we prefer a natural choice for the set of generators,
namely, \(S\defeq A\). This may be infinite even if~\(A\)
is finitely generated, but this shall not bother us: it is the price
to pay for a natural construction.  Let
\[
\chain^\alpha(A) \defeq (CC(\comg{\dvr[A]}{J}{\alpha}),B+b)
\]
be the chain complex constructed above for the generating set
\(S=A\).
We show that this chain complex is natural.  A unital algebra
homomorphism \(f\colon A_1\to A_2\)
induces a bounded unital algebra homomorphism
\(\dvr[f]\colon \dvr[A_1] \to \dvr[A_2]\).
It is compatible with the two projections
\(p_i\colon \dvr[A_i]\to A_i\)
in the following sense: \(p_2\circ \dvr[f]= f\circ p_1\).
Hence~\(\dvr[f]\)
maps the kernel \(J_1\defeq \ker p_1\)
into the kernel \(J_2\defeq \ker p_2\).
By assumption, it extends uniquely to a bounded unital algebra
homomorphism
\(\comg{\dvr[A_1]}{J_1}{\alpha} \to
\comg{\dvr[A_2]}{J_2}{\alpha}\),
which in turn induces a bounded chain map
\[
\chain^\alpha(f)\colon
\chain^\alpha(A_1)\to\chain^\alpha(A_2)
\]
Thus we have constructed a functor~\(\chain^\alpha\)
from the category of commutative \(\resf\)\nb-algebras
to the category of chain complexes of bornological \(\dvf\)\nb-vector
spaces.  We may forget the bornology on~\(\chain^\alpha\)
because we shall not use it.  Let \(\homo^\alpha(A)\) be the
homology of the chain complex~\(\chain^\alpha(A)\).  This is also
functorial for unital algebra homomorphisms.

The identity map on~\(\dvr[A]\)
extends uniquely to a bounded algebra homomorphism
\(\comg{\dvr[A]}{J}{\alpha} \to
\comg{\dvr[A]}{J}{\beta}\)
for \(\alpha,\beta\in[0,1]\)
with \(\alpha\le \beta\).  This induces a chain map
\[
\sigma_{\alpha\beta}\colon \chain^\alpha(A)\to\chain^\beta(A),
\]
which in turn induces a \(\dvf\)\nb-linear map
\[
\sigma_{\alpha\beta}\colon \homo^\alpha(A)\to\homo^\beta(A).
\]

Exactly the same considerations apply to the $\alpha$\nb-versions of
the de Rham complexes associated to~$A$.  We set
\[
\chaindR^\alpha(A) \defeq
(\comg{\dvr[A]}{J}{\alpha} \otimes_{\underline{\dvr[A]}} \Omega^*_{\underline{\dvr[A]}},d)
\]
and write $\homdR^\alpha (A)$ for the homology of this
\(\N\)\nb-graded complex.  Again, for \(\alpha\le \beta\), we
obtain a natural chain map
\[
\sigma_{\alpha\beta}\colon \chaindR^\alpha(A)\to\chaindR^\beta(A).
\]

We shall be most interested in the \emph{homotopy limits} of the
projective systems of complexes that we obtain that way.  To be
specific let us start again with the projective system of cyclic
complexes defined by the maps
\[
\sigma_{\nicefrac1{(m+1)},\nicefrac1m}\colon \chain^{\nicefrac1{(m+1)}}(A)\to \chain^{\nicefrac1m}(A).
\]
The homotopy limit is defined as the mapping cone of the bounded
chain map
\begin{equation}
  \label{eq:hoprojlim_definition}
  \prod_{m=1}^\infty \chain^{\nicefrac1m}(A) \to
  \prod_{m=1}^\infty \chain^{\nicefrac1m}(A),\qquad
  (x_m) \mapsto (x_m - \sigma_{\nicefrac1{(m+1)},\nicefrac1m}(x_{m+1})),
\end{equation}
shifted by~$1$.

\begin{definition}
  \label{def:ccrig}
  We denote the chain complexes obtained as homotopy limits of the
  projective systems
  \[
  \left(\chain^{\nicefrac1m}(A)\right),\qquad
  \left(\chaindR^{\nicefrac1m}(A)\right)
  \]
  by \(\chain^\rig(A)\) and \(\chaindR^\rig(A)\), respectively.
  Their homology is denoted by \(\homo^\rig(A)\) and
  \(\homdR^\rig(A)\).
\end{definition}

The long exact sequence for the homology of a mapping cone implies a
short exact sequence
\[
0 \to \mathop{\lim\nolimits^1}_{m}
\homo^{\nicefrac1m}_{*+1}(A) \to
\homo^\rig_*(A) \to
\lim_{m} \homo^{\nicefrac1m}_*(A) \to 0.
\]
Clearly, \(\chain^\rig\), \(\homo^\rig\) and
\(\chaindR^\rig\), \(\homdR^\rig\) inherit functoriality for unital
\(\resf\)\nb-algebra homomorphisms from the corresponding
$\alpha$\nb-versions.

\begin{remark}
  \label{rem:rig_through_pro-algebra}
  By definition, \(\homo^\alpha_*(A) =
  \HP_*(\comg{\dvr[A]}{J}{\alpha})\) is the periodic cyclic homology
  of a certain complete bornological \(\dvf\)\nb-algebra.  Thus
  \(\homo^\rig_*(A)\) is exactly the periodic cyclic homology, as
  defined in~\cite{corval}, of the projective system
  of complete bornological \(\dvf\)\nb-algebras
  \[
  \left(\comg{\dvr[A]}{J}{\nicefrac1m}\right),\qquad
  m\in\N_{\ge1}.
  \]
\end{remark}

The free presentation~\(\dvr[A]\) of~\(A\) used to
define~\(\chain^\alpha(A)\) is natural but possibly very large.  For
computations, we want to use smaller generating sets.  These give
homotopy equivalent chain complexes:

\begin{proposition}
  \label{pro:free_presentation_independent}
  Let~\(A\) be a \(\resf\)\nb-algebra and let \(S\subseteq A\) be a
  generating set.  Let \(\alpha\in[0,1]\) and let \(J^S\triqui
  \dvr[S]\) be the kernel of the canonical homomorphism \(p^S\colon
  \dvr[S]\to A\).

  There are bounded maps \(f\colon \comg{\dvr[S]}{J^S}{\alpha} \to
  \comg{\dvr[A]}{J}{\alpha}\) and \(g\colon
  \comg{\dvr[A]}{J}{\alpha} \to \comg{\dvr[S]}{J^S}{\alpha}\), such
  that $f\circ g$ and $g\circ f$ are homotopic to the identity
  through a dagger-continuous homotopy.
\end{proposition}

\begin{proof}
  The inclusion map \(S\to A\) induces a unital homomorphism
  \(f\colon \dvr[S]\to\dvr[A]\).  It maps the kernel~\(J^S\) of
  \(p^S\colon \dvr[S]\to A\) into the kernel~\(J\) of \(p\colon
  \dvr[A]\to A\) because \(p\circ f=p^S\).  Hence it extends
  uniquely to a bounded unital algebra homomorphism \(f\colon
  \comg{\dvr[S]}{J^S}{\alpha} \to
  \comg{\dvr[A]}{J}{\alpha}\) by
  Proposition~\ref{pro:lambda_functorial}.  Since~\(S\)
  generates~\(A\), the homomorphism \(p^S\colon \dvr[S]\to A\) is
  surjective.  For each \(a\in A\), choose some \(g(a)\in \dvr[S]\)
  with \(p^S(g(a))=a\); we may assume \(g(a)=a\) for all \(a\in S\).
  These choices define a unital homomorphism \(g\colon
  \dvr[A]\to\dvr[S]\) because~\(\dvr[S]\) is commutative.  By
  construction, \(p^S\circ g(a)=p(a)\) for all \(a\in A\).  This
  implies \(p^S\circ g=p\) and hence \(g(J) \subseteq J^S\).
  Hence~\(g\) extends uniquely to a bounded unital algebra
  homomorphism \(g\colon \comg{\dvr[A]}{J}{\alpha}
  \to \comg{\dvr[S]}{J^S}{\alpha}\) by
  Proposition~\ref{pro:lambda_functorial}.  We have \(g\circ f =
  \id_{\dvr[S]}\) because \(g(a)=a\) for all \(a\in S\).

  The homomorphism \(f\circ g\colon \dvr[A]\to\dvr[A]\) is homotopic
  to the identity map through the homotopy \(H\colon
  \dvr[A]\to\dvr[A,t]\) defined by \(H(a) \defeq t\cdot a+ (1-t)
  fg(a)\) for all \(a\in A\).  Since \(p\circ fg = p\), \(H\) maps
  \(J\defeq \ker(p)\triqui \dvr[A]\) into \(J\otimes \dvr[t]\), the
  kernel of \(p\otimes \id_{\dvr[t]} \colon \dvr[A,t] =
  \dvr[A]\otimes\dvr[t] \to A\otimes \dvr[t]\).  Now
  Corollary~\ref{cor:lambda_tensor_dagger} and
  Proposition~\ref{pro:lambda_functorial} show that~\(H\) extends
  uniquely to a bounded unital algebra homomorphism
  \(\comg{\dvr[A]}{J}{\alpha} \to
  \comg{\dvr[A]}{J}{\alpha} \hotimes
  \comling{\dvr[t]}\).  Here we may identify
  \(\comling{\dvr[t]} \cong \dvr[t]^\updagger\) by
  Theorem~\ref{the:MW_completion}.  Thus~\(H\) is a
  dagger-continuous homotopy between \(f\circ g\) and the identity
  map.
\end{proof}

\begin{corollary}
  \label{cor:anypres}
  Let $A$ and~$S$ be as in
  Proposition~\textup{\ref{pro:free_presentation_independent}}.
  Then \(\chain^\alpha(A)\) is naturally chain homotopy equivalent
  to \(CC(\comg{\dvr[S]}{J^S}{\alpha},B+b)\).  And
  \(\chain^\rig(A)\) is naturally chain homotopy equivalent to the
  homotopy limit of the chain complexes
  \(CC(\comg{\dvr[S]}{J^S}{\alpha},B+b)\) for \(\alpha=1/m\),
  \(m\to\infty\).
  Analogous statements hold for the de Rham complexes
  \(\chaindR^\alpha(A)\) and \(\chaindR^\rig(A)\).
\end{corollary}

\begin{proof}
  The first assertion follows from
  Proposition~\ref{pro:free_presentation_independent} and
  Proposition~\ref{pro:dagger-continuous_homotopy_HP}.(b).  The maps
  \(CC(f)\) and \(CC(g)\) that we get from
    Proposition~\ref{pro:free_presentation_independent} and the
  chain homotopy equivalence are compatible for different~\(\alpha\)
  and hence induce a chain homotopy equivalence between the homotopy
  projective limits, giving the second statement.  The proof for
  the de Rham version is the same.
\end{proof}

Roughly speaking, we get chain complexes homotopy equivalent
to \(\chain^\alpha(A)\) and~\(\chain^\rig(A)\) when we replace
\(p\colon \dvr[A] \to A\) by any presentation of~\(A\) by a free
commutative algebra over~\(\dvr\): a surjective homomorphism
\(\dvr[S] \to A\) is equivalent to a map \(S\to A\) whose image
generates~\(A\).
Proposition~\ref{pro:free_presentation_independent} remains true if
the map \(S\to A\) is not injective.

\begin{theorem}
  \label{the:quasi}
  Let~$A$ be a finitely generated $\resf$\nb-algebra.  The complexes
  \(\chain^\rig(A)\) and~\(\chaindR^\rig(A)\) made
    \(2\)\nb-periodic are quasi-isomorphic.
\end{theorem}

\begin{proof}
  By Corollary~\ref{cor:anypres}, the complexes
  \(\chain^\rig(A)\) and~\(\chaindR^\rig(A)\) are quasi-isomorphic
  to the homotopy limits of the complexes
  \[
  CC(\comg{\dvr[S]}{J^S}{1/m},B+b)\quad\text{and}\quad
  (\comg{\dvr[S]}{J^S}{\nicefrac1m} \otimes \Omega^*_{\dvr[S]},d),
  \]
  respectively, for a set~$S$ of generators.  We may assume that~$S$
  is finite.  Then, by~\eqref{eq:HP_Rdagger}
  \(CC(\comg{\dvr[S]}{J^S}{1/m},B+b)\) is quasi-isomorphic to the
  de~Rham complex \((\comg{\dvr[S]}{J^S}{\nicefrac1m} \otimes
  \Omega^*_{\dvr[S]},d)\) made \(2\)\nb-periodic.  Homotopy
  invariance (Proposition
  \ref{pro:dagger-continuous_homotopy_HP}.(a)) shows that the
  de~Rham complex for $\comg{\dvr[S]}{J^S}{1/m}$ is homotopy
  equivalent to the one for $\comg{\dvr[A]}{J}{1/m}$.  The homotopy
  limit respects quasi-isomorphisms.  This implies the assertion.
\end{proof}

Our construction is invariant under polynomial homotopies of the
underlying $\resf$\nb-algebras:

\begin{proposition}
  \label{pro:homo_homotopy-invariant}
  Let \(A_1,A_2\) be finitely generated commutative
  \(\resf\)\nb-algebras and \(\alpha\in[0,1]\) rational.
  Let \(f_0,f_1\colon A_1\rightrightarrows A_2\) be unital
  homomorphisms that are polynomially homotopic, that is, there is
  a unital homomorphism \(F\colon A_1 \to A_2[t]\) with \(\ev_t\circ
  F = f_t\) for \(t=0,1\).  Then the induced bounded chain maps
  \(\chain^\alpha(f_0)\) and~\(\chain^\alpha(f_1)\) are chain homotopic.
  So are \(\chain^\rig(f_0)\) and~\(\chain^\rig(f_1)\).  Thus
  \(\homo^\alpha(f_0)= \homo^\alpha(f_1)\) and \(\homo^\rig(f_0)=
  \homo^\rig(f_1)\), and similarly for~\(\homdR\).
\end{proposition}

\begin{proof}
  For \(i=1,2\), let \(p_i\colon \dvr[A_i]\to A_i\) be the canonical
  homomorphisms and let \(J_i \defeq \ker(p_i) \triqui A_i\).  The
  homomorphism \(F\colon A_1\to A_2[t]\) induces a unital
  \(\dvr\)\nb-algebra homomorphism \(F_*\colon \dvr[A_1] \to
  \dvr[A_2][t] = \dvr[A_2]\otimes \dvr[t]\), which maps the
  generator \(a\in A_1\) to \(\sum b_n\otimes t^n\) with \(b_n\in
  A_2\subseteq \dvr[A_2]\) if \(F(a) = \sum b_n t^n \in A_2[t]\).
  We have \((p_2\otimes \id_{\dvr[t]})\circ F_* = F\circ p_1\)
  because this holds on all generators \(a\in A_1\).  Hence \(F(J_1)
  \subseteq J_2 \otimes \dvr[t]\).  Now
  Corollary~\ref{cor:lambda_tensor_dagger} and
  Proposition~\ref{pro:lambda_functorial} show that~\(F\) extends
  uniquely to a bounded unital algebra homomorphism
  \(\bar{F}_*\colon \comg{\dvr[A_1]}{J}{\alpha} \to
  \comg{\dvr[A_2]}{J}{\alpha} \hotimes \dvr[t]^\updagger\).
  Thus~\(\bar{F}_*\) is a dagger-continuous homotopy between the
  homomorphisms induced by \(f_0\) and~\(f_1\).  Now
  Proposition~\ref{pro:dagger-continuous_homotopy_HP} shows that the
  induced maps \(\chain^\alpha(f_0)\) and \(\chain^\alpha(f_1)\) are
  chain homotopic, and similarly for \(\chaindR^\alpha\).
  Since this holds for all~\(\alpha\) the induced maps
  \(\chain^\rig(f_0)\) and \(\chain^\rig(f_1)\) are chain homotopic
  as well, and similarly for~\(\chaindR^\rig\).  Then the
  induced maps on homology are equal.
\end{proof}

\section{Comparison to rigid cohomology}
\label{sec:compare_rigid}

Let~\(A\) be a finitely generated, commutative \(\resf\)\nb-algebra.
We are going to identify \(\homdR^\rig(A)\) with Berthelot's
rigid cohomology of~\(A\), as our notation already suggests.
Berthelot defines rigid cohomology of~\(A\), or more generally of
separated \(\resf\)\nb-schemes of finite type,
in~\cite{Berthelot-finitude}.  His construction is quite involved.
Even in the affine case, it explicitly uses non-affine schemes.
Gro\ss{}e-Kl\"onne's theory of dagger spaces (or rigid analytic
spaces with overconvergent structure sheaf) \cite{gkdagger}
simplifies Berthelot's construction a little bit.  In the following,
we describe this simplified construction.  We begin by recalling the
relevant things from dagger geometry.


For $n\in\N$ the \emph{Washnitzer algebra} over~$\dvf$ is
\[
W_{n} \defeq \dvr[x_{1}, \ldots, x_{n}]^{\updagger} \otimes_{\dvr} \dvf.
\]
We equip $\dvr[x_{1}, \ldots, x_{n}]^{\updagger}$ with the bornology
as at the beginning of Section~\ref{subsec:dag} and~$W_{n}$
with the induced bornology.  A \emph{dagger} $\dvf$\nb-algebra is a
quotient of some~$W_{n}$ by an ideal.

Let~$L$ be a dagger $\dvf$\nb-algebra.  We denote the set of maximal
ideals of~$L$ by~$\Sp(L)$.  We denote an element in~$\Sp(L)$ either
by $x$ or by~$\fm_{x}$ depending on whether we think of it as a
point~$x$ in the set~$\Sp(L)$ or as a maximal ideal $\fm_{x} \triqui
L$.

Consider an element~$f$ of the dagger algebra~$L$ and $x \in
\Sp(L)$.  The \emph{residue field} $\dvf(x) \defeq L / \fm_{x}$ is a
finite field extension of $\dvf$.  Hence the absolute value
$\abs{\hphantom{x}}$ on~\(\dvf\) extends uniquely to an absolute
value $\abs{\hphantom{x}}\colon \dvf(x) \to \R_{\geq 0}$.  The image
of~$f$ under the composition $L \to \dvf(x) \to \R_{\geq 0}$ will be
denoted by~$\abs{f(x)}$.

\comment{Replaced~\(V\) by~\(T\) for special subsets to avoid
  notational conflict.}
A subset $T \subseteq X = \Sp(L)$ is called \emph{special} if there
are elements $f_{0}, \ldots, f_{r} \in L$ which generate the unit
ideal and such that
\begin{equation}
  \label{eq:special-subset}
  T = \{ x \in X : \abs{f_{i}(x)} \leq \abs{f_{0}(x)} \text{ for all } i = 1, \ldots, r\}.
\end{equation}
Special subsets are stable under finite intersections.  A subset $U
\subseteq X$ is called \emph{admissible open} if there is a covering
$(T_{i})_{i}$ of~$U$ by special subsets satisfying the following
condition.  For every morphism of dagger algebras $L_1 \to L_2$ such
that the induced map $\Sp(L_2) \to \Sp(L_1)$ factors through~$U$,
the induced covering of~$\Sp(L_2)$ admits a finite subcovering.  A
covering of the admissible open subset~$U$ by admissible open
subsets~$U_{i}$ is called \emph{admissible}, if for every $L_1 \to
L_2$ as above, the covering of $\Sp(L_2)$ induced by the covering
$(U_{i})_{i}$ admits a finite refinement by special subsets.  The
admissible open subsets with the admissible coverings form a
Grothendieck pretopology.  By a \emph{sheaf} on~$X$ we mean a sheaf
for this pretopology.

To a special subset~$T$ as in~\eqref{eq:special-subset} one
associates the dagger algebra
\begin{equation}
  \label{eq:associated-dagger-algebra}
  \Gamma(T, \cO_{X}) \defeq L \langle x_{1}, \ldots, x_{r}
  \rangle^{\updagger} / \langle f_{i} - f_{0} x_{i},\ i = 1, \ldots, r \rangle.
\end{equation}
Here $L\langle x \rangle^{\updagger} \defeq L \hotimes
\dvr[x]^{\updagger}$.  The natural map $L \to \Gamma(T, \cO_{X})$
induces a bijection $\Sp(\Gamma(T, \cO_{X})) \xrightarrow{\cong} T
\subseteq X$.  Tate's Acyclicity Theorem implies that the
construction $T \mapsto \Gamma(T, \cO_{X})$ extends uniquely to a
sheaf of rings~$\cO_{X}$ on~$X$, called the \emph{structure sheaf}.
The pair $(X,\cO_{X})$ is called an \emph{affinoid dagger space}.
One may glue affinoid dagger spaces along admissible open
subsets to get the notion of a dagger space.  For example, every
admissible open subset $U \subseteq X$ with structure sheaf given by
the restriction of~$\cO_{X}$ is a dagger space, which is in general
not affinoid.

Let~$L$ be a dagger $\dvf$\nb-algebra.  There is a
universal $\dvf$\nb-derivation from~$L$ to a finite $L$\nb-module,
denoted by $d\colon L \to \Omega^{\updagger,1}_{L/\dvf}$.  This
construction extends uniquely to a coherent sheaf
$\Omega^{\updagger,1}_{X/\dvf}$ on any dagger space~$X$ over~$\dvf$.
In the usual way, we obtain from this a de Rham complex
$\Omega^{\updagger}_{X/\dvf}$ of coherent sheaves on~$X$.
Explicitly, if $L= W_{n} / \langle f_{1}, \ldots, f_{r} \rangle $,
then
\[
\Omega^{\updagger,1}_{L/\dvf} \cong (L \otimes_{\dvr[x_{1}, \ldots, x_{n}]} \Omega^{1}_{\dvr[x_{1}, \ldots, x_{n}]/\dvr} ) / \langle df_{1}, \ldots, df_{r} \rangle.
\]
\begin{remark}
    If~$R$ is a finitely generated commutative
    $\dvr$\nb-algebra, then~$\ul{R^\updagger}$ is a dagger
    $\dvf$\nb-algebra and the de Rham complex
    $\Omega^{\updagger,\star}_{\ul{R^\updagger}/\dvf}$, just
    introduced, is the one of Definition~\ref{def:dR}.
\end{remark}

Rigid cohomology is defined as the de Rham cohomology of certain
dagger spaces that arise as tubes, which are defined by means of
specialisation maps.  We begin by recalling the latter.  These can
be defined in a more general setting, but to avoid technicalities,
we restrict to the situation we need.  Let~$R$ be a finitely
generated commutative $\dvr$\nb-algebra.  Then $\ul{R^{\updagger}}
\defeq R^{\updagger} \otimes_{\dvr} \dvf$ is a dagger
$\dvf$\nb-algebra, and we write $X = \Sp(\ul{R^{\updagger}})$.  The
\emph{specialisation map} is a map of sets
\[
\spec\colon X \to \Spec(R / \dvgen R),
\]
which is continuous for the Grothendieck pretopology on~$X$ and the
Zariski topology on $\Spec(R / \dvgen R)$.  It is constructed as
follows.  The residue field $\dvf(x)$ of a point $x\in X$ is a
finite extension of~$\dvf$.  We denote its valuation ring by
$\dvr(x) \subseteq \dvf(x)$ and its residue field by $\resf(x)$.
Since the elements of~$R^{\updagger}$ are power bounded, their
images in $\dvf(x)$ lie in $\dvr(x)$.  Now the kernel of the
composite map \(R \to \dvr(x) \to \resf(x)\) is a prime ideal of~$R$
containing $\dvgen R$, hence corresponds to a unique prime ideal of
$R / \dvgen R$.  This is the desired point $\spec(x) \in \Spec(R /
\dvgen R)$.

We are now ready to discuss tubes.  In the situation of the
preceding paragraph, let~$Z$ be a closed subset of $\Spec(R / \dvgen
R)$.  The \emph{tube of~$Z$} in $X = \Sp(\ul{R^{\updagger}})$ is the
subset
\begin{equation}
  \label{eq:tube}
  ]Z[ \defeq \spec^{-1}(Z) \subseteq X.
\end{equation}
It is an admissible open subset of~$X$ (see
\cite{berth}*{Prop.~1.1.1}).  Hence it inherits the structure of a
dagger space.

\begin{remark}
  \label{rem:explicit-epsilon-tubes}
  We can describe tubes more explicitly.  Let $f_{1}, \ldots, f_{r}
  \in R$ be elements whose images in $R / \dvgen R$ define the
  closed subset $Z \subseteq \Spec(R / \dvgen R)$.  We may consider
  the~$f_{i}$ as elements of~$\ul{R^{\updagger}}$.  From the
  construction of the specialisation map, we see that a point $x \in
  X$ belongs to the tube~$]Z[$ if and only if $\abs{f_{i}(x)} < 1$
  for all $i = 1, \ldots, r$, that is,
  \[
  ]Z[ =  \{ x\in X :  \abs{f_{i}(x)} < 1 \text{ for } i  = 1, \ldots, r\}.
  \]
  Hence~$]Z[$ is the increasing union $]Z[ = \bigcup_{n \geq 1}
  [Z]_{\pepsilon^{1/n}}$, where $[Z]_{\pepsilon^{1/n}}$ denotes the
  special subset
  \begin{align*}
    [Z]_{\pepsilon^{1/n}} &\defeq \{  x\in X : \abs{f_{i}(x)} \leq \pepsilon^{1/n} \text{ for } i  = 1, \ldots, r\}  \\
    &= \{  x\in X : \abs{f_{i}^{n}(x)} \leq \pepsilon \text{ for } i  = 1, \ldots, r\}.
  \end{align*}
  This gives an admissible covering of~$]Z[$ (see
  \cite{berth}*{Prop.~1.1.9}).  Using
  \eqref{eq:associated-dagger-algebra} and the remarks there, we get
  \[
  [Z]_{\pepsilon^{1/n}} \cong \Sp( \ul{R[x_{1}, \ldots, x_{r}]^{\updagger}} / \langle f_{i}^{n}  - \dvgen x_{i}, i = 1, \ldots, r \rangle ).
  \]
\end{remark}

Finally, we can define rigid cohomology.  Let~$A$ be a
finitely generated commutative \(\resf\)\nb-algebra.  Choose a
smooth commutative \(\dvr\)\nb-algebra~$R$ together with a
surjection $p\colon R \to A$.  As before, we let~$X$ be
$\Sp(\ul{R^{\updagger}})$.  Via~$p$ we identify $\Spec(A)$ with a
closed subscheme~$Z$ of $\Spec(R/\dvgen R)$ and consider its
tube~$]Z[$ in~$X$ (by construction, $]Z[$ only depends on the
underlying set of the scheme~$Z$).

\begin{definition}
  \label{def:rigid-cohomology}
  The \emph{rigid cohomology} of~$A$ with coefficients in~$\dvf$ is
  the de Rham cohomology of the tube~$]Z[$,
  \[
  H_\rig^*(A,\dvf) \defeq \Hy^*(]Z[,\Omega^{\updagger}_{]Z[/\dvf}).
  \]
\end{definition}

\begin{remark}
  Berthelot's original definition is
  \cite{Berthelot-finitude}*{(1.3.1)}.  It depends on the additional
  choice of a compactification $\overline Z$ of $Z$.  This
  compactification is used to define a functor~$j^{\updagger}$ which
  associates to a sheaf on the tube~$]\overline Z[$ the sheaf of its
  overconvergent sections.  This is necessary because Berthelot
  works with rigid spaces.  In contrast, in Gro{\ss}e-Kl\"onne's
  theory of dagger spaces, the overconvergence is already built in.
  This is why the compactification or the functor~$j^{\updagger}$ do
  not appear in the formula in
  Definition~\ref{def:rigid-cohomology}.  Berthelot proves the
  independence of choices up to isomorphism in
  \cite{Berthelot-finitude}*{\S1}.

  By \cite{gkdr}*{Proposition~3.6}, the rigid cohomology groups
  defined in Definition~\ref{def:rigid-cohomology} are canonically
  isomorphic to the groups defined by Berthelot.  The isomorphism is
  functorial in the following sense.  Suppose that $f\colon A_{1}
  \to A_{2}$ is a morphism of finitely generated commutative
  $\resf$\nb-algebras, $R_{1}, R_{2}$ are smooth commutative
  $\dvr$\nb-algebras with surjections $p_{i}\colon R_{i} \to A_{i}$,
  $i = 1,2$, and there is a homomorphism $F\colon R_{1} \to
  R_{2}$ such that the diagram
  \[
  \xymatrix{
    R_{1} \ar[d]^{p_{1}} \ar[r]^{F}   & R_{2} \ar[d]^{p_{2}}  \\
    A_{1} \ar[r]^{f}   & A_{2}
  }
  \]
  commutes.  Then~$F$ induces a morphism between the tubes
  $]\Spec(A_{2})[ \to ]\Spec(A_{1})[$, which in turn induces a
  homomorphism $H_\rig^*(A_{1},\dvf) \to H_\rig^*(A_{2},\dvf)$.
  This homomorphism coincides with the one constructed by Berthelot.
\end{remark}

The main result of this section is the following.

\begin{theorem}
  \label{thm:rig}
  Let~\(A\) be a finitely generated, commutative
  \(\resf\)\nb-algebra.  There are natural isomorphisms
  \(\homdR_j^\rig(A) \cong H_\rig^{j}(A,\dvf)\).  Hence
    \[
    \homo_j^\rig(A) \cong \bigoplus_{n\ge0} H_\rig^{2n+j}(A,\dvf).
    \]
\end{theorem}

\begin{proof}
  By Theorem~\ref{the:quasi}, the complexes \(\chain^\rig(A)\)
    and~\(\chaindR^\rig(A)\) made \(2\)\nb-periodic are
    quasi-isomorphic.  So the statement about \(\homdR_j^\rig(A)\)
    implies the one about \(\homo_j^\rig(A)\).

  Let \(S \subseteq A\) be a finite set generating~$A$ as
  $\resf$\nb-algebra.  Let \(R\defeq \dvr[S]\), let \(p\colon R\to
  A\) be the homomorphism induced by the inclusion of~\(S\)
  into~\(A\), and let \(J\triqui R\) be the kernel of~\(p\) (the
  ideal~$J$ was denoted by~$J^S$ in
  Section~\ref{sec:homology_residue}).  Since~\(R\) is smooth
  over~\(\dvr\), we may use it to compute rigid cohomology.
  The proof of Theorem~\ref{the:quasi} shows that
    \(\chaindR^\rig(A)\) is quasi-isomorphic to the homotopy limit
    of the complexes \((\comg{\dvr[S]}{J}{\nicefrac1m} \otimes
    \Omega^*_{\dvr[S]},d)\) for \(m\in\N\).

  As before, we write $X = \Sp(\ul{R^{\updagger}})$ and $Z =
  \Spec(A) \subseteq \Spec(R/\dvgen R)$.  Before continuing with the proof of \ref{thm:rig} we relate the
  tube~\eqref{eq:tube} to the bornological
  completions~\(\comg{R}{J}{\alpha}\) using the tube algebras of
  Definition~\ref{def:alpha-tube}.

\begin{lemma}
  \label{lem:covering-of-tube}
  For $m \in \N_{\geq 1}$ the algebra $\comg{R}{J}{\nicefrac1m}$ is
  a dagger $\dvf$\nb-algebra.  The affinoid dagger space
  $\Sp(\comg{R}{J}{\nicefrac1m})$ is an admissible open subset
  of~$X$ contained in~$]Z[$, and
  \[
  ]Z[ = \bigcup_{m \geq 1} \Sp( \comg{R}{J}{\nicefrac1m} )
  \]
  is an admissible covering.
\end{lemma}

\begin{proof}
  Choose generators $g_{1}, \ldots, g_{s}$ of the ideal $J^{m} \triqui R$
  and set
  \[
  S_{m} \defeq R[ y_{1}, \ldots, y_{s} ] / \langle g_{i} - \dvgen y_{i}, i = 1, \ldots, s \rangle.
  \]
  Then $\dvf \otimes S_{m}^{\updagger}$ is a dagger
  $\dvf$\nb-algebra, and by~\eqref{eq:associated-dagger-algebra} and
  the remarks there, the affinoid dagger space $\Sp( \dvf \otimes
  S_{m}^{\updagger} )$ is the special open subset
  \[
  [Z]'_{\pepsilon^{1/m}} \defeq \{ x \in X : \abs{g_{i}(x)} \leq \pepsilon \text{ for all } i = 1, \ldots, s\}
  \]
  of~$X$.  Let $f_{1}, \ldots, f_{r}$ be generators of the
  ideal~$J$.  We use the notation introduced in
  Remark~\ref{rem:explicit-epsilon-tubes}.  From the inclusions of
  ideals \(J^{mr} \subseteq \langle f_{1}^{m} , \ldots, f_{r}^{m}
  \rangle \subseteq J^{m}\) we deduce that \([Z]'_{\pepsilon^{1/m}}
  \subseteq [Z]_{\pepsilon^{1/m}} \subseteq
  [Z]'_{\pepsilon^{1/mr}}\).  Since the $[Z]_{\pepsilon^{1/n}}$, $n
  \in \N_{\geq 1}$, form an admissible covering of~$]Z[$ by
  Remark~\ref{rem:explicit-epsilon-tubes}, it follows formally that
  also the $[Z]'_{\pepsilon^{1/m}}$, $m\in \N_{\geq 1}$, form an
  admissible covering of the tube $]Z[ \subseteq X$.

  Hence it is enough to show that $\dvf \otimes S_{m}^{\updagger}
  \cong \comg{R}{J}{\nicefrac1m}$.  By Theorem~\ref{the:MW_tube} and
  Example~\ref{exa:alpha_tube} we have isomorphisms
  \[
  \comg{R}{J}{\nicefrac1m} \cong \dvf \otimes \tub{R}{J}{\nicefrac1m}^\updagger = \dvf \otimes \tub{R}{J^{m}}{1}^{\updagger}.
  \]
  Moreover, we have a surjective homomorphism $S_{m} \to
  \tub{R}{J^{m}}{1} \subseteq \ul{R}$ induced by $y_{i} \mapsto
  \dvgen^{-1}g_{i}$, $i=1,\ldots, s$.  It induces an isomorphism
  after tensoring with~$\dvf$.  Hence $ \tub{R}{J^{m}}{1} \cong
  S_{m}/I$, where $I \triqui S_{m}$ is the ideal of
  $\dvr$\nb-torsion elements.  Lemma~\ref{lem:onto} implies
  $S_{m}^{\updagger}/IS_{m}^{\updagger} \cong
  \tub{R}{J^{m}}{1}^{\updagger}$ and hence $\dvf \otimes
  S_{m}^{\updagger} \cong \dvf \otimes
  \tub{R}{J^{m}}{1}^{\updagger}$, as desired.
\end{proof}

  Recall that the rigid cohomology of~$A$ is defined as the
  cohomology of the de Rham complex $\Omega^{\updagger}_{]Z[/\dvf}$ on
  the tube~$]Z[$.  We use the admissible covering of~$]Z[$ from
  Lemma~\ref{lem:covering-of-tube} to compute this cohomology.  To
  simplify notation, we set $U_{m} = \Sp( \comg{R}{J}{\nicefrac1m}
  )\subseteq ]Z[$.  Explicitly, we have $\Gamma( U_{m},
  \Omega^{\updagger,l}_{]Z[/\dvf} ) \cong \comg{R}{J}{\nicefrac1m}
  \otimes_{\ul{R}} \Omega_{\ul{R}}^l $.

  Let \( \Omega^{\updagger}_{]Z[/\dvf} \to \cI\) be a Cartan--Eilenberg
  injective resolution in \(\cO_{]Z[}\textup{-Mod}\).  Then
  \(H^*(]Z[,\Omega^{\updagger}_{]Z[/\dvf})\) is the cohomology of the
  complex of global sections \(\Gamma(]Z[,\cI)\).  Let
  \(\Vect_\dvf^{\N^\op}\) be the category of contravariant functors
  \(\N\to\Vect_\dvf\) with natural transformations as morphisms.
  The section functor \(\Gamma(]Z[,{-})\colon \cO_{]Z[}\textup{-Mod}
  \to \Vect(\dvf)\) from sheaves to vector spaces factors as the
  composite of the functor
  \[
  \Gamma(U_\bullet,{-})\colon
  \cO_{]Z[}\textup{-Mod} \to \Vect(\dvf)^{\N^\op},\qquad
  \cS\mapsto \{\Gamma(U_m,\cS_{|U_m})\}_m,
  \]
  followed by
  \[
  \varprojlim\colon \Vect(\dvf)^{\N^\op}\to\Vect(\dvf).
  \]
  Since the~$U_{m}$ are affinoid and the higher cohomology of
  coherent sheaves on affinoid dagger spaces vanishes
  (see~\cite{gkdagger}*{Proposition~3.1}), the map
  \(\Gamma(U_\bullet,\Omega^{\updagger}_{]Z[/\dvf})\to
  \Gamma(U_\bullet, \cI)\) is a levelwise quasi-isomorphism.

  Write~$\holim$ for the homotopy limit construction of a projective
  system of complexes as explained
  around~\eqref{eq:hoprojlim_definition}.  As already mentioned, it
  respects levelwise quasi-isomorphisms.  Collecting what we have
  done so far, we get quasi-isomorphisms
  \begin{align*}
    \chaindR^\rig(A) &\simeq \holim_{m} \comg{R}{J}{\nicefrac1m} \otimes_{\ul{R}} (\Omega^{*}_{\ul{R}},d) \\
    & \cong \holim_{m} \Gamma(U_{m}, \Omega^{\updagger}_{]Z[/\dvf} ) \\
    & \simeq \holim_{m} \Gamma(U_{m}, \cI ).
  \end{align*}

  Since each~\(\cI^l\) is injective, the maps
  \(\Gamma(U_{m+1},\cI^l)\to \Gamma(U_m,\cI^l)\) are surjective.
  This implies that $\holim_{m} \Gamma(U_{m}, \cI ) $ is
  quasi-isomorphic to $\varprojlim_{m} \Gamma(U_{m}, \cI ) \cong
  \Gamma( ]Z[, \cI)$; taking cohomology, we get the theorem.
\end{proof}

\section{A quick route to the construction of complexes computing
  rigid cohomology}
\label{sec:properties}

In Sections \ref{sec:bornologies} and~\ref{sec:borno}, we
have developed a general conceptual framework for treating the kind
of generalized weak completions that are relevant to our approach to
rigid cohomology.  It should be noted however that much of the
material in these sections can be avoided if one is interested only
in a natural description of a complex computing rigid cohomology.
Also, only part of the results in Sections \ref{sec:HH_dagger}
  and~\ref{sec:homology_residue} are needed for that purpose.  We
now sketch a more direct route to the main result in
Section~\ref{sec:compare_rigid}.

\begin{definition}
  Let~$R$ be a $\dvr$\nb-algebra and~$J$ an ideal in~$R$ that
  contains~\(\dvgen^k\) for some $k\in \N$. We define the $J$\nb-adic bornology on~$\ul{R}$
  as the bornology generated by subsets~$S$ of the form
  \[
  S = C\sum_{n\geq 0} \lambda_n M^n,
  \]
  where $C\in \dvf$, $M$ is a finitely generated $\dvr$\nb-submodule
  of~$J$ and $\lambda_n \in \dvf$ are such that $\abs{\lambda_n}\leq
  \pepsilon^{-\alpha n}$ for some $\alpha <1$.
\end{definition}

Let $\comJ{R}{J}$ be the bornological completion of~$\ul{R}$
with respect to this bornology.  The easy fact that this completion
coincides with~$\comg{R}{J}{1}$, as defined in
Definition~\ref{def:lambda_completion} (for the fine bornology
on~$R$), is explained in Lemma~\ref{lem:alpha_bornology}.

Also the following is not difficult to see (compare
Proposition~\ref{pro:dagger_for_quotient}).

\bprop If~$R$ is finitely generated and $J=\dvgen R$, then
$\comJ{R}{J}= \underline{R^\dagger}$. \eprop

Now let~$A$ be a finitely generated, commutative
\(\resf\)\nb-algebra.  We choose a presentation $J\to R \to A$ by a
free commutative $\dvr$\nb-algebra~$R$.  We obtain a projective system
\[
\comJ{R}{J^{n+1}}\to \comJ{R}{J^{n}}
\]
This defines a bornological pro-$\dvf$-algebra $\left(\comJ{R}{J^{n}}\right)$.

\bdefin We define
\bgl\label{eq:HP}\HP_*(A) = \HP_*(\left(\comJ{R}{J^{n}}\right))\egl
\edefin

Here $\HP_*$ for a pro-algebra is defined as the homology of the
homotopy limit of the projective system of cyclic $B-b$-complexes as
in Definition~\ref{def:ccrig}.

We briefly explain why this definition of $\HP_*(A)$ gives the same
result as~$\homo^\rig (A)$ in Definition~\ref{def:ccrig}, and
thus also describes $H_*^\rig(A,\dvf)$ if~$A$ is finitely generated.

First, by definition, one trivially has $\comg{R}{J^n}{1} = \comg{R}{J}{\nicefrac1n}$.
The less trivial result proved in
Proposition~\ref{pro:linear_growth_on_tube}, is that
$\comJ{R}{J^n}=\tub{R}{J}{\nicefrac1n}^\updagger\otimes
\dvf$. The proof of Proposition~\ref{pro:linear_growth_on_tube} also
slightly simplifies for the case $\alpha =\nicefrac1n$ that we need here.
This identification of $\comJ{R}{J^n}$ with the completed tube algebra
yields an isomorphism of $\HP_*(\comJ{R}{J})$ with the homology of
the de Rham complex $\left( \comJ{R}{J}
  \otimes_{\ul{R}}\Omega^*_{\ul{R}},\, d\right)$ as in
\eqref{eq:HP_Rdagger}.

Finally, the discussion in Section~\ref{sec:compare_rigid} remains
valid -- we simply have to replace $\comg{R}{J}{\nicefrac1m}$ there
by $\comg{R}{J^m}{1}=\comJ{R}{J^m}$.  We obtain that $\HP_*(A)$, as
defined in~\eqref{eq:HP}, describes the rigid cohomology of~$A$
made \(2\)\nb-periodic.

Of course, in this discussion, one could also directly replace the
homotopy limit of the cyclic complexes by the homotopy limit of the
corresponding de~Rham complexes and thus avoid using many of the
results in Sections \ref{sec:HH_dagger}
and~\ref{sec:homology_residue}.

\end{document}